\documentclass[12pt,reqno]{amsart}
\usepackage{amsfonts,amsmath,amssymb,amsthm,amsrefs}
\usepackage{latexsym}
\usepackage{layout}
\usepackage{color}
\usepackage{verbatim}
\usepackage[tmargin=2.0cm, bmargin=2.0cm, lmargin=3.0cm, rmargin=3.0cm]{geometry}

\newcommand{\red}{\textcolor[rgb]{1.00,0.00,0.00}}
\newcommand{\blue}{\textcolor[rgb]{0.00,0.00,1.00}}

\newtheorem{theorem}{Theorem}
\newtheorem{example}{Example}
\newtheorem{prop}[theorem]{Proposition}
\newtheorem{lemma}[theorem]{Lemma}
\newtheorem{corol}[theorem]{Corollary}

\theoremstyle{definition}
\newtheorem{defi}[theorem]{Definition}
\theoremstyle{remark}
\newtheorem{remark}[theorem]{Remark}

\newcommand{\p}{\mathbb{P}}
\newcommand{\e}{\mathbb{E}}

\newcommand{\reals}{\mathbb{R}}
\newcommand{\ind}{\mathbf{1}}

\newcommand{\md}{\mathrm{d}}
\newcommand{\drift}{c}

\newcommand{\wq}{w^{(q)}}

\def\beq{\begin{eqnarray}} \def\eeq{\end{eqnarray}}

\def\al*#1{\begin{align*}#1\end{align*}}

\def\ga*#1{\begin{gather*}#1\end{gather*}}

\def\alat*#1#2{\begin{alignat*}{#1}#2\end{alignat*}}
\def\bea{\begin{eqnarray*}}
\def\eea{\end{eqnarray*}}
\def\ml*#1{\begin{multline*}#1\end{multline*}}

\newcommand{\W}{\mathbb{W}}

\newcommand{\C}{\mathcal {C}}
\newcommand{\B}{\mathcal {B}}
\newcommand{\A}{\mathcal {A}}

\allowdisplaybreaks

\title{Fluctuation theory for Level-dependent L\'evy risk processes }

\author{Irmina Czarna}
\address{Faculty of Pure and Applied Mathematics, Wroc\l aw University
of Science and Technology, ul. Wybrze\.ze Wyspia\'nskiego 27, 50-370 Wroc\l aw, Poland.}
\email{irmina.czarna@gmail.com}

\author{Jos\'e-Luis P\'erez}
\address{Department of Probability and Statistics, Centro de Investigaci\'on en Matem\'aticas A.C. Calle Jalisco s/n. C.P. 36240,
Guanajuato, Mexico.}
\email{jluis.garmendia@cimat.mx}

\author{Tomasz Rolski
}
\address{Mathematical Institute, University of Wroc\l aw, pl. Grunwaldzki 2/4, 50-384 Wroc\l aw, Poland.}
\email{tomasz.rolski@gmail.com}

\author{Kazutoshi Yamazaki
}
\address{Department of Mathematics, Faculty of Engineering Science, Kansai University, 3-3-35 Yamatecho,
Suita-shi, Osaka 564-8680, Japan.}
\email{kyamazak@kansai-u.ac.jp}

\thanks{I. Czarna is partially supported by the National Science Centre Grant No. 2015/19/D/ST1/01182. T. Rolski is partially supported by the National Science Centre Grant No. 2015/17/B/ST1/01102. J. L. P\'erez  is  supported  by  CONACYT,  
project no.\ 241195. K. Yamazaki is supported by MEXT KAKENHI grant no.\ 17K05377.
 }

\date{\today}

\begin{document}

\begin{abstract}
A level-dependent L\'evy process solves the stochastic differential equation $dU(t)=dX(t)-\phi(U(t))\,dt$, 
where $X$ is a spectrally negative L\'evy process. A special case is a multi-refracted L\'evy process with 
$\phi_k(x)=\sum_{j=1}^k\delta_j\ind_{\{x\ge b_j\}}$. A general rate function $\phi$ that is non-decreasing and 
locally Lipschitz is also considered. We discuss solutions of the above stochastic differential equation and investigate 
the so-called scale functions, which are counterparts of the scale functions from the theory of L\'evy processes. 
We show how fluctuation identities for $U$ can be expressed via these scale functions. We demonstrate that the derivatives of the scale 
functions are solutions of Volterra integral equations.

\vspace{3mm}


\end{abstract}

\maketitle

\newpage
\vspace{5cm}

%\tableofcontents

\newpage

\section{Introduction}\label{sec:intro}
In this paper, we consider a \textit{level-dependent L\'evy process} $U(t)$, which solves
the following stochastic differential equation (SDE):
\begin{equation}\label{sde1}
dU(t) =dX(t)-\phi(U(t))\,dt, 
\end{equation}
where $X(t)$ is a spectrally negative 
L\'evy process.
%\red{[For $d$ we sometimes use a special font for the derivative. Maybe easier to see if we stick to this special font for the derivative as we use $d$ for  a level?]}
%A special case of \eqref{sde1}  was considered by  Kyprianou and Loeffen \cite{kyprianouloeffen2010}, wherein $\phi(x)=\delta\ind_{\{x>b\}}$ 
%for given $\delta>0$. In such a case the solution $U$ is said to be a {\it refracted L\'evy process}.
In Chapter VII of the book of Asmussen and Albrecher \cite{asmussen_albrecher_2010}, the following alternative form of \eqref{sde1} is analysed:
 \begin{equation}\label{sde2}
dU(t) =-dS(t)+p(U(t))\,dt, 
\end{equation}
where $S(t)$ is a compound Poisson process with non-negative summands,  and $p(x)>0$ for all $x\in\mathbb{R}$. The function $p(x)$ is a level-dependent premium rate.
Notice that in \eqref{sde1}, if $X(t)$ has paths of bounded variation, then we can write  $X(t)=-A(t)+ct$, where $A(t)$ is a pure jump subordinator, and \eqref{sde1} can be rewritten as
\begin{eqnarray*}
dU(t) &=&-dA(t)-(\phi(U(t))-c)\,dt\\
&=&-dA(t)+p(U(t))\,dt,
\end{eqnarray*}

by setting  $p(x)=c-\phi(x)$.
Such a level-dependent risk process is dual to a storage process $V$ with a general release rate, which solves
\begin{equation}\label{sde.storage}
  dV(t)=-dX(t)-p(V(t))\,dt,\end{equation}
with the additional condition that $p(0)=0$ (see, e.g., Chapter XIV of Asmussen \cite{asmussen2003}).

The main contribution of this paper is to develop a theory of scale functions for level-dependent L\'evy risk processes.
If $X$ is a spectrally negative L\'evy process with L\'evy exponent $\psi(\lambda)$, then for all $q\ge0$ 
 the unique solution $W^{(q)}$ of the equation $(\psi(\lambda)-q)\mathcal{L}[W^{(q)}](\lambda)=1$ is said to be a 
 scale function, where by $\mathcal{L}[f](\lambda)=\int_0^\infty e^{-\lambda x} f(x)\,dx$ we denote
 the Laplace transform of the function $f$. Moreover, $W^{(q)}(x)=0$ for all $x<0$.
Required notations, definitions, and facts regarding L\'evy processes are recalled in Section~\ref{ss:basic}.

It turns out that some exit probabilities, and moreover the solutions
of interest in risk theory for  L\'evy processes, can be expressed by scale functions. The name {\it scale function} comes from the formula $W^{(0)}(x-d)/W^{(0)}(a-d)$ 
for the probability of exiting the interval $(d,a)$ via the point $a$ for the process $X$ such that $X(0)=x$. % There exists also another scale function $Z^{(q)}$.
A useful survey paper regarding the theory of scale functions of L\'evy processes and their applications in risk theory was provided by 
Kuznetsov et al.\ \cite{kuznetsovetal2012}. We also refer the reader to the books of Kyprianou \cites{kyprianou2006, kyprianou2014}. %\blue{[should we put both books? Its the same one just different editions, right?]}. 

Kyprianou and Loeffen \cite{kyprianouloeffen2010} developed a parallel theory of processes fulfilling equation \eqref{sde1} with 
$\phi(x)=\delta\ind_{\{x>b\}}$, and called the solution $U$ a {\it refracted L\'evy process}.
In their theory, it is essential that $\delta>0$. In this paper, we extend the theory first to solutions $U$ of \eqref{sde1} with some non-decreasing 
function 
\begin{equation}\label{rate_function-mult}
\phi_k(x)=\sum_{j=1}^k\delta_j\ind_{\{x\ge b_j\}},
\end{equation}
called a \textit{multi-refracted L\'evy process}, and then to the case of a general non-decreasing continuously differentiable  $\phi$. 
We prove the existence and  uniqueness of a solution to \eqref{sde1} for rate functions $\phi$ fulfilling the conditions mentioned above. 
Notice that without these assumptions the uniqueness problem for SDE \eqref{sde1} does not have an obvious solution.
Most known results in the literature require either some Lipschitzian properties for $\phi$ or a Brownian component of $X$. See, for example, \cite{Situ}.

The paper consists of two parts. In the first part, we develop some fluctuation formulas for multi-refracted L\'evy processes.
This theory represents a direct, but non-trivial, extension of the paper of Kyprianou and Loeffen \cite{kyprianouloeffen2010 }. We derive formulas for {the} two-sided exit problem, 
one-sided exit problem, resolvents, and ruin probability. We express the solutions to these problems using the scale functions  $w_k^{(q)}$ and $z_k^{(q)}$. 
We also demonstrate that the derivatives of  $w_k^{(q)}$ and $z_k^{(q)}$ fulfill Volterra integral equations of the second kind.
In the second part, we analyze processes with general $\phi$ (or $p$). Here, we assume that $\phi$ is non-decreasing and locally Lipschitz continuous. 
Formulas are developed by approximating $\phi$ by  an approximating sequence $\phi_n$, where $\phi_n$ are the rate functions
of some multi-refracted processes. In the limit, we obtain a level-dependent L\'evy risk process $U$
and the scale functions $w^{(q)}$ and $z^{(q)}$. In this case, the uniqueness of the solution to \eqref{sde1} is clear. 
The derivatives of these scale functions are the solutions of Volterra integral equations of the second kind, as in the multi-refracted case.
As corollaries, we derive ruin probabilities.

The theory of level-dependent L\'evy risk processes or storage processes has a long history. It was mainly developed for compound Poisson processes 
with the ruin function being the main interest
(or the stationary distribution in the context of storage processes). There has been little work regarding scale functions for such  processes. 
Of particular interest is the paper by Brockwell et al.\ \cite{brockwelletal82}, where in the setting of storage processes, the process $X$ is a bounded variation L\'evy process.
The existence of a solution to \eqref{sde.storage} was studied there, and stationary distributions of $V$ were characterised as solutions of Volterra equations.
 For further references and historical comments, we refer the reader to Asmussen   \cite{asmussen2003} and Asmussen and Albrecher \cite{asmussen_albrecher_2010}.

 \subsection{Basic concepts and notations for L\'evy processes}\label{ss:basic}
 Here we present basic concepts and notations from the theory of L\'evy processes (which can be found in the books of
 Kyprianou  \cite{kyprianou2006}, \cite{kyprianou2014}). 
 In this paper,  $X=\{X(t),t\geq 0\}$ is a
 spectrally negative L\'evy process on the filtered probability space $(\Omega,\mathcal{F},\{\mathcal{F}_t , t\geq0\}, \mathbb{P})$.
 %that is a process with stationary and independent increments and no positive jumps.
To avoid trivialities, we exclude the case where $X$ has monotone paths.
As the L\'{e}vy process $X$ has no positive jumps, its moment generating function exists for all $\lambda \geq 0$:
$$
\e \left[ \mathrm{e}^{\lambda X(t)} \right] = \mathrm{e}^{t \psi(\lambda)} ,
$$
where
$$
\psi(\lambda) := \gamma \lambda + \frac{1}{2} \sigma^2 \lambda^2 + \int^{\infty}_0 \left( \mathrm{e}^{-\lambda z} - 1 + \lambda z {\ind_{(0,1)}(z)}  \right) \Pi(\mathrm{d}z) ,
$$
for $\gamma \in \reals$ and $\sigma \geq 0$, and where $\Pi$ is a $\sigma$-finite measure on $(0,\infty)$ such that
$$
\int^{\infty}_0 (1 \wedge z^2) \Pi(\mathrm{d}z) < \infty .
$$
The measure $\Pi$ is called the L\'{e}vy measure of $X$. Finally, note that $\e \left[ X(1) \right] = \psi'(0+)$. We will use the standard Markovian notation: 
the law of $X$ when starting from $X_0 = x$ is denoted by $\p_x$, and the corresponding expectation by $\e_x$. We simply write $\p$ and $\e$ when $x=0$.%Finally, for a random variable $Z$ and an event $A$, $\e [Z;A] := \e [Z \ind_A]$.

When the process $X$  has paths of bounded variation, that is, when $\int^{1}_0 z \Pi(\mathrm{d}z)<\infty$ and $\sigma=0$, we can write
$$
X(t) = \drift t - A(t) ,
$$
where $\drift := \gamma+\int^{1}_0 z \Pi(\mathrm{d}z) $ is the drift of $X$ and $A=\{{ A(t)},t\geq 0\}$ is a pure jump subordinator.
\section{Multi-refracted L\'evy processes}\label{sec:mult.refr}
%A case when $\phi(x)=\delta \ind_{\{U(t)>b\}}$ was  considered in \cite{kyprianouloeffen2010}. They called surplus process $U(t)$ by a refracted process.

For $k \geq 1$, $0<\delta_1,\dots,\delta_k$, and $-\infty<b_1<...<b_k<\infty$, we consider the function
$$\phi(x)=\phi_k(x)=\sum_{i=1}^k\delta_i\ind_{\{x > b_i\}}.$$
The corresponding SDE \eqref{sde1} is given by
\begin{equation}\label{procU_n}
					dU_{k}(t)=dX(t)-\sum_{i=1}^{k}\delta_i \ind_{\{U_{k}(t)>b_i\}}\,dt.
		\end{equation}
		
		In this section, we show that \eqref{procU_n} admits a unique solution  in the strong sense for the case that $X$ 
		is a spectrally negative L\'evy process (not the negative of a subordinator). Here, when $X$ has paths of bounded variation we assume that
		\begin{equation}\label{E:delta}
 0<\delta_1+...+\delta_k < \drift := \gamma + \int_{(0,1)} z \Pi(\mathrm{d}z).
\end{equation}
Note that the special case with $k=1$ was already studied in Kyprianou and Loeffen \cite{kyprianouloeffen2010}.
Furthermore, we study the dynamics of multi-refracted L\'evy processes by establishing a suite of identities, 
written in terms of scale functions, related to the one- and two-sided exit problems and resolvents. 
We also present the formula for the ruin probability. Finally, we show that the scale functions for the multi-refracted processes 
satisfy a Volterra-type integral equation, which we use in Section \ref{general premium} to define the
scale functions for general level-dependent processes. 

Observe that from the SDE \eqref{procU_n}, for any $0 \leq j \leq k$ in each level
interval $(b_{j},b_{j+1}]$ (where $b_0:=-\infty$ and $b_{k+1}:=\infty$) the process $U_k$ evolves as
$X_{j}:=\{X(t)-\sum_{i=1}^{j} \delta_i t: t\geq 0  \}$, which is a spectrally negative L\'evy process that
is not the negative of a subordinator, because we assume \eqref{E:delta}.  
The Laplace exponent of $X_j$ on $[0, \infty)$ is given by
$$
\lambda \mapsto \psi_j(\lambda) := \psi(\lambda) - (\delta_1+...+\delta_j) \lambda,
$$
with right-inverse $$\varphi_{j}(q) = \sup \{ \lambda \geq 0 \mid \psi_j(\lambda) = q\}.$$ 
We will use equivalently the notation $U_0=X_0:=X$ and $\varphi_{0}(q):=\Phi(q)$.
Moreover, for all $0 \leq j \leq k$, $X_j$ has the same L\'evy measure $\Pi$ and diffusion coefficient $\sigma$ as $X$. 
Furthermore, it is easy to notice the recursive relationship between the processes $X_j$ and $X_{j+1}$, i.e.,
$X_{j+1}=\{X_j(t)-\delta_{j+1} t : t\geq 0 \}$.
%\begin{lemma}\label{lem:unique.l} Suppose that $\phi$ is Lipschitz with constant $C>0$, that is 
%$\forall_{ x, y \in \mathbb{R}}$ $|\phi(x)-\phi(y)|\le C|x-y|$.  If there exists a solution of
  %$U(t)=u_0+X(t)-\int_0^t\phi(U(s))\,ds$, then it is unique.
  
%\end{lemma}

%\begin{proof}
  %Suppose there are two solutions $U^1(t)$ and $U^2(t)$. Set $V(t)=|U^1(t)-U^2(t)|$. It fulfills the inequality
  %$$V(t)\le \int_0^t|\phi(U^1(s))-\phi(U^2(s))|\,ds\le C\int_0^tV(s)\,ds.$$
  %We now use Lemma 116 from Situ \cite{Situ} with $\rho(x)=x$ to obtain $V(t)=0$.
  %\end{proof}
  
\begin{theorem}[Existence]\label{lem:mr.exist}
For $k \geq 1$, $0<\delta_1,\dots,\delta_k$, and $-\infty<b_1<...<b_k < \infty$, there exists a strong 
solution $U_k$ to the SDE \eqref{procU_n}. 
		\end{theorem}

	\begin{proof}

		%\textbf{1. (The bounded variation case)}

		 As in Section 3 of \cite{kyprianouloeffen2010}, we provide a pathwise solution 
		to \eqref{procU_n} when the driving L\'evy process $X$ is of bounded variation. We defer the proof for the unbounded variation case to Appendix \ref{UBVcase}. 
		We prove the result by induction. First, the base case ($k=1$) holds by \cite{kyprianouloeffen2010}. 
		Next, we assume that there exists a strong solution $U_{k-1}$ to the SDE
		\begin{equation}\label{sde_bv}
			dU_{k-1}(t)=dX(t)-\sum_{i=1}^{k-1}\delta_i \ind_{\{U_{k-1}(t)>b_i\}}\,dt
		\end{equation}
		with the initial condition $U_{k-1}(0)=x$ for any arbitrary $x$, in order to show that we can construct a process $U_ k$ that solves the SDE \eqref{procU_n}. 
		
		%and assume by induction that there exists a process $U_{k-1}$, which is a solution to \eqref{sde_bv}. 

		 To this end, we define a sequence of times $(S_n)_{n \geq 0}$ and processes $(\overline{U}_{k-1}^n(t), t \geq S_n)_{n \geq 0}$
		 recursively as follows.  First, we set $\overline{U}_{k-1}^0(t) := U_{k-1}(t)$, $t \geq S_0 := 0$, 
		which exists and solves \eqref{sde_bv} by the inductive hypothesis. For $n \geq 1$, we recursively set 
				\begin{align*}
			T_n& :=\inf\left\{t>S_{n-1}: \overline{U}_{k-1}^{n-1}(t) \geq b_k \right\},\\
			S_n& :=\inf\left\{t>T_{n}:\overline{U}_{k-1}^{n-1}(t)-\delta_k(t-T_n) < b_k\right\},
		\end{align*}
		and  $\{ \overline{U}_{k-1}^n, t \geq S_n\}$, starting at $\overline{U}_{k-1}^{n}(S_n)=\overline{U}_{k-1}^{n-1}(S_n)-\delta_k(S_n-T_n)$ and solving
		\begin{equation}\label{U_1_k_1_k}
		d\overline{U}_{k-1}^n(t) = dX(t)-\sum_{i=1}^{k-1}\delta_i \ind_{\{\overline{U}_{k-1}^n(t)>b_i\}}\,dt,  %\quad \textrm{ for $t>S_n$ and $n=1,2,\dots$} ,
		\end{equation}
		which again exists by the inductive hypothesis. %assumption.

		Here, one can observe that the difference between any two consecutive times $S_n$ and $T_n$ is strictly positive. 
	This is because of the fact that on $[T_n, S_n]$, $d \overline{U}^n_{k-1} = d X_{k-1} -\delta_k$, and $X_{k-1}$ is of 
	bounded variation with drift $c - \sum_{j=1}^{k-1} \delta_j > 0$. Thus by Theorem 6.5 in \cite{kyprianou2014} $b_k$ is irregular for $(-\infty, b_k)$, and hence $T_n <  S_n$.
On the other hand, $\overline{U}^n_{k-1}$ always jumps at $S_n$, while it is continuous (creeps upwards) at $T_n$, and so $T_n < S_n < T_{n+1}$.

Now, proceeding as in \cite{kyprianouloeffen2010}, we construct a solution $\{U_k(t):t\geq0\}$ to \eqref{procU_n} as follows:
		\begin{align}\label{pathwise_constr}
			U_k(t)= \begin{cases} \overline{U}_{k-1}^n(t)  \quad \text{for $t\in[S_n,T_{n+1})$ and $n=0,1,2,\dots,$} \\ 
				\overline{U}_{k-1}^{n-1}(t)-\delta_k(t-T_n) \quad \text{for $t\in[T_n,S_n)$ and $n=1,2,\dots$}  \end{cases}
		\end{align}
		Our final step is to prove that the above pathwise-constructed process $U_k$ is a strong solution to the equation \eqref{procU_n}. 
		
		First, for  $t\in[S_0,T_1)$ we note that $U_k(t) = \overline{U}_{k-1}^0(t) =U_{k-1}(t)$ and $\ind_{\{U_{k}(t)>b_k\}}=0$, 
and therefore
		\begin{align*}
			U_k(t)=X(t)-\sum_{i=1}^{k-1}\delta_i\int_0^t \ind_{\{U_{k-1}(s)>b_i\}}\,ds
			=X(t)-\sum_{i=1}^{k}\delta_i\int_0^t \ind_{\{U_{k}(s)>b_i\}}ds,
		\end{align*}
	which solves \eqref{procU_n}.
		 Now, let $t\in[T_1,S_1)$. Then, $U_k(t)=U_{k-1}(t)-\delta_k(t-T_1)$, and hence
		\begin{align*}
		U_k(t) &= U_{k-1}(t) - \delta_k (t-T_1) = X(t)-\sum_{i=1}^{k-1}\delta_i \int_0^t \ind_{\{U_{k-1}(t)>b_i\}}\,dt - 
		\delta_k (t-T_1) \\&= X(t)-\sum_{i=1}^{k-1}\delta_i \int_0^{T_1} \ind_{\{U_{k-1}(t)>b_i\}}\,dt  -
		\sum_{i=1}^{k-1}\delta_i \int_{T_1}^t \ind_{\{U_{k-1}(t)>b_i\}}\,dt - \delta_k (t-T_1) \\
		&= X(t)-\sum_{i=1}^{k}\delta_i \int_0^{T_1} \ind_{\{U_{k}(t)>b_i\}}\,dt  -\sum_{i=1}^{k}\delta_i \int_{T_1}^t \ind_{\{U_{k-1}(t)>b_i\}}\,dt \\
				&= X(t)-\sum_{i=1}^{k}\delta_i \int_0^{t} \ind_{\{U_k(t)>b_i\}}\,dt,		\end{align*}
		where the second to the last equality holds because $U_k =U_{k-1} \leq b_k$ on $[0, T_1)$, and the last equality holds because 
		$U_{k-1} \geq U_k \geq b_k \geq b_i$  for $i \leq k$ on $[T_1,S_1)$. 		
		In particular, we have that
				\begin{align}
		U_k(S_1) &=U_{k-1}(S_1)-\delta_k(S_1-T_1)= X(S_1)-\sum_{i=1}^{k}\delta_i \int_0^{S_1} \ind_{\{U_k(t)>b_i\}}\,dt.		\label{Y_at_S_1}\end{align}
		
		From \eqref{U_1_k_1_k} and \eqref{Y_at_S_1}, it holds that
		\begin{align}
		\begin{split}
\overline{U}_{k-1}^1(t)  &= U_k(S_1) + (X(t) - X(S_1))-\sum_{i=1}^{k-1}\delta_i \int_{S_1}^t \ind_{\{\overline{U}_{k-1}^1(t)> b_i\}}\,dt \\
		&=  X(S_1)-\sum_{i=1}^{k}\delta_i \int_0^{S_1} \ind_{\{U_k(t)>b_i\}}\,dt + (X(t) - X(S_1))-\sum_{i=1}^{k-1}\delta_i \int_{S_1}^t 
		\ind_{\{\overline{U}_{k-1}^1(t)> b_i\}}\,dt \\
		&=  X(t)-\sum_{i=1}^{k}\delta_i \int_0^{S_1} \ind_{\{U_k(t)>b_i\}}\,dt -\sum_{i=1}^{k-1}\delta_i \int_{S_1}^t \ind_{\{\overline{U}_{k-1}^1(t)> b_i\}}
		\,dt.
		\end{split}	\end{align}
		Then, for  $t \in [S_1, T_2)$, by \eqref{U_1_k_1_k}, \eqref{pathwise_constr}, and the fact that $\overline{U}^1_{k-1} = U_k \leq b_k$ on $[S_1,T_2)$, 
		we have
		\begin{align*}
U_k(t)=\overline{U}_{k-1}^1(t) 
		&=  X(t)-\sum_{i=1}^{k}\delta_i \int_0^{S_1} \ind_{\{U_k(t)>b_i\}}\,dt -\sum_{i=1}^{k-1}\delta_i \int_{S_1}^t \ind_{\{\overline{U}_{k-1}^1(t)> b_i\}}\,dt \\
		&=  X(t)-\sum_{i=1}^{k}\delta_i \int_0^{t} \ind_{\{U_k(t)>b_i\}}\,dt.
			\end{align*}
%
%\eqref{U_1_def_kazu} and \eqref{Y_at_S_1} give
%		\begin{align*}
%		Y(t) &= U_{k-1}^1(t) = Y(S_1) + (X(t) - X(S_1))-\sum_{i=1}^{k-1}\delta_i \int_{S_1}^t \ind_{\{U_{k-1}^1(t)> b_i\}}\,dt \\
%		&=  X(S_1)-\sum_{i=1}^{k-1}\delta_i \int_0^{S_1} \ind_{\{Y(t)>b_i\}}\,dt + (X(t) - X(S_1))-\sum_{i=1}^{k-1}\delta_i \int_{S_1}^t \ind_{\{U_{k-1}^1(t)> b_i\}}\,dt \\
%				&=  X(t)-\sum_{i=1}^{k-1}\delta_i \int_0^{t} \ind_{\{Y(t)>b_i\}}\,dt,		\end{align*}
%				where the last equality holds because $U^1_{k-1} = Y$ on $[S_1,T_1]$.
				
				On $[T_2,S_2)$, it follows from \eqref{U_1_k_1_k} and \eqref{pathwise_constr} that
				
						\begin{align*}
						U_k(t) &= \overline{U}_{k-1}^1(t) -\delta_k(t-T_2)\\
%U_{k-1}^1(t) 
		&=  X(t)-\sum_{i=1}^{k}\delta_i \int_0^{S_1} \ind_{\{U_k(t)>b_i\}}\,dt -\sum_{i=1}^{k-1}\delta_i \int_{S_1}^{T_2} 
		\ind_{\{\overline{U}_{k-1}^1(t)> b_i\}}\,dt \\&-\sum_{i=1}^{k-1}\delta_i \int_{T_2}^t \ind_{\{\overline{U}_{k-1}^1(t)> b_i\}}\,dt  -\delta_k(t-T_2)\\
		&=X(t)-\sum_{i=1}^{k}\delta_i \int_0^{t} \ind_{\{U_k(t)>b_i\}}\,dt.
			\end{align*}
		
		The previous identity follows from the fact that $U^1_{k-1} = U_k \leq b_k$ on $[S_1,T_2)$ and $\overline{U}^1_{k-1} \geq U_k \geq b_k \geq b_i$ for all $1 \leq i \leq k$ on $[T_2, S_2).$
				
		Hence, by proceeding by induction on the time intervals $[S_n, T_{n+1})$ and $[T_{n+1}, S_{n+1})$ for $n=2,3,...$ we obtain that, for any $t>0$,
		the process defined in \eqref{pathwise_constr} fulfills \eqref{procU_n}. This completes the proof of the existence of a strong solution for the bounded variation case.
		\end{proof}

		The proof of the uniqueness follows verbatim from Proposition 15 of \cite{kyprianouloeffen2010} (see also Example 2.4 on
p. 286 of \cite{karatzas91}), using the fact that the function $\phi_k$ is non-decreasing for $k=1,2,\dots$. 
 Hence, we have the following result.

		\begin{lemma} \textbf{(Uniqueness)} \label{lemma_uniqueness}
%Assume that $\delta_j>0$ for all $j=1,2,\dots,k$.
Under the assumptions of Theorem \ref{lem:mr.exist}, there exists a unique strong solution to (\ref{procU_n}).
\end{lemma}

Using the argument given in Remark 3 of \cite{kyprianouloeffen2010}, we can obtain the following.

			\begin{lemma} \textbf{(Strong Markov property)} 
			Under the assumptions of Theorem \ref{lem:mr.exist}, for each $k\geq1$, the process $U_k$, which is the unique solution to \eqref{procU_n}, is a strong Markov process.
			\end{lemma}

\subsection{Scale functions}\label{sec:scale}
In this section, we present a few facts concerning scale functions that are important for writing many fluctuation identities for 
spectrally negative L\'evy processes.

First, for any $k \geq 0$ the scale functions $W_k^{(q)}$ and $Z_k^{(q)}$ of $X_k$ are defined as follows.
For $q \geq 0$, the $q$-scale function $W_k^{(q)}$ of the process $X_k$ is defined as the continuous function on $[0,\infty)$ 
whose Laplace transform satisfies
\begin{align} \label{laplace_W_k}
\int_0^{\infty} \mathrm{e}^{- \lambda y} W_{k}^{(q)} (y) \mathrm{d}y = \frac{1}{\psi_k(\lambda)- q} ,
 \quad \text{for $\lambda > \varphi_{k}(q).$}
\end{align}
The scale function $W_k^{(q)}$ is positive, strictly increasing and continuous for $x\geq0$. 
We extend $W_k^{(q)}$ to the whole real line by setting $W_k^{(q)}(x)=0$ for $x<0$. In particular, we write $W_k = W_k^{(0)}$ when $q=0$. 
We also define
\begin{equation}\label{eq:zqscale}
Z_k^{(q)}(x) = 1 + q \int_0^x W_k^{(q)}(y)\mathrm d y, \quad x \in \mathbb R.
\end{equation}
Note that $Z_k = Z_k^{(0)}=1$ when $q=0$. 
In particular, for $k=0$, we set $W^{(q)}=W_0^{(q)}$ and $Z^{(q)}=Z_0^{(q)}$ for the scale functions of the spectrally negative L\'evy process $X$.

For any $k \geq 0$, the initial value of $W_k^{(q)}$ is given by
\begin{equation*}\label{initialvalues}
\begin{split}
W_k^{(q)}(0) &=
\begin{cases}
\frac{1}{c-\sum_{j=1}^k \delta_j,} & \text{when $\sigma=0$ and $\int_{0}^1 z \Pi(\mathrm{d}z) < \infty$,} \\
0, & \text{otherwise.}
\end{cases}
\end{split}
\end{equation*}

In \cite{kyprianouloeffen2010}, many fluctuation identities, including the probability of ruin for $U_1$, have been derived
using the scale functions for $U_1$. For $q \geq 0$, $x,d\in \reals$ and $b_1>d$, define
\begin{equation}\label{small w}
w_1^{\left(q\right)}(x;d) := W^{\left(q\right)}(x-d) + \delta_1\int^{x}_{b_1}W_1^{\left(q\right)}(x-y)W^{\left(q\right)'}(y-d)\md y .
\end{equation} 
%\red{[So we can have $d > b_1$, right? Looks like that for this case $w_1^{\left(q\right)}(x;d) = W_1(x-d)$.]}\blue{[I think not exactly...
%\begin{align*}
%w_1^{\left(q\right)}(x;d) :&= W^{\left(q\right)}(x-d) + \delta_1\int^{x}_{b_1}W_1^{\left(q\right)}(x-y)W^{\left(q\right)'}(y-d)\md y\\
%&=W^{\left(q\right)}(x-d) + \delta_1\int^{x}_{d}W_1^{\left(q\right)}(x-y)W^{\left(q\right)'}(y-d)\md y\\
%&=W^{\left(q\right)}(x-d) + \delta_1\int^{x-d}_{0}W_1^{\left(q\right)}(x-d-y)W^{\left(q\right)'}(y)\md y\\
%&=(1-\delta_1W^{(q)}(0))W_1^{(q)}(x-d).
%\end{align*}
%]}
In the remainder of this paper, we will use the convention that $\int_a^b=0$ if $b<a$. 
Hence, note that when $x \leq b_1$ we have that
\begin{equation}\label{w_1}
w_1^{\left(q\right)}(x;d)=W^{\left(q\right)}(x-d) .
\end{equation}
%\blue{In the rest of the paper we set $w_0^{(q)}(x;d)=W^{(q)}(x-d)$ for all $x,d\in\mathbb{R}$}.

The following definition will be useful for a compact presentation of the main results of Section \ref{sec:main}. 

\begin{defi}\label{Delta}
For any $0 \leq k$, $-\infty =: b_0 < b_1 < \cdots < b_k<b_{k+1} :=\infty$, and $y\in\mathbb{R}$, define
\begin{align*}
\Xi_{\phi_k}(y) := 1- W^{(q)}(0)\phi_k(y).
\end{align*}

For the unbounded variation case, we note that the fact that $W^{(q)}(0)=0$ implies that $\Xi_{\phi_k}(y)=1$ for all $y\in\mathbb{R}$. 
On the other hand, in the bounded variation case with $y\in(b_i,b_{i+1}]$ and $i\leq k-1$, we have that
	\[
	\Xi_{\phi_k}(y)=1- W^{(q)}(0)\sum_{j=1}^i\delta_j=\prod_{j=1}^{i} \left(1-\delta_j W_{j-1}^{(q)}(0)\right),
	\]
and similarly for $y>b_k$,	we obtain $\Xi_{\phi_k}(y)=\prod_{j=1}^{k} \left(1-\delta_j W_{j-1}^{(q)}(0)\right)$.
	\end{defi}

%\begin{defi}\label{Delta}
%We set $\Xi_0:=1$, then for $k \geq 1$ define
%\begin{align*}
%\Xi_k := \left\{ \begin{array}{ll} 1 & \textrm{if $X$ is of unbounded variation} \\
%W(0)/W_k(0) &  \textrm{if $X$ is of bounded variation.} \end{array}\right.
%\end{align*}
%Since for the case of bounded variation processes, note that
%\begin{align*}
%1-\delta_j W_{j-1}^{(q)}(0) = 1- \frac {\delta_j} {c - \delta_1 - \cdots - \delta_{j-1}} = \frac {c - \delta_1 - \cdots - \delta_j} {c - \delta_1 - \cdots - \delta_{j-1}}.
%\end{align*}
%Hence for $k \geq 0$,
%\begin{align*}
%&\prod_{j=1}^{k} \left(1-\delta_j W_{j-1}^{(q)}(0)\right) = \frac {c - \delta_1} {c } \times 
%\frac {c - \delta_1  - \delta_2} {c - \delta_1 }\times  \cdots \times 
%\frac {c - \delta_1 - \cdots - \delta_k} {c - \delta_1 - \cdots - \delta_{k-1}}\\& = \frac {c - \sum_{i=1}^{k} \delta_i } c
%= \frac {W(0)} {W_k(0)}=\Xi_k.
%\end{align*}
%\end{defi}

\subsection{Exit problems and resolvents}\label{sec:main}

For $ a \in \mathbb{R}$ and $k\geq 1$, define the following first-passage stopping times:
\begin{align*}
\tau_{k}^{a,-} & := \inf\{t>0 \colon X_k(t)<a\} \quad \textrm{and} \quad \tau_{k}^{a,+} := \inf\{t>0 \colon X_k(t)\geq a\},\\
\kappa_k^{a,-} & := \inf\{t>0 \colon U_k(t)<a\} \quad \textrm{and} \quad \kappa_k^{a,+} := \inf\{t>0 \colon U_k(t)\geq a\},
\end{align*}
with the convention that $\inf \emptyset=\infty$.

%\green{[One option is to change from $k$ to $j$ or something but if it is too much work to change, it is not a big deal.]}

%\blue{If you don't mind I will stay with $\kappa_k$- is it OK for you ?}

First, we state the result for the two-sided exit problem for $k$-multi-refracted L\'evy processes.
The special case with $k=1$ was already derived in Theorem 4 of  Kyprianou and Loeffen \cite{kyprianouloeffen2010}. 
We remark that for $X_k$ and $U_1$ we have, for example for $a > d$ and $x \leq a$, that
\begin{align}
\e_x \left[ \mathrm{e}^{-q \tau_k^{a,+}} \ind_{\{\tau_k^{a,+}< \tau_k^{d,-}\}} \right] &= \frac{W_k^{(q)}(x-d)}{W_k^{(q)}(a-d)},  \quad k \geq 0,\\ \label{two-sided_kypr&loeffen}
\e_x \left[ \mathrm{e}^{-q \kappa_1^{a,+}} \ind_{\{\kappa_1^{a,+}< \kappa_1^{d,-}\}} \right] &= \frac{\wq_1(x;d)}{\wq_1(a;d)} ,
\end{align}
where the latter holds by \cite{kyprianouloeffen2010}.

\begin{theorem}{\textbf{(Two-sided exit problem)}}\label{th:mr.recursion}

Fix $k \geq 1$ and $q \geq 0$. %$-\infty < b_1 < \cdots < b_k $ $\delta_j > 0$, $1 \leq j \leq k$, 
\begin{itemize}
\item[(i)]  For  $d < b_1<\dots< b_k\leq a$ and $d \leq x \leq a$, we have 
\begin{equation}\label{main_twosided_up}
\e_x \left[ \mathrm{e}^{-q \kappa_k^{a,+}} \ind_{\{\kappa_k^{a,+}< \kappa_k^{d,-}\}} \right] =\frac{\wq_k(x;d)}{\wq_k(a;d)},
\end{equation} 
where $\wq_k$ is defined by the recursion
%\red{[should we construct more carefully $\wq_k(x;d)$ above Theorem 5? This function is used in other theorems.]}\blue{-
%I think it is better to leave this definition here, because otherwise you also have define above $u^{(q)}$ etc. ??}
\begin{equation}\label{eq:mr.recursion}
  \wq_k(x;d):= \wq_{k-1}(x;d)+\delta_{k}\int^{x}_{b_{k}}W_k^{\left(q\right)}(x-y)w_{k-1}^{(q)\prime}(y;d)\md y.
   \end{equation}
The function $\wq_{k-1}(x;d)$ is the scale function associated with the process $U_{k-1}$, and the initial function is given by $w_{0}^{\left(q\right)}(x;d)=W^{(q)}(x-d)$.
%is defined in (\ref{small w}). 
\item[(ii)] For $0 < b_1<\dots< b_k\leq a$ and $0 \leq x \leq a$,
	\begin{equation}\label{main_twosided_down}
	\e_x \left[ \mathrm{e}^{-q \kappa_k^{0,-}} \ind_{\{\kappa_k^{0,-}<\kappa_k^{a,+}\}} \right] =z_k^{(q)}(x)-\frac{z^{(q)}_k(a)}{\wq_k(a)}\wq_k(x),
	\end{equation}
	where $z^{(q)}_{k}$ is defined by the recursion
	\begin{equation}\label{eq:mr.recursion_z_q}
	z^{(q)}_{k}(x):=z^{(q)}_{k-1}(x)+\delta_{k}\int^{x}_{b_{k}}W_k^{\left(q\right)}(x-y)z_{k-1}^{\left(q\right)\prime}(y)\md y.
	\end{equation}
	The scale function $z^{(q)}_{k-1}(x)$ is associated with the process $U_{k-1}$, and the initial function is given by $z_{0}^{\left(q\right)}(x)=Z^{(q)}(x)$. %given by
	%\begin{equation*}\label{z1}
	%z_{1}^{\left(q\right)}(x)=Z^{(q)}(x)+\delta_1q\int_{b_1}^x W_1^{(q)}(x-y)W^{(q)}(y)dy.
	%\end{equation*}

\end{itemize}
\end{theorem}

%\red{[In the proof, 
%	I think $\e_x \left[ \mathrm{e}^{-q \kappa_k^{0,-}} \ind_{\{\kappa_k^{0,-}<\infty\}} \right]$ is first obtained. 
%		Theorem \ref{one_sided} is proved separately.  Is this because this way we obtain a simpler form?   
%		Also, by setting $q=0$, the ruin probability can be more easily obtained?]}
%		
%\red{Is $z_k^{(0)} = 1$? Then I guess we should point this out?}
%\blue{-Yes it is, but you have to prove this first. So, I don't see if it will be easier to get the ruin probability if we set $q=0$ in
%$\e_x \left[ \mathrm{e}^{-q \kappa_k^{0,-}} \ind_{\{\kappa_k^{0,-}<\infty\}} \right]$, because in this case
%you  ALSO have to compute $z_{k-1}(\infty)$ and $w_{k-1}(\infty)$ which I did in the proof for ruin probability...}

\begin{corol}{\textbf{(One-sided exit problem)}}\label{one_sided}\\ 
Fix $k \geq 1$. %$\delta_j > 0$, $1 \leq j \leq k$. %and $-\infty < b_1 < \cdots < b_k$.

\begin{itemize}
	\item[(i)] For $x \geq 0$, $b_1 > 0$ and $q>0$, we have
		\begin{equation}\label{main_onesided_down}
	\e_x \left[ \mathrm{e}^{-q \kappa_k^{0,-}} \ind_{\{\kappa_k^{0,-}<\infty\}} \right] =z_k^{(q)}(x)-\frac{\int_{b_k}^{\infty} 
	e^{-\varphi_k(q)z}z_{k-1}^{(q)\prime}(z)\md z }{\int^{\infty}_{b_{k}}e^{-\varphi_k(q)z}w_{k-1}^{(q)\prime}(z)\md z}\wq_k(x).
	\end{equation}
	
	%\blue{Let us take $b_1<0<b_2$, in this case we have
	%\begin{align*}
	%z_1^{(q)}(x)&=Z^{(q)}(x)+q\delta_1\int_{b_1}^xW_1^{(q)}(x-y)W^{(q)}(y)dy\\
	%&=Z^{(q)}(x)+q\delta_1\int_{0}^xW_1^{(q)}(x-y)W^{(q)}(y)dy\\
	%&=Z_1^{(q)}(x).
	%\end{align*}
	%This implies that
	%\[
	%z_2^{(q)}(x)=Z_1^{(q)}(x)+q\delta_2\int_{b_2}^xW_2^{(q)}(x-y)W_1^{(q)}(y)dy.
	%\]
	%Similarly
	%\begin{align*}
	%w_1^{(q)}(x)&=W^{(q)}(x)+q\delta_1\int_{b_1}^xW_1^{(q)}(x-y)W^{(q)\prime}(y)dy\\
	%&=W^{(q)}(x)+q\delta_1\int_{0}^xW_1^{(q)}(x-y)W^{(q)\prime}(y)dy\\
	%&=(1-\delta_1W^{(q)}(0))W_1^{(q)}(x).
	%\end{align*}
	%Hence
	%\[
	%w_2^{(q)}(x)=(1-\delta_1W^{(q)}(0))\left(W_1^{(q)}(x)+\delta_2\int_{b_2}^xW_2^{(q)}(x-y)W_1^{(q)\prime}(y)dy\right).\]
	%Hence
	%\begin{align*}
	%\e_x \left[ \mathrm{e}^{-q \kappa_k^{0,-}} \ind_{\{\kappa_k^{0,-}<\infty\}} \right]&=Z_1^{(q)}(x)+q\delta_2\int_{b_2}^xW_2^{(q)}(x-y)W_1^{(q)}(y)dy\\
	%&-\frac{q\int_{b_2}^{\infty} 
		%e^{-\varphi_2(q)z}W_{1}^{(q)\prime}(z)\md z }{\int^{\infty}_{b_{2}}e^{-\varphi_k(q)z}W_{1}^{(q)\prime}(z)\md z}\left(W_1^{(q)}(x)+\delta_2\int_{b_2}^xW_2^{(q)}(x-y)W_1^{(q)\prime}(y)dy\right)
	%\end{align*}
	%Recovering the usual refracted case at the level $b_2$ for the process $X_t-\delta_1t$. So it might hold for some barriers $b_1,b_2,\dots,b_i<0$ by induction?
		%}
	\item[(ii)]
	For $b_1<\dots <b_k\leq a$, $x \leq a$ and $q\geq 0$,
	\begin{equation}\label{main_onesided_up}
	\e_x \left[ \mathrm{e}^{-q \kappa_k^{a,+}} \ind_{\{\kappa_k^{a,+}<\infty\}} \right] =
	\frac{u^{(q)}_k(x)}{u^{(q)}_k(a)}.
	\end{equation}
	Here,
	$u^{(q)}_{k}$ is defined by the recursion
\begin{equation}\label{eq:mr.recursion_u_q}
	u^{(q)}_{k}(x) := u^{(q)}_{k-1}(x)+\delta_{k}\int^{x}_{b_{k}}W_k^{\left(q\right)}(x-y)u_{k-1}^{\left(q\right)\prime}(y)\md y,
	\end{equation}
	where the function $u^{(q)}_{k-1}(x)$ is associated with the process $U_{k-1}$, and the initial function is given by $u_{0}^{\left(q\right)}(x)=e^{\Phi(q)x }$. 
	%\red{[change to $\varphi_0(q)$? I dont think $\Phi(q)$ has been defined. Please check throughout the paper.]}
	%is given by
%	\begin{equation*}\label{u1}
%	u_{1}^{\left(q\right)}(x)=e^{\Phi(q)x }+\delta_1 \Phi(q) \int_{b_1}^x e^{\Phi(q)y} W_1^{\left(q\right)}(x-y)dy.
%	\end{equation*}
%	\blue{
%		\begin{equation*}\label{u1}
%		u_{1}^{\left(q\right)}(x)=e^{\varphi_0(q)x }+\delta_1 \varphi_0(q) \int_{b_1}^x e^{\varphi_0(q)y} W_1^{\left(q\right)}(x-y)dy.%?
%		\end{equation*}	
		%}
	\end{itemize}
	\end{corol}
For $q=0$, we  write $w_k^{(0)}(x;d)=w_k(x;d)$ with $k \geq 0$.
Moreover, for $d=0$ we denote $w_k^{(q)}(x;0)=w_k^{(q)}(x)$.
Furthermore, we also use the same convention for the functions $z^{(q)}_k$ and $u^{(q)}_k$.

\begin{theorem}{\textbf{(Resolvents)}}\label{Resolvents}\\
Fix a Borel set $\mathcal{B} \subseteq \mathbb{R}$ and $k \geq 1$. %$\delta_j > 0$. %and $-\infty < b_1 < \cdots < b_k $. 

% $1 \leq j \leq k$ and $-\infty =: b_0 < b_1 < \cdots < b_k<b_{k+1} :=\infty$. 

\begin{itemize}
\item[(i)] For $d < b_1<\dots< b_k\leq a$, $d \leq x \leq a$ and $q \geq 0$, 
\begin{align}\label{resolvent1}
\e_x\left[\int_0^{\kappa_k^{a,+}\wedge \kappa_k^{d,-}}e^{-qt}\ind_{\{U_k(t)\in \mathcal{B}\}}dt\right]
= \int_{\mathcal{B} \cap (d, a)} \frac{ \frac{w_k^{(q)}(x;d)}{w_k^{(q)}(a;d)}w_k^{(q)}(a;y)-w_k^{(q)}(x;y)}{ \Xi_{\phi_k}(y)} dy,
\end{align}

where the scale function $\wq_{k}(x;z)$ is defined in Theorem \ref{th:mr.recursion} (i).
%\red{[do we need to repeat this? we can omit?]}
\item[(ii)] For $x \geq 0$, $b_1>0$ and $q > 0$,
\begin{align}\label{resolvent2}
\e_x\left[\int_0^{\kappa_k^{0,-}}e^{-qt}\ind_{\{U_k(t)\in \mathcal{B}\}}dt\right]
= \int_{\mathcal{B} \cap (0, \infty)} \frac{\frac{w_k^{(q)}(x)}{v_k^{(q)}(0)}v_k^{(q)}(y)-w_k^{(q)}(x;y)}
{\Xi_{\phi_k}(y)} dy,
\end{align}
where $v_{k}^{(q)}(y):=\delta_{k}\int^{\infty}_{b_{k}}e^{-\varphi_k(q)z}w_{k-1}^{(q)\prime}(z;y)\md z$, and the scale function
 $\wq_{k}(x;z)$ is defined in Theorem \ref{th:mr.recursion} (i).
\item[(iii)] For  $ x, b_k \leq a$ and $q \geq 0$, %\red{$q>0?$}
	\begin{align}\label{resolvent3}
		\e_x\left[\int_0^{\kappa_k^{a,+}}e^{-qt}\ind_{\{U_k(t)\in \mathcal{B}\}}dt\right]
		=\int_{\mathcal{B} \cap (-\infty,a)} \frac{\frac{u_k^{(q)}(x)}{u_k^{(q)}(a)}w_k^{(q)}(a;y)-w_k^{(q)}(x;y)}
		{\Xi_{\phi_k}(y)} dy,
	\end{align}
	where the functions $w_k^{(q)}(x;y)$ and $u_k^{(q)}(x)$ are defined in Theorems \ref{th:mr.recursion} and \ref{one_sided}, respectively.
\item[(iv)] For $ x \in \mathbb{R}$ and $q > 0$,
	\begin{align}\label{resolvent4}
		\e_x\left[\int_0^{\infty}e^{-qt}\ind_{\{U_k(t)\in \mathcal{B}\}}dt\right]
		= \int_{\mathcal{B} }  \frac{\frac{u_k^{(q)}(x) \int_{b_k}^{\infty} e^{-\varphi_k(q)z} w_{k-1}^{(q)\prime}(z;y) dz  }
		{\int_{b_k}^{\infty} e^{-\varphi_k(q)z} u_{k-1}^{(q)\prime}(z) dz }-w_k^{(q)}(x;y)}
		{\Xi_{\phi_k}(y)} dy,
	\end{align}
where the functions $w_k^{(q)}(x;y)$ and $u_k^{(q)}(x)$ are defined in Theorems \ref{th:mr.recursion} and \ref{one_sided}, respectively. 
\end{itemize}
	\end{theorem}
	Note that in the above expressions, the derivatives of the scale functions $w^{(q)}_k$, $z^{(q)}_k$, and $u^{(q)}_k$ appear.
	We refer to Lemma \ref{Smoothness} for a further explanation regarding why the integrals given in the above identities are well defined.

\begin{corol}{\textbf{(Ruin probability)}}\label{ruin_probab}
For any $k\ge 1$, $b_1 > 0$, $x\geq0$, and $0<\sum_{j=1}^k\delta_j<\mathbb{E}[X(1)]$, we have
\begin{align}\label{rp2}
\mathbb{P}_x\left(\kappa_k^{0,-}<\infty\right)=1-\frac{\mathbb{E}[X(1)]-\sum_{j=1}^k\delta_j}{1-\sum_{j=1}^{k} \delta_j w_{j-1}(b_j)}w_k(x).
\end{align}
\end{corol}

The proofs of the above theorems and corollaries  are given in the Appendix, because the arguments tend to be technical, and the results intuitively 
hold in a similar manner to the case presented in \cite{kyprianouloeffen2010}. 
In Appendix \ref{BVcase}, we begin by providing the proofs of the identities \eqref{main_twosided_up} and \eqref{resolvent1} under the assumption
that $X$ is a L\'evy process that has paths of bounded variation. 
%Then in the next step we prove the identity \eqref{resolvent1} for the unbounded variation case using approximation arguments as in the proof 
%of Theorem 6 of \cite{kyprianouloeffen2010}. 
%Then using identity \eqref{resolvent1} we prove the rest of Theorem \ref{Resolvents} for the general (bounded and unbounded variation) L\'evy process $X$. Next, applying results from Theorem \ref{Resolvents} we prove that identity \eqref{main_twosided_up} holds true for general (bounded and unbounded variation) 
%L\'evy processes. 
The proofs for the case of unbounded variation are presented in Appendix \ref{UBVcase}.
%since in our opinion this part tends to be technical and intuitively hold similarly to the case presented in \cite{kyprianouloeffen2010}.
We remark here that our further reasoning is common for a general L\'evy process $X$, and hence
using formulas \eqref{main_twosided_up} and \eqref{resolvent1} we obtain the remainder of the identities for a general class of spectrally negative L\'evy processes.
For details, we refer the reader to Appendix \ref{Generalcase}.

\subsection{Analytical properties of multi-refracted scale functions}
In this section, we summarize the properties of the multi-refracted scale functions
  $w_k^{(q)}$ and $z_k^{(q)}$, which will be crucial for further proofs in this paper. Moreover, the following results
are important for many applied probability models, such as optimal dividend problems. 

\begin{lemma}{\textbf{(Smoothness)}}\label{Smoothness}
 In general,  $w_k^{(q)}(\cdot;d)$ is a.e.~continuously differentiable.
In particular, if $X$ is of unbounded variation, then $w_k^{(q)}(\cdot;d)$ is $C^1(d,\infty)$. 
 On the other hand, if we assume that $W_k^{(q)}(\cdot-d) \in C^1(d,\infty)$ for all $k \geq 0$ and $X$ is of bounded variation, 
then $w_k^{(q)}(\cdot;d)$ is also $C^1((d,\infty) \backslash \{ b_1, \ldots, b_k\})$, where in particular $w_0^{(q)}(\cdot;d)$ is $C^1(d,\infty)$. 
	\end{lemma}
	\begin{proof}
	 
First take $k=0$. Because $w_0^{(q)}(\cdot;d)=W^{(q)}(\cdot-d)$, the statement follows from Lemma 2.4 in \cite{kuznetsovetal2012}.

%$k=1$, then using Leibniz integral rule we obtain
%$$ w_1^{(q)\prime}(x)=\left\{\begin{array}{ccc}
%W^{(q)\prime}(x)  & \textrm{for $x < b_1$} \\ 
%W^{(q)\prime}(x)(1+\delta_1 W_1^{(q)}(0))+\delta_1\int_{b_1}^x W_1^{(q)\prime}(x-z)W^{(q)\prime}(z)dz & \textrm{for $x > b_1$}
%\end{array} \right..
%$$

Next, for induction, we assume that for the bounded variation case $w_{k-1}^{(q)}(x;d) \in C^1(d,\infty)$ except at $b_1,...,b_{k-1}$.
From \eqref{eq:mr.recursion}, we compute
$$ w_k^{(q)\prime}(x;d)=\left\{\begin{array}{cc}
w_{k-1}^{(q)\prime}(x;d)& \textrm{for $x < b_k$} \\ 
w_{k-1}^{(q)\prime}(x;d)(1+\delta_k W_k^{(q)}(0))+\delta_1\int_{b_k}^x W_k^{(q)\prime}(x-z)w_{k-1}^{(q)\prime}(z;d)dz & \textrm{for $x > b_k.$}
\end{array} \right..
$$
The above equation, together with the inductive assumption and the fact that function $W_k^{(q)} $ is in general a.e.~continuously differentiable, 
gives that $w_k^{(q)}(x;d)$ is also a.e.~continuously differentiable. 
Moreover, for the unbounded variation case, because $W_k^{(q)}(0)=0$ we have that
$\lim_{x \to b_k-}w_k^{(q)\prime}(x;d)=w_{k-1}^{(q)\prime}(b_k;d)= \lim_{x \to b_k +}w_k^{(q)\prime}(x;d)$. 
Then, because $W_k^{(q)}(\cdot-d)$, $w_{k-1}^{(q)}(\cdot;d) \in C^1(d,\infty)$ we obtain that $w_k^{(q)}(\cdot;d)$ is $C^1(d,\infty)$. 
Finally, for the bounded variation case we have that $W_k^{(q)}(0)>0$, and then 
\begin{align*}
\lim_{x \to b_k-}w_k^{(q)\prime}(x;d)=w_{k-1}^{(q)\prime}(b_k;d) \neq 
w_{k-1}^{(q)\prime}(b_k;d)(1+\delta_k W_k^{(q)}(0))= \lim_{x \to b_k+}w_k^{(q)\prime}(x;d),
\end{align*}
so that the derivative does not exist at $b_k$. 
Furthermore, for the bounded variation case it follows from the inductive assumption and the assumption that $W_k^{(q)}(\cdot-d) \in C^1(d,\infty)$ that
$w_k^{(q)}$ belongs to $C^1((d,\infty) \backslash \{ b_1, \ldots, b_k\})$.
%Now using the same argument as for $k=1$, we obtain that $w_{k}^{(q)} $ has no derivative at $b_{k}$
%and   $w_{k}^{(q)} \in C^1(0,\infty)$ for the unbounded variation case.
 
\end{proof}

\begin{lemma}{\textbf{(Monotonicity)}}
The scale function $w_k^{(q)}$ defined in \eqref{eq:mr.recursion} is an increasing function for $k\geq 0$ and $x\geq d$.
\end{lemma}
\begin{proof}
Because we have for $k=0$ that $w_k^{(q)}(\cdot;d)=W^{(q)}(\cdot-d)$, the monotonicity follows directly from the definition. 
  
Now, assume for induction that $w_{k-1}^{(q)}$ is an increasing function for $x\geq d$.
Then, again from \eqref{eq:mr.recursion}, for any $x>y$ we obtain that
\begin{align*}
&w_k^{(q)}(x;d)-w_k^{(q)}(y;d)\\&=w_{k-1}^{(q)}(x;d)+\delta_k\int_{b_k}^x W_k^{(q)}(x-z)w_{k-1}^{(q)\prime}(z;d)dz
-w_{k-1}^{(q)}(y;d)-\delta_k\int_{b_k}^{y} W_k^{(q)}(y-z)w_{k-1}^{(q)\prime}(z;d)dz\\&
=w_{k-1}^{(q)}(x;d)-w_{k-1}^{(q)}(y;d)+\delta_k\int_{b_k}^y\left( W_k^{(q)}(x-z)-W_k^{(q)}(y-z)\right)w_{k-1}^{(q)\prime}(z;d)dz\\
&+\delta_k\int_{y}^x W_k^{(q)}(x-z)w_{k-1}^{(q)\prime}(z;d)dz >0.
\end{align*}
The result follows by the induction assumption and the facts that $\delta_k>0$ and $W_k^{(q)}$ is a non-negative, increasing function in its domain. 
\end{proof}
  
	As a consequence of the previous lemma, we have the following result.
		\begin{corol} \label{derivative}
			For any $d \in \mathbb{R}$ and $k\geq 0$, $w_{k}^{(q)\prime}(x;d)> 0$ for a.e. $x > d$.
		\end{corol}
		
	\begin{lemma}{\textbf{(Behavior at $0$)}} For any $d \in \mathbb{R}$, $b_1>d$ and $k \geq 1$,
	 $w_k^{(q)}(d;d)=W^{(q)}(0)$ and $z_k^{(q)}(0)=Z^{(q)}(0)=1$. In particular, $w_k^{(q)}(0)=W^{(q)}(0)$. 
	\end{lemma}
	\begin{proof}
	From \eqref{small w}, it is easy to check that  $w_1^{(q)}(d;d)=W^{(q)}(0)$. Hence from \eqref{eq:mr.recursion} it follows that
	$w_k^{(q)}(d;d)=w_{k-1}^{(q)}(d;d)$, and by induction we obtain that $w_k^{(q)}(d;d)=W^{(q)}(0)$ for any $k \geq 1$.
	A similar argument yields that $z_k^{(q)}(0)=Z^{(q)}(0)=1$ for any $k \geq 1$.
	\end{proof}
	
By solving the recursion given in (\ref{eq:mr.recursion}) with the initial condition (\ref{small w}), we get the following result.
%Moreover we use the convention that if \red{there} does not exist $l \in \mathbb{N}$ such that $0<l<i$ we set
%$\delta_l  =0$ and similarly in other products.\blue{Can we remove this sentence? An exmpty product is always understood as $1$.}

% \red{[do we need to say this? I think it is always understood that the product of an empty set is $1$?]}
%\blue{ I will leave this sentence to be very formal, just in case...?}
\begin{prop}
For $k\geq 1$, we have
\begin{eqnarray}\label{w_n_explicit}\nonumber
 \lefteqn{w_k^{(q)}(x;d)=W^{(q)}(x-d) + \sum_{i=1}^k \delta_i \prod_{0<l<i} \left(1+\delta_l W_l^{\left(q\right)}(0)\right) 
 \int^{x}_{b_{i}}W_i^{\left(q\right)}(x-y)W^{\left(q\right)'}(y-d)\md y} \\\nonumber &  +  &
\sum_{1\leq i_1 <i_2 \leq k} \delta_{i_1} \delta_{i_2} \prod_{0<l\neq i_1<i_2} \left(1+\delta_l  W_l^{\left(q\right)}(0)\right) 
 \int^{x}_{b_{i_2}}W_{i_2}^{\left(q\right)}(x-y_1) \times \\\nonumber  & & \times \quad
\int^{y_1}_{b_{i_1}}W_{i_1}^{\left(q\right)'}(y_1-y_2)W^{\left(q\right)'}(y_2-d)\,\md y_2\, \md y_1
\\
\nonumber &  + &
\sum_{1\leq i_1<i_2<i_3 \leq k} \delta_{i_1} \delta_{i_2} \delta_{i_3} \prod_{\substack{0<l\neq i_2<i_3 \\ l\neq i_1}} 
\left(1+\delta_l  W_l^{\left(q\right)}(0)\right)  \int^{x}_{b_{i_3}}W_{i_3}^{\left(q\right)}(x-y_1)\int^{y_1}_{b_{i_2}}
W_{i_2}^{\left(q\right)'}(y_1-y_2) \times \\\nonumber   & & \times \quad \int^{y_2}_{b_{i_1}} W_{i_1}^{\left(q\right)'}(y_2-y_3)
W^{\left(q\right)'}(y_3-d)\md y_3 \md y_2 \md y_1 \\\nonumber & + &\\\nonumber &   \vdots \\\nonumber &  + &
\prod_{i=1}^k \delta_i   \int^{x}_{b_{k}}W_k^{\left(q\right)}(x-y_1)\int^{y_1}_{b_{k-1}}W_{k-1}^{\left(q\right)'}(y_1-y_2)
  \int^{y_2}_{b_{k-2}} W_{k-2}^{\left(q\right)'}(y_2-y_3) \times \ldots \\   & & \ldots \times \quad 
	\int^{y_{k-1}}_{b_{1}}W_1^{\left(q\right)'}(y_{k-1}-y_{k})W^{\left(q\right)'}(y_k-d)\,\md y_k\ldots\, \md y_1.
\end{eqnarray}
%\red{[$0<l\neq i_1<i_2$ and $\substack{0<l\neq i_2<i_3 \\ l\neq i_1}$ are probably ambiguous? Not clear if $l$ should be less than $i_2$.]}
\end{prop}

Here, we present some useful lemmas, which will be important for the proofs given in Appendices
%\red{[Appendices? Not so sure which is more correct.]}
\ref{BVcase} and \ref{Generalcase}.
%\blue{[I think I agree with Kazu and Appendices sounds better]}
\begin{lemma}\label{w_k_above_b_k_lemma}
For any  $q\geq 0$, $x\in\mathbb{R}$, $b_0=-\infty$, and $b_k = \max_{1 \leq j \leq k} b_j <y$ for $k \geq 0$,
\begin{equation}\label{w_k_above_b_k}
w_k^{(q)}(x;y)=\Xi_{\phi_k}(y) W_k^{(q)}(x-y). 
\end{equation}
\end{lemma}
\begin{proof}

%\red{[This time, the base case is missing. How about adding the following?]}
%\green{ Yes, because last time you wanted me to delete the base case... :)))}

We prove the result by induction. The base case ($k=0$) is clear, because this holds with $\Xi_{\phi_0}(y) = 1$ and $w_0^{(q)}(x;y) = W_0^{(q)}(x-y)$.

Assume for induction that $w_{k-1}^{(q)}(x;y)=\Xi_{\phi_{k-1}}(y)W_{k-1}^{(q)}(x-y)$ for $b_{k-1} < y$. Then, by \eqref{eq:mr.recursion}, for $b_{k-1} < b_k < y$,
\begin{align}\label{w_k_induction}\nonumber
w_k^{(q)}(x;y)&=w_{k-1}^{(q)}(x;y)+\delta_k\int_{b_k}^xW_k^{(q)}(x-z)w_{k-1}^{(q)\prime}(z;y)dz\\\nonumber
&=w_{k-1}^{(q)}(x;y)+\Xi_{\phi_{k-1}}(y) \delta_k\int_{b_k-y}^{x-y}W_k^{(q)}(x-y-z)W_{k-1}^{(q)\prime}(z)dz\\
&=\Xi_{\phi_{k-1}}(y)\left(W_{k-1}^{(q)}(x-y)+\delta_k  \int_{0}^{x-y}W_k^{(q)}(x-z-y)W_{k-1}^{(q)\prime}(z)dz\right).
\end{align}
According to the proofs of Theorems 4 and 5 of \cite{kyprianouloeffen2010}, we obtain that  the scale functions $W_k^{(q)}$ and $W_{k-1}^{(q)}$ are related by the following formula: for $z\geq 0$,
\begin{align}\label{useful_inedity2}
\delta_k\int_0^z W_{k}^{(q)}(z-y)W_{k-1}^{(q)}(y)dy=\int_0^z W_k^{(q)}(y)dy-\int_0^z W_{k-1}^{(q)}(y)dy.
\end{align}
Differentiating the above formula in $z$ gives us 
\begin{align*}
\delta_k\int_0^z W_k^{(q)}(y)W_{k-1}^{(q)\prime}(z-y)dy+\delta_kW^{(q)}_k(z)W_{k-1}^{(q)}(0)=W_{k}^{(q)}(z)-W_{k-1}^{(q)}(z).
\end{align*}
The above equation together with (\ref{w_k_induction}) implies that, for $b_k<y$,
\[
w_k^{(q)}(x;y)=\Xi_{\phi_k}(y)W_k^{(q)}(x-y).
\]
\end{proof}

\begin{lemma} \label{lemma_convergence_assumptotics_w}
Suppose that $q > 0$. For $\delta > \varphi_k(q)$, we have that $w^{(q)}_k (x;d) e^{-\delta x} \xrightarrow{x \uparrow \infty} 0$ for $k \geq 0$. 
\end{lemma}
\begin{proof}
%Proof by induction \blue{How about "We prove by induction"}\green{Before Kazu proposed the sentence "Proof by induction 
%instead of "We prove by induction".... -for me both are OK :)))} .
We prove the result by induction.
The base case with $k=0$ is clear, as we know that $W^{(q)}(x) e^{-\Phi(q)x}$ converges to a finite value as $x\to\infty$. Now, suppose that the claim holds for 
$k-1$ (i.e., $w^{(q)}_{k-1} (x;d) e^{-\delta x} \xrightarrow{x \uparrow \infty} 0$ for $\delta > \varphi_{k-1}(q)$).
Then we will show that it also holds for $k$. We have
\begin{align*}
 e^{-\delta x} \wq_k(x;d)&= e^{-\delta x} \wq_{k-1}(x;d)+\delta_{k} e^{-\delta x} \int^{x}_{b_{k}}W_k^{\left(q\right)}(x-y)w_{k-1}^{(q)\prime}(y;d)\md y \\
&\leq  e^{-\delta x} \wq_{k-1}(x;d)+\delta_{k} e^{-\delta x} \psi_k'(\varphi_k(q))^{-1}\int^{x}_{b_{k}} e^{\varphi_k(q)(x-y)}w_{k-1}^{(q)\prime}(y;d)\md y \\
&=  e^{-\delta x} \wq_{k-1}(x;d)+\delta_{k} e^{-(\delta - \varphi_k(q)) x} \psi_k'(\varphi_k(q))^{-1}\int^{x}_{b_{k}} e^{-\varphi_k(q)y}w_{k-1}^{(q)\prime}(y;d)\md y.
   \end{align*}
   Because $q > 0$ implies $\varphi_k (q) > \varphi_{k-1}(q)$, the inductive hypothesis and  Corollary \ref{derivative} imply that
	$\int^{x}_{b_{k}} e^{-\varphi_k(q)y}w_{k-1}^{(q)\prime}(y;d)\;\md y$ converges to a finite value as 
	$x \rightarrow \infty$. Therefore, $e^{-\delta x} \wq_k(x;d)$ vanishes in the limit as $x \rightarrow \infty$.  
	By induction, the proof is complete. 
	
\end{proof}

\begin{remark}\label{remark_convergence_assumptotics_w}
  It is easy to check that similar results as those in Lemmas \ref{Smoothness} and \ref{lemma_convergence_assumptotics_w}
  %\red{Lemmas \ref{Smoothness} and \ref{lemma_convergence_assumptotics_w}?}
  hold for the functions $z^{(q)}$ and $u^{(q)}$.
\end{remark}

%\subsection{Proofs of Theorems \ref{th:mr.recursion}, \ref{one_sided}, \ref{Resolvents} }\label{sec:proofs}

%Then using Theorem \ref{th:mr.recursion}(i) and Theorem \ref{Resolvents}(ii) we get \eqref{main_twosided_down}. 
%Finally taking appropriate limits in formula \eqref{main_twosided_up} we obtain the rest of identities.} 

%Hence
%\begin{align*}
%\e_x&\left[\int_0^{\kappa_1^{a,+}\wedge \kappa_1^{0,-}}e^{-qt}\ind_{\{U_k(t)\in B\cap[b_1,a]\}}dt\right]\\
%&=\int_{B\cap[b_1,a]}\left\{\frac{w_1^{(q)}(x)}{w_1^{(q)}(a)}(1-\delta_1W^{(q)}(0))W_1^{(q)}(a-y)-(1-\delta_1W^{(q)}(0))W_1^{(q)}(x-y)\right\}dy.
%\end{align*}
%But Theorem 6 in \cite{kyprianouloeffen2010} reads
%\begin{align*}
%\e_x\left[\int_0^{\kappa_1^{a,+}\wedge \kappa_1^{0,-}}e^{-qt}\ind_{\{U_k(t)\in B\cap[b_1,a]\}}dt\right]=\int_{B\cap[b_1,a]}\left\{\frac{w_1^{(q)}(x)}{w_1^{(q)}(a)}W_1^{(q)}(a-y)-W_1^{(q)}(x-y)\right\}dy.
%\end{align*}
%Is there something missing for the bounded variation case?
%}

\subsection{Integral equations for the multi-refracted scale functions.}\label{ss.integral}
In this section, we will show that the scale functions $w_k^{(q)}$, $z_k^{(q)}$ and $u_k^{(q)}$  are the solutions to some integral equations.
This representation will be important for defining the scale functions associated with the level-dependent L\'evy
process that we study in the next section. 
\begin{prop}\label{Volterra_type}
  Suppose that $\phi_k(x)=\sum_{i=1}^k\delta_i\ind_{\{x>b_i\}}$ and $k,q\ge0$.
  Then, for $x,d\in\mathbb{R}$, $d < b_1$ the function $w_k^{(q)}(x;d)$ is a unique solution to the following equation: %respectively,
\begin{align}\label{eq:volterra.wk}
w_k^{(q)}(x;d)&=W^{(q)}(x-d)+\int_d^xW^{(q)}(x-y)\phi_k(y)w_k^{(q)\prime}(y;d)\;dy
%&=W^{(q)}(x)+\sum_{i=1}^k\delta_k\int_{b_i}^xW^{(q)}(x-y)w_k^{(q)\prime}(y)dy,\notag
\end{align}
%and
%\begin{align}\label{eq:volterra.z}
%z_k^{(q)}(x)&=Z^{(q)}(x)+\int_0^xW^{(q)}(x-y)\phi_k(y)z_k^{(q)\prime}(y)\;dy.
%%&=Z^{(q)}(x)+\sum_{i=1}^k\delta_k\int_{b_i}^xW^{(q)}(x-y)z_k^{(q)\prime}(y)dy\notag
%\end{align}
\end{prop}

\begin{proof}
In the proof, we use the partial Laplace transform
\begin{align*}
\mathcal{L}_{b}[f(x;d)](\lambda)=\int_b^\infty e^{-\lambda x}f(x;d)\,dx,\qquad\text{$b\geq0$,}
\end{align*}
and the following property that we obtain using Fubini's theorem:
$$\mathcal{L}_{b}\left[\int_b^xg(x-y)f(y;d)\,dy\right](\lambda)=
\mathcal{L}_{b}[f(x;d)](\lambda) \mathcal{L}_0 [g(x)](\lambda).$$
Moreover, integration by parts gives us that
$$\mathcal{L}_{b}[f^{\prime}(x;d)](\lambda)=\lambda \mathcal{L}_{b}[f(x;d)](\lambda)- e^{-\lambda b} f(b;d).$$
%\green{We still do not understand this.  To be safe, we added an argument later to extend from $d=0$ to the general case. }
%\blue{Please look what I wrote below- do you think now it is OK?}
  Using the principle of induction, we will demonstrate that for $k=0,1,2,\ldots$ it holds that
  \begin{align}
w_k^{(q)}(x;d)=W^{(q)}(x-d)+\sum_{i=1}^k\delta_i \int_{b_i}^xW^{(q)}(x-y)w_k^{(q)\prime}(y;d)dy, \label{w_k_using_delta}
\end{align}
which is equivalent to \eqref{eq:volterra.wk}. 

First, the identity \eqref{w_k_using_delta} for the base case ($k=0$) holds because $w_0^{(q)}(x;d)=W^{(q)}(x-d)$, and this is the unique solution.  

 Next, under the inductive hypothesis that $w_{k-1}^{(q)}(x;d)$ is a unique solution to (\ref{w_k_using_delta}) with $k$ replaced with $k-1$, we show that $w_k^{(q)}(x;d)$ must be a unique solution to (\ref{w_k_using_delta}). To this end, given a function $u$ solving
\begin{align}\label{5}
u(x;d)=W^{(q)}(x-d)+\sum_{i=1}^k\delta_i\int_{b_i}^xW^{(q)}(x-y)u^{\prime}(y;d)\;dy,
\end{align}
we show it necessarily holds that  $u(x;d) = w_k^{(q)}(x;d)$. The proof is given separately for $x \leq b_k$ and $x > b_k$. 

Suppose that $x \leq b_k$. Then, \eqref{5} reduces to 
\begin{align*}
u(x;d)=W^{(q)}(x-d)+\sum_{i=1}^{k-1}\delta_i\int_{b_i}^xW^{(q)}(x-y)u^{\prime}(y;d)\;dy,
\end{align*}
which is precisely $w_{k-1}^{(q)}(x;d)$ by the inductive hypothesis. Because $w_{k-1}^{(q)}(x;d)$ and $w_{k}^{(q)}(x;d)$ coincide for $x \leq b_k$ 
(in view of \eqref{eq:mr.recursion}), we must have that  $u(x;d) = w_k^{(q)}(x;d)$. 

To prove the identity for $x > b_k$, we show that $\mathcal{L}_{b_k}[u(x;d)](\lambda)=\mathcal{L}_{b_k}[w_{k}^{(q)}(x;d)](\lambda)$  for sufficiently large $\lambda$.

 %Next assume by  induction that $w_{k-1}^{(q)}(x;d)$ is a unique solution to (\ref{w_k_using_delta}).
%&=W^{(q)}(x)+\sum_{i=1}^k\delta_k\int_{b_i}^xW^{(q)}(x-y)w_k^{(q)\prime}(y)dy,\notag

 %\red{[Instead of saying this previous sentence, should we give an inductive hypothesis here?]}\blue{Agreed.}

%\red{[how about moving the previous computation later (maybe before ``Hence by \eqref{4} and \eqref{4_a}"? )]}\blue{OK.}
%On the other hand  suppose $u$ solves 
%\begin{align}\label{5}
%u(x;d)=W^{(q)}(x-d)+\sum_{i=1}^k\delta_i\int_{b_i}^xW^{(q)}(x-y)u^{\prime}(y;d)\;dy.
%\end{align}
%Assume, as an inductive hypothesis, that the identity holds for $k-1$. %\blue{Yes I think so}
%Moreover note that
%Then because
% for $x\leq b_k$  \eqref{5} reduces to
%the above identity equals
%\begin{align*}
%u(x;d)=W^{(q)}(x-d)+\sum_{i=1}^{k-1}\delta_i\int_{b_i}^xW^{(q)}(x-y)u^{\prime}(y;d)\;dy,
%\end{align*}
%So by the hypothesis of induction 
%we have $u(x;d)=w_{k-1}^{(q)}(x;d)$ for $x\in (-\infty, b_k]$.
We have
\begin{align*}
u(x;d)&=W^{(q)}(x-d)+\sum_{i=1}^k\delta_i\int_{b_i}^xW^{(q)}(x-y)u^{\prime}(y;d)\;dy\notag\\
&=W^{(q)}(x-d)+\sum_{i=1}^{k-1}\delta_i\int_{b_i}^{b_k}W^{(q)}(x-y)u^{\prime}(y;d)\;dy\\
&+\sum_{i=1}^k\delta_i\int_{b_k}^xW^{(q)}(x-y)u^{\prime}(y;d)\;dy\notag\\
&=W^{(q)}(x-d)+\sum_{i=1}^{k-1}\delta_i\int_{b_i}^{b_k}W^{(q)}(x-y)w_{k-1}^{(q)\prime}(y;d)\;dy\\
&+\sum_{i=1}^k\delta_i\int_{b_k}^xW^{(q)}(x-y)u^{\prime}(y;d)\;dy\notag\\
&=w_{k-1}^{(q)}(x;d)+\sum_{i=1}^k\delta_i\int_{b_k}^xW^{(q)}(x-y)u^{\prime}(y;d)\;dy\\
&-\sum_{i=1}^{k-1}\delta_i\int_{b_k}^xW^{(q)}(x-y)w_{k-1}^{(q)\prime}(y;d)\;dy.\notag
\end{align*}
Applying the operator $\mathcal{L}_{b_k}$ to both sides of the previous equation, we obtain that
\begin{align*}
\mathcal{L}_{b_k}[u(x;d)](\lambda)&=\mathcal{L}_{b_k}[w_{k-1}^{(q)}(x;d)](\lambda)+
\frac{\sum_{i=1}^k\delta_i}{\psi(\lambda)-q}\left(\lambda\mathcal{L}_{b_k}[u(x;d)](\lambda)-e^{-\lambda b_k}u(b_k;d)\right)\notag\\
&-\frac{\sum_{i=1}^{k-1}\delta_i}{\psi(\lambda)-q}\left(\lambda\mathcal{L}_{b_k}[w_{k-1}^{(q)}(x;d)](\lambda)-e^{-\lambda b_k}w_{k-1}^{(q)}(b_k;d)\right)\notag\\
&=\frac{\psi_{k-1}(\lambda)-q}{\psi(\lambda)-q}\mathcal{L}_{b_k}[w_{k-1}^{(q)}(x;d)](\lambda)+\frac{\sum_{i=1}^k\delta_i}{\psi(\lambda)-q}
\lambda\mathcal{L}_{b_k}[u(x;d)](\lambda)\notag\\
&-\frac{\delta_k}{\psi(\lambda)-q}e^{-\lambda b_k}w_{k-1}^{(q)}(b_k;d).
\end{align*}
This implies that
\begin{align}\label{4_a}
\mathcal{L}_{b_k}[u(x;d)](\lambda)=\frac{\psi_{k-1}(\lambda)-q}{\psi_k(\lambda)-q}\mathcal{L}_{b_k}[w_{k-1}^{(q)}(x;d)](\lambda)-
\frac{\delta_k}{\psi_{k-1}(\lambda)-q}e^{-\lambda b_{k}}w_{k-1}^{(q)}(b_k;d).
\end{align}
 Applying the operator $\mathcal{L}_{b_k}$ on both sides of \eqref{eq:mr.recursion} gives
\begin{align}\label{4}
\mathcal{L}_{b_k}[w_k^{(q)}(x;d)](\lambda)&=\mathcal{L}_{b_k}[w_{k-1}^{(q)}(x;d)](\lambda)+\frac{\delta_k}{\psi_k(\lambda)-q}
\mathcal{L}_{b_k}[w_{k-1}^{(q)\prime}(x;d)](\lambda)\notag\\
&=\mathcal{L}_{b_k}[w_{k-1}^{(q)}(x;d)](\lambda)\notag\\
&+\frac{\delta_k}{\psi_k(\lambda)-q}\left(\lambda\mathcal{L}_{b_k}[w_{k-1}^{(q)}(x;d)](\lambda)-e^{-\lambda b_{k}}w_{k-1}^{(q)}(b_k;d)\right)\notag\\
&=\frac{\psi_{k-1}(\lambda)-q}{\psi_k(\lambda)-q}\mathcal{L}_{b_k}[w_{k-1}^{(q)}(x;d)](\lambda)-\frac{\delta_k}{\psi_k(\lambda)-q}e^{-\lambda b_{k}}w_{k-1}^{(q)}(b_k;d).
\end{align}
Hence, by \eqref{4_a} and \eqref{4} we obtain that $\mathcal{L}_{b_k}[u(x;d)](\lambda)=\mathcal{L}_{b_k}[w_{k}^{(q)}(x;d)](\lambda)$ for sufficiently large $\lambda$, which in turn implies that $u$ is the unique solution of
	\eqref{5} and that $u(x;d)=w_{k}^{(q)}(x;d)$ for $x\geq0$.
	\end{proof}
	%\red{[From here, it should be out of the proof?]}
	In a similar manner as in the proof of Proposition \ref{Volterra_type}, we can show that the scale functions $z_k^{(q)}$  and $u_k^{(q)}$
can be characterised as the solution to the following integral equations, respectively:
	\begin{align*}%\label{eq:volterra.z}
	z_k^{(q)}(x)&
	=Z^{(q)}(x)+\sum_{i=1}^k\delta_i\int_{b_i}^xW^{(q)}(x-y)z_k^{(q)\prime}(y)dy,\notag
	\end{align*}
	\begin{align*}%\label{eq:volterra.z}
u_k^{(q)}(x)&
	=e^{\Phi(q) x}+\sum_{i=1}^k\delta_i\int_{b_i}^xW^{(q)}(x-y)u_k^{(q)\prime}(y)dy.\notag
	\end{align*}

\section{Theory for a general premium rate function }\label{general premium}
In this section, we extend the theory of multi-refracted processes to the solutions of \eqref{sde1} with 
a general premium rate function {$\phi$}. We prove the existence and uniqueness of these solutions using the theory
already developed for the multi-refracted case. To this end, we approximate a general rate function $\phi$ 
by a sequence of rate functions $(\phi_n)_{n\geq0}$ of the form given in \eqref{rate_function-mult}. This will define a sequence of multi-refracted L\'evy processes, which will allow us to prove the existence and uniqueness 
of solutions to \eqref{sde1} by a limiting argument. 

Because most of the results presented in this section are based
on the previous results for multi-refracted processes, we adopt the following assumptions:
%\begin{enumerate}
%\item[\textbf{[A]}] The function $\phi$ is non-decreasing %non-negative,
  %%and continuously differentiable,
  %and $\phi(x)=0$ for \red{$x \leq 0$}. %$x\leq d\in {\mathbb R}$
  %In the bounded variation case, we assume that $\phi(x)<c$ for all $x\in\mathbb{R}$. Furthermore $\phi$ is either \red{locally Lipschitz continuous}  or $\phi$
   %is of the form given in  \eqref{rate_function-mult}. %\blue{[How about "or $\phi$ is of the form given in \eqref{rate_function-mult}"?]}
%\end{enumerate}
%\blue{Irmina I am sorry. The second comment of the referee critizes the fact that if $\phi(x)=0$ for $x \leq d$ using $d$ as the level in which we stop the process for the fluctuaction identities is restrictive. But do we only need this (that the support of $\phi$ is bounded from below) for defining the function $u^{(q)}$? 
%Because if this is the case	I was thinking that perhaps we can follow the referee's comment and formulate assumption [A] as
\begin{enumerate}
	\item[\textbf{[A]}]
The function $\phi$ is non-decreasing and for fixed $d' \in \mathbb{R}$, $\phi(x)=0$ for $x \leq d'$.   %$x\leq d\in {\mathbb R}$
Furthermore $\phi$ is either locally Lipschitz continuous or $\phi$ is of the form given in~\eqref{rate_function-mult}. In the bounded variation case, we assume that $\phi(x)<c$ for all $x\in\mathbb{R}$.
\end{enumerate}
%And leave everything as it is (except for the following remark) and change the definition of $u^{(q)}$ as
	%\begin{align*}
	%%\label{fun_u}
	%u^{(q)}(x)=e^{\Phi(q)x}+\int_{d'}^x  \phi(y) W^{(q)}(x-y)u^{(q)\prime}(y)\;dy.
	%\end{align*}

%\begin{remark}\label{re:phi=0}
  %%Our aim is to compute quantities which are functions of  $\kappa^{a,+}$ and $\kappa^{d,-}$ when $X(0)=x$, where $d\le x\le a$ and $d<a$.
  %%Therefore without loss of generality we can suppose that $\phi(x)=0$
  %%for $x<d$. 
	%Notice that the assumption in [A] that \red{$\phi(0)=0$} is not very restrictive. Suppose that $\phi(0)>0$, then
%\begin{eqnarray*}
%dU(t)&=&dX(t)-\phi(U(t))\,dt\\
%&&= dX(t)-\phi(0)\,dt-(\phi(U(t))-\phi(0))\,dt\\
%&&=d\tilde{X}(t)-\tilde{\phi}{(U(t))}\,dt,
%\end{eqnarray*}
%where
%$\tilde{X}$ is a L\'{e}vy process defined by $(\gamma-\phi(0),\sigma,\Pi)$
%and $\tilde{\phi}(t)=\phi(t)-\phi(0)$.
%
%The assumption that $\phi$ is \red{locally Lipschitz continuous} is used  in the proof of Proposition \ref{conv_processes}. % We believe that requirement ``continuously differentiable'' can be relaxed. 
%\end{remark}

Now, consider $U$, which is the solution of \eqref{sde2}. Between jumps, the evolution of $U$ is deterministic, according to $\dot{x}=p(x)$ with $x(0)=x_0$. 
To solve this equation, we introduce
$$\Omega_x(y)=\int_x^y\frac{1}{p(v)}\,dv.$$
Notice that $\frac{d}{dt}\Omega_{x_0}(x_t)=\dot{x}_t(p(x_t))^{-1}=1$, and so $\Omega_{x_0}(x_t)=t$. Hence, $x_t=\Omega_{x_0}^{-1}(t)$.
A problem appears if  $\Omega_0(y)<\infty$ for all $0<y\le \infty$. In this case, $\Omega^{-1}_x(t)$ is defined for all $x\ge0$, and $0<t\le \Omega_x(\infty)<\infty$.
In the paper by Albrecher  et al. \cite{Albetal}, the authors considered the case in which the process $U$ is driven by a compound Poisson process and a 
rate function $p(x)=C+x^2$. 
In view of the previous remark, we can see that the process $U$ cannot be constructed for all $t\ge0$.
\begin{comment}
We have the following cases:
\begin{itemize}
\item $\Omega_0(y)<\infty$ and $\Omega_x(\infty)=\infty$. In this case $\Omega^{-1}_x(t)$ is well defined for all $x\ge0$ and $t\ge0$.
\item  $\Omega_0(y)<\infty$ and $\Omega_x(\infty)<\infty$. In this case  $\Omega^{-1}_x(t)$ is defined for all $x\ge0$ and $0<t\le \Omega_x(\infty))$.
\item  $\Omega_0(y)=\infty$ and $\Omega_x(\infty)=\infty$. In this case  $\Omega^{-1}_x(t)$ is defined for all $x>0$ and $t\ge 0$.
\item  $\Omega_0(y)=\infty$ and $\Omega_x(\infty)<\infty$. In this case  $\Omega^{-1}_x(t)$ is defined for all $x>0$ and $t\le \Omega_x(\infty))$.
  \end{itemize}
Therefore we must assume $p(x)\le c+x^\alpha$, where $\alpha\le 1$.
\footnote{In \cite{Albetal} there is an example with $p(x)=c+x^2$, which we are not sure whether in this case it is meaningful to consider the risk process $U$.}
\end{comment}

\subsection{Existence, uniqueness, and monotonicity}
In this section we will prove the existence and uniqueness of a solution to \eqref{sde1} with a rate function $\phi$ that satisfies Assumption \textbf{[A]}. 
As we already mentioned, we will use an approximating argument, in which a monotonicity property of the solutions based on their driving rate functions 
will be crucial. %We state this property in the following result.
%\green{[I propose a comment:]}
%\green{In the following lemma and its proof we tacitly assume the existence of considered solutions of SDE's. 
%This lemma will be applied to multi-refracted processes, which fulfill the requirement. } 
%\red{I don't know. Probably easier to read if the assumption is given explicitly in the lemma?  }
%\blue{Maybe it is easier to just formulate the Lemma for step functions so we can assume that the solution exists. 
%And then after we prove the existence and uniqueness, perhaps write a Remark or Corollary stating that monotonocity still holds for the general case by similar arguments to Lemma 19?}

We  consider a sequence of functions $(\phi_n)_{n\geq1}$ that satisfy the following conditions:  
  \begin{enumerate}
   \item[(a)] $\lim_{n\to\infty}\phi_n=\phi$ uniformly on compact sets.
   \item[(b)] For $x\in\mathbb{R}$,
     \begin{eqnarray*}
             \phi_1(x)\leq \phi_2(x)\leq \ldots\leq \phi(x).
     \end{eqnarray*}
   \item[(c)] For each $n\geq 1$ and $x\in\mathbb{R}$, we have that $\phi_n(x)=\sum_{j=1}^{m_n}\delta^n_j\ind_{\{x>b_j^n\}}$, for some $m_n \in \mathbb{N}$, 
	$ 0<b_1^n<\ldots<b_{m_n}^n$, and $\delta_j^n>0$ for each $j=1,\dots,m_n$.
     \end{enumerate}
	For each $n\geq1$ we denote the solution of \eqref{sde1} with the rate function $\phi_n$ by $U_n$. 
		\begin{remark}
		According to the convention from Section \ref{sec:mult.refr}, $\phi_n$ is an $m_n$-multi-refracted rate function.
Notice that from now on $\phi_n$ is not exactly the rate function from Section \ref{sec:mult.refr}, because now we are not indexing the sequence $(\phi_n)_{n\geq1}$ only by 
the number of barriers. We decided to use the above notation for clarity of Section \ref{general premium}.
		\end{remark}
		We now show how to construct a specific sequence $(\phi_n)_{n\geq1}$ that satisfies the conditions mentioned above. For each $n\geq 1$ 
		we choose a grid $\Pi^n=\{b^n_l=l2^{-n}:l=1,\dots,m_n=n2^{n}\}$ and set
     $\delta_j^n=\phi(b^n_j)-\phi(b^n_{j-1})$, with $b_0^n=0$. Furthermore we define the {\it approximating sequence} of the rate function $\phi$ as follows:
     \[
     \phi_n(x)=\sum_{j=1}^{m_n}\delta_j^n\ind_{\{x>b_j^n\}}\qquad\text{for $n\geq 1$ and $x\in\mathbb{R}$.}
     \]
     For any $n\geq 1$, we have from Lemma \ref{lem:mr.exist} that there exists 
     a unique solution $U_n$ to \eqref{sde1} with the rate function $\phi_n$.
		Moreover, the following lemma implies that the sequence $(U_n(t))_{n\geq1}$ is non-decreasing for any $t\geq0$.

\begin{lemma}\label{lem:mon}
	Suppose that for each $n\geq 1$, $\phi_n(x)\le \phi_{n+1}(x)$ for all $x\in\mathbb{R}$.
	%and let $U_n$ be a multi-refracted L\'evy process with rate function 
	%$\phi^i$ ($i=1,2$), i.e., the unique strong solution of \eqref{sde1} with 
	%$\phi(x)=\phi^i(x)=\sum_{j=1}^{\red{k^i}}\delta_j^i\ind_{\{x>\red{b^i}_j\}}$. 
	Then $U_{n+1}(t)\le U_n(t)$ for all $t\ge 0$.
\end{lemma}

\begin{proof}
Consider $\varepsilon>0$, and define the function $\phi_{n+1}^{\varepsilon}(x):=\phi_{n+1}(x)+\varepsilon$. Then, $\phi_n(x)<\phi_{n+1}^{\varepsilon}(x)$ 
	for all $x\in\mathbb{R}$. Consider the process $U^{\varepsilon}_{n+1}$, which is a solution to the following SDE:
		\[
		U^{\varepsilon}_{n+1}(t)=X(t)-\int_0^t\phi_{n+1}^{\varepsilon}(U^{\varepsilon}_{n+1}(s))\;ds,\qquad\text{$t\geq0$.}
		\]
		%\red{Do we need to say $\varepsilon$ must be small enough so that $\phi^2 + \varepsilon$ satisfies [A]? 
		%Also, $\phi^2_{\varepsilon}(x):=\phi^2(x)+\varepsilon$ only for $x > 0$ and it is zero otherwise?} 
		%\red{Also, I still feel we cannot just assume that $U^2_\varepsilon$ exists. But it seems that if we assume that $U^2$ exists then using some 
		%arguments similar to Remark 18 above, then $U_\varepsilon^2$ should exist as well?}\blue{Yeah we need to assume that 
		%$U_{\varepsilon}$ exists for any $\varepsilon>0$ small enough...so see my previous comment.}
				Moreover, define
		\[
		\varsigma{:=}\inf\{t>0:U_n(t)<U^{\varepsilon}_{n+1}(t)\},
		\]
		and assume that $\varsigma<\infty$. We remark that because $U_n$ and $U^{\varepsilon}_{n+1}$ have the same jumps, the crossing cannot occur 
		at a jump instant. Hence, $U_n(t)-U^{\varepsilon}_{n+1}(t)$ is non-increasing in some $[\varsigma-\epsilon, \varsigma)$  for small enough $\epsilon$, and
%$U_n(\varsigma)=U^{\varepsilon}_{n+1}(\varsigma)$, and 
		\begin{align}\label{d1}
		0\geq\frac{d}{dt}\left(U_n(t)-U^{\varepsilon}_{n+1}(t)\right)\Bigg|_{t=\varsigma-}=\phi_{n+1}^\varepsilon(U^{\varepsilon}_{n+1}(\varsigma))-\phi_n(U_n(\varsigma)) >0,
		\end{align}
		which yields that 
		$\varsigma=\infty$, implying that $U^{\varepsilon}_{n+1}\leq U_n$. 
		
		Now, because $\phi_{n+1}(x)<\phi_{n+1}^{\varepsilon}(x)$ 
	for all $x\in\mathbb{R}$, using the same argument as above we obtain that $U^{\varepsilon}_{n+1}\leq U_{n+1}$.
		
\smallskip		
On the other hand, let
$$\Delta_{\varepsilon}(t):=U_{n+1}(t)-U^{\varepsilon}_{n+1}(t)=\int_0^t \Big( \phi_{n+1}^{\varepsilon}(U^{\varepsilon}_{n+1}(s))-\phi_{n+1}(U_{n+1}(s)) \Big) \;ds.$$
Then, from classical calculus 
\begin{eqnarray*}
\lefteqn{(\Delta_{\varepsilon}(t))^2=2\int_0^t  \Delta_{\varepsilon}(s)( \phi_{n+1}^{\varepsilon}(U^{\varepsilon}_{n+1}(s))-\phi_{n+1}(U_{n+1}(s)))\;ds}\\
&&= 2\int_0^t \Delta_{\varepsilon}(s)( \phi_{n+1}(U^{\varepsilon}_{n+1}(s))-\phi_{n+1}(U_{n+1}(s)))\;ds+ 2\varepsilon \int_0^t \Delta_{\varepsilon}(s) \;ds .
\end{eqnarray*}

Since we have that $U^{\varepsilon}_{n+1}(t)\leq U_{n+1}(t)$ for all $t\geq 0$ and $\phi_{n+1}$ is a non-decreasing function, we obtain that
$$\int_0^t \Delta_{\varepsilon}(s)( \phi_{n+1}(U^{\varepsilon}_{n+1}(s))-\phi_{n+1}(U_{n+1}(s)))\;ds \leq 0,$$ and finally 
	 $$(\Delta_{\varepsilon}(t))^2 \leq 2\varepsilon \int_0^t \Delta_{\varepsilon}(s) \;ds \stackrel{\varepsilon \downarrow  0}{\longrightarrow} 0. $$
	From this we conclude that 
	\[
		U_{n+1}(t)=\lim_{\varepsilon\downarrow0}U^{\varepsilon}_{n+1}(t)\leq U_n(t).
		\]
\end{proof}

The following result establishes the existence and uniqueness of a solution to the SDE \eqref{sde1} with a general premium rate function {$\phi$}, which satisfies condition \textbf{[A]}. 
In the literature, there exist general results regarding the existence of solutions of SDEs driven by L\'evy processes. See, for example, the book of Applebaum \cite{applebaum2004}. In particular, 
we wish to quote Brockwell et al.\ \cite{brockwelletal82} and references therein, where the existence and uniqueness 
of a solution to \eqref{sde1} were considered for storage processes in the case that the driving L\'evy process is of bounded variation.  
However, in this paper we will proceed in a natural manner, by demonstrating the uniform convergence on compact sets of a sequence of 
$m_n$-multi-refracted L\'{e}vy processes $(U_n)_{n\geq0}$, 
defined by the approximating sequence of rate functions $(\phi_n)_{n\geq0}$ to the solution $U$ of \eqref{sde1} with the corresponding rate function $\phi$.
%and passing to the limit with the solutions corresponding to the approximating sequence. We will show that $U_n\Rightarrow U$ uniformly on compacts to have the existance and uniqueness of solutions of \eqref{sde1}. 
%For a general version of Gronwall's inequality we refer to Ethier and Kurtz \cite{EthierKurtz}, p. 498.

For the proof of the following proposition we introduce the following notation, for any $t>0$,
\[
\overline{X}(t):=\sup_{0\leq s\leq t}X(s)\quad\text{and }\underline{X}(t):=\inf_{0\leq s\leq t}X(s).
\]

In the following proposition we assume that $\phi$ is continuously differentiable, because its statement for $\phi=\phi_k$ was already proven. 
%\blue{[Rewrote the proof of the following result removing the sublinear conditi%on.]}
\begin{prop}\label{conv_processes} Suppose that the locally Lipschitz continuous rate function  $\phi$ satisfies condition \textbf{[A]}. 
% and sublinear, ie. $|\phi(x)|\le a+bx$ for all $x>0$ and some $a,b\ge0$.% 
Then, there exists a unique solution $U$ to the SDE \eqref{sde1} with rate function $\phi$. 
Furthermore, the sequence $(U_n)_{n\geq 1}$ converges uniformly to $U$ a.s.\ on compact time intervals.
  \end{prop}

\begin{proof}
We will first show the existence, by proving the uniform convergence of the sequence $(U_n)_{n\geq1}$ on compact sets to a solution of \eqref{sde1} with rate function $\phi$.
  To this end, let $(\phi_n)_{n\geq1}$ be the approximating sequence for  $\phi$. For each $n\geq1$, we consider $U_n$ 
as a unique solution to
  $$U_n(t)=X(t)-\int_0^t\phi_n(U_n(s))\,ds.$$
  Since $\phi_1(x)\le \phi_2(x)\le \ldots \le \phi(x)$, we have by Lemma \ref{lem:mon} that $U_1(t)\ge U_2(t)\ge \ldots $. 
	Now, fix an arbitrary $T>0$. It follows that
  \begin{align*}
    U_n(t)\leq X(t)\leq |\overline{X}(T)|,\qquad\text{ $0\le t\le T$}.
  \end{align*}
  On the other hand, using the fact that $\phi(x)\geq0$ for $x\in\mathbb{R}$, we have that
  \begin{align*}
  U_n(t)=X(t) -\int_0^t\phi_n(U_n(s))ds\geq X(t)-\phi(|\overline{X}(T)|)T,\qquad\text{$0\le t\le T$.}
  \end{align*}
  Hence, 
  \begin{align*}
  |U_n(t)|\le  (|\overline{X}(T)|\vee|\underline{X}(T)|)+\phi(|\overline{X}(T)|)T=:K_T,\qquad \text{$0\le t\le T$, $n\geq 1$}.
  \end{align*}
  Since the sequence  $n \mapsto U_n$ is non-increasing and bounded below, we can define $U(t)=\lim_{n\to\infty}U_n(t)$.
  %\begin{comment}
    %Clearly $U(t)\le U_n(t)\le U_1(t)$ 
	%and by the Mean Value Theorem we have
  %\begin{eqnarray*}
    %|U_n(t)-U(t)|&\le& \int_0^t|\phi(U(s))-\phi(U_n(s))|\,ds+\int_0^t|\phi(U_n(s))-\phi_n(U_n(s)) | \,ds\\
    %&\le & T\sup_{s\in I}|\phi(s)-\phi_n(s)|+\red{L}\int_0^t|U(s)-U_n(s)|\,ds,
%\end{eqnarray*}
%where $I=[-K_T,K_T]$.
%Hence, $$\sup_{0\le t\le T}|U(t)-U_n(t)|\le T\sup_{s\in I}|\phi(s)-\phi_n(s)|e^{\sup_{s\in I}|\phi'(s)|T}.$$
%\end{comment}
By the uniform convergence of $\phi_n$ to $\phi$ on compact sets
  %Because $\sup_{s\in I}|\phi(s)-\phi_n(s)|\to 0$ \blue{[We haven't define $I$ yet. How about "By the uniform convergence of $\phi_n$ to $\phi$  on compact sets"? ]} 
	and the pointwise  convergence of $U_n$ to $U$ we have $\phi_n(U_n(t))\to \phi(U(t))$ pointwise. By the bounded convergence theorem
$$U(t)=\lim_{n\to\infty} U_n(t)=X(t)-\lim_{n\to\infty}\int_0^t\phi_n(U_n(s))\,ds
  =X(t)-\int_0^t\phi(U(s))\,ds\qquad\text{$0\le t\le T$.}$$
  
Now for the uniform convergence of $U_n$ to $U$,
by the Mean Value Theorem and the assumption that $\phi$ is locally Lipschitz, we have for $0\le t\le T$
  \begin{eqnarray*}
    |U_n(t)-U(t)|&\le& \int_0^t|\phi(U(s))-\phi(U_n(s))|\,ds+\int_0^t|\phi(U_n(s))-\phi_n(U_n(s)) | \,ds\\
    &\le & T\sup_{s\in I}|\phi(s)-\phi_n(s)|+L_I\int_0^t|U(s)-U_n(s)|\,ds,
\end{eqnarray*}
where $I := [-K_T,K_T]$ and $L_I$ is the associated Lipschitz constant for $\phi$ over the interval $I$.
Hence by Gronwall's inequality we get $$\sup_{0\le t\le T}|U(t)-U_n(t)|\le T\sup_{s\in I}|\phi(s)-\phi_n(s)|e^{L_I  T}.$$ 
This shows the uniform convergence on compact time intervals.
\begin{comment}
  he uniform convergence of $\phi_n$ to $\phi$ and $U_n$ to $U$ on compact sets, together with the continuity of $\phi$, we obtain that  
  \begin{align*}
U(t) = \lim_{n \rightarrow \infty} U_n(t)=X(t) -  \lim_{n \rightarrow \infty}\int_0^t\phi_n(U_n(s))ds  = X(t) -  \int_0^t\phi(U(s))ds.\end{align*}
\end{comment}

Now, in order to show the uniqueness, consider two solutions of \eqref{sde1}, say $U$ and $\tilde{U}$. Then, for $0\le t\le T$,
\begin{align*}
|U(t)-\tilde{U}(t)|\leq \int_0^T|\phi(U(s))-\phi(\tilde{U}(s))|ds\leq L_I\int_0^T|U(s)-\tilde{U}(s)|ds.
\end{align*}
This implies that
\begin{align*}
\sup_{t\in[0,T]}|U(t)-\tilde{U}(t)|\leq L_I\int_0^T\sup_{t\in[0,s]}|U(t)-\tilde{U}(t)|ds.
\end{align*}
Then, Gronwall's Lemma implies that
\begin{align*}
\sup_{t\in[0,T]}|U(t)-\tilde{U}(t)|=0,
\end{align*}
and therefore there exists a unique solution to \eqref{sde1}.
\end{proof}

\subsection{Theory of scale functions for level-dependent L\'evy processes.}\label{ss:integral}
%\subsubsection{An integral equation}\label{ss:integral}

  In this section, we introduce the scale function $w^{(q)}$ for the level-dependent L\'evy process 
$U$ with rate function $\phi$, as a unique solution to some integral equation. Recall that we denote the scale function of the driving L\'evy process $X$ as $W^{(q)}$.
Here, for $x,d\in\mathbb{R}$ we define $w^{(q)}(x;d)$ as the solution to the following integral equation:
  \begin{equation}\label{eq:volterra.w}
    w^{(q)}(x;d)=W^{(q)}(x-d)+\int_d^xW^{(q)}(x-y)\phi(y){w^{(q)\prime}}(y;d)\,dy,%\qquad\text{$x\geq d$,}
  \end{equation}
  provided that $ w^{(q)}$ is a.e. differentiable. For any $x\geq0$, we denote $w^{(q)}(x):=w^{(q)}(x;0)$.
	
	Now, for any $x\geq0$ define $\Xi_{\phi}(x):=1-W^{(q)}(0)\phi(x)$, which is strictly positive by Assumption \textbf{[A]}. 
  However, a more useful form for our subsequent analysis is obtained by differentiating equation \eqref{eq:volterra.w}. Thus, for $x > d$ , we obtain 
\begin{align}\label{eq:volterra.w.prime}\notag
  &w^{(q)\prime}(x;d)\\\notag&=\frac{1}{1-\phi(x)W^{(q)}(0)}{W^{(q)\prime}}((x-d)+)+\int_d^x \frac{\phi(y)}{1-\phi(x)W^{(q)}(0)} {W^{(q)\prime}}(x-y){w^{(q)\prime}}(y;d)dy\\
	&=\Xi_{\phi}(x)^{-1} W^{(q)\prime}((x-d)+)+\int_d^x \Xi_{\phi}(x)^{-1} \phi(y) W^{(q)\prime}(x-y){w^{(q)\prime}}(y;d)\;dy, 
\end{align} 
with the boundary condition $w^{(q)}(d;d) = W^{(q)}(0)$. Notice that in some cases we can have that $w^{(q)\prime}(d+;d) \geq \lim_{x\downarrow d}W^{(q)\prime}(x-d)=\infty$. 
%\red{[Is this boundary condition for the differential equation \eqref{eq:volterra.w}?]}

%\footnote{\green{ From Kuznetsov {\it et al} on page 20, formula (2.18), for $q=0$ (this argument can be extended to arbitrary $q>0$)
%we see that for $W^\prime$ we can take either of versions from (2.18). Hence we obtain the solution
%of the second kind Volterra equation $w^{\prime}$. Any version defines  scale function $w^\prime$.
%We can make a theory in a fashion of Theorem 3.6 (with $p=1$) from the book by Gripenberg et al taking 
%$J=[0,T]$ for any $T>0$. Therefore we can add a sentence as above.}}

\begin{remark}\label{derivatives}
Because in general the derivative of the function $W^{(q)}$ is not defined for all $x\in\mathbb{R}$, we take the right first derivatives of $W^{(q)}$, 
which always exist (see, for example, the proof of Lemma 2.3 in Kuznetsov et al. \cite{kuznetsovetal2012}).
We also understand $w^{(q)\prime}$ as the right hand derivative if it does not exist. Then, $w^{(q)\prime}$ is the solution of the Volterra equation \eqref{eq:volterra.w.prime},
 which together with the initial condition defines the scale function $w^{(q)}$ uniquely. This will be proved in Lemma \ref{lem:u.uprime} below.
%In the bounded variation case, notice that, in view of the assumption $0<p(x)=c-\phi(x)=c(1-W^{(q)}(0)\phi(x))=c\Xi_{\phi}(x)$, we have that $\Xi_{\phi}(x)>0$
\end{remark}

  \begin{lemma}\label{lem:u.uprime}
  Assume that a measurable function $u$ is differentiable a.e. Then, $u$ is the solution of \eqref{eq:volterra.w} if and only if $u^{\prime}$ is the solution of  \eqref{eq:volterra.w.prime}, with  the boundary condition $u(d;d)=W^{(q)}(0)$.
  \end{lemma}

\begin{proof}
Suppose that $u$ fulfills  \eqref{eq:volterra.w}. Then, differentiation of 
 \eqref{eq:volterra.w} yields  \eqref{eq:volterra.w.prime}. Conversely, suppose that $u^{\prime}$ fulfills
\eqref{eq:volterra.w.prime}. Then, defining $u(x;d)=W^{(q)}(0)+\int_d^x u^{\prime}(t;d)\,dt$, we can verify by inspection that
$u$ fulfills  \eqref{eq:volterra.w}.
\end{proof}

In a similar manner we define the scale function $z^{(q)}$ as the solution to
\begin{align}\label{eq:volterra.z.prime}\notag
  z^{(q)\prime}(x)&=\frac{q}{1-\phi(x)W^{(q)}(0)}W^{(q)}(x)+\int_0^x \frac{\phi(y)}{1-\phi(x)W^{(q)}(0)} {W^{(q)\prime}}(x-y)z^{(q)\prime}(y)dy\\
	&=\Xi_{\phi}(x)^{-1}qW^{(q)}(x)+ \int_0^x \Xi_{\phi}(x)^{-1} \phi(y) W^{(q)\prime}(x-y){z^{(q)\prime}}(y)dy 
\end{align}
with the boundary condition $z^{(q)}(0) = 1$.
%\blue{[Change to $z^{(q)}$?]}\red{Done}
%Recall that  $Z^{(q)}(0)=1$ and $Z^{(q)\prime}(x)=qW^{(q)}(x)$. Furthermore we have 
% {[Maybe remove this and put the initial conditions below the integral equation%?]}\red{This I recall is usefull somewhere else. Keep it for a while}.

 In the remainder of this section we will prove the existence and uniqueness of solutions to equations \eqref{eq:volterra.w.prime} 
and \eqref{eq:volterra.z.prime}, which belong to the family of Volterra equations. To this end, we now present an outline of the theory of such equations.
For $x\geq d$, a Volterra equation is given by
  \begin{equation}\label{eq:volt.eq}u(x;d)=g(x;d)+\int_d^x K(x,y)u(y;d)\,dy,\end{equation}
  where $K$, $g$, and  $u$ are measurable, $g$ is non-negative and the integrals are well defined. We set 
	\begin{equation}\label{function_g_W}
	g(x;d)=W^{(q)\prime}((x-d)+)/\Xi_{\phi}(x)  
	\end{equation}
	 and 
	\begin{equation}\label{function_g_Z}
	g(x)=q  W^{(q)}(x)/\Xi_{\phi}(x)
	\end{equation}
	to obtain ${w^{(q)\prime}}(x;d)$
and $z^{(q)\prime}(x)$, respectively.

The idea of how to solve Volterra {equations} \eqref{eq:volt.eq} is as follows.
For $T>d$, we {consider} a kernel $K:D=\{(x,y): d\le y < x \le T\}\to\reals_+ $
(that is $K(x,y)=0$ for $y>x$), which is assumed to be measurable, 
and define the following operator for a nonnegative measurable function $f$: 
$$K\diamond f(x;d)=\int_d^x K(x,y)f(y;d)\,dy,\qquad d\le x\le T.$$
%\blue{[Maybe some integrability assumption on $f$ is needed?]}\red{$K\ge0$ and %$f\ge0$.}
Now, consider the following Pickard iteration: We set 
$u_1=g$  and
\begin{equation}\label{eq:pickard}
  u_{n+1}=g+K\diamond u_n\qquad \mbox{for $n\geq 1$}.\end{equation}
By induction, one can prove that
\begin{equation}\label{pickard}
u_{n+1}(x;d)=g(x;d)+\int_d^x\sum_{l=1}^nK^{(l)}(x,y)g(y;d)\,dy,\qquad\text{for $n\geq 1$, and $x\in[d,T]$,}
\end{equation}
where  $K^{(1)}=K$ and $K^{(l+1)}(x,y)=\int_y^xK^{(l)}(x,w)K(w,y)\,dw$.
 Notice that the sequence 
$\left(\sum_{l=1}^j K^{(l)}(x,y)\right)_{j\geq 1}$ is non-decreasing, and it is convergent if there exists a majorant function $\zeta(x,y)<\infty$ such that
\[\sum_{l=1}^j K^{(l)}(x,y)\le \zeta(x,y)\qquad\text{ for all $j\geq 1$ and $(x,y)\in D$}.
\]
 We now  pass with $n\to\infty$ in \eqref{pickard}. We can enter with the limit under the sum provided
\begin{equation}\label{zetaintegrable}\int_d^T\zeta(x,y)g(y;d)\,dy<\infty.
\end{equation}
Thus there exists a sequence $(u_n)_{n\geq1}$ converging to the function $u$ defined by
\begin{equation}\label{volterra}
u(x;d)=g(x;d)+\int_d^xK^{\dagger}(x,y)g(y;d)\,dy,\qquad\text{for $x\in[d,T]$}
\end{equation}
where
  $K^{\dagger}(x,y)=\sum_{l=1}^\infty K^{(l)}(x,y)$.
Now, passing in \eqref{eq:pickard} with $n\to\infty$ on both  sides
  (notice that $(u_n)$ is a non-decreasing sequence) we see that $u$ is a solution of \eqref{eq:volt.eq}.

Furthermore, suppose there exist two solutions $u_1$ and $u_2$ to \eqref{eq:volt.eq}. Then $u=u_1-u_2$
is a solution to $u(x)= \int_d^xK(x,y)u(y)\,dy$.
Now under the assumption that the kernel $K$ is positive and integrable, Gronwall's Lemma yields $u\equiv 0$, which shows the uniqueness.
We remark that in {the} theory of Volterra equations, the continuity of $g$ and $K$ is not required.
The presented theory is in the spirit of $L^1$ kernels as in Chapter 9.2 of
\cite{gripenberg}.

%\blue{[I think the following result can be interesting, but perhaps its better to put it somewhere else?]}
%In the following lemma we denote by $W_{\phi(q)}(x)$, the scale function for $(X,\p^{\Phi(q)})$; for details see \cite{kuznetsovetal2012} or \cite{kyprianou2006}.
%\begin{prop}\label{prop:wq}
%{We have that} $w^{(q)}(x)=e^{\Phi(q)x}w_{\Phi(q)}(x)$ {for $x\geq 0$},
%where $w_{\Phi(q)}(x)$ is the solution of
 %$$w_{\Phi(q)}(x)=W_{\phi(q)}(x)+\int_0^x W_{\phi(q)}(x-y)\phi(y)w_{\Phi(q)}^{'}(y)\,dy.$$
%\end{prop}
%
%\begin{proof}
%It is clear that
 %$$ e^{\Phi(q)x} w_{\Phi(q)}(x)=e^{\Phi(q)x}W_{\phi(q)}(x)+\int_0^xe^{\Phi(q)(x-y)} W_{\phi(q)}(x-y)\phi(y) e^{\Phi(q)y} w_{\Phi(q)}^{'}(y)\,dy.$$
 %{So using the fact that $W^{(q)}(x)=e^{\Phi(q)x}W_{\Phi(q)}(x)$ for $x\geq 0$ we have the result.}
%\end{proof}

\begin{comment}
In this case we can make easily the following observation.
Suppose that $-\infty<\inf_{x}\phi(x)$  and  denote $\delta=-\inf_{x}\phi(x)+\epsilon$.
We can rewrite the sde \eqref{sde1} as follows
\begin{eqnarray*}
dU(t)&=&dX(t)-\phi(U(t))\,dt\\
&=&dX(t)+\delta\,dt-(\delta-\phi(U(t)))\,dt\\
dU(t)&=&dX(t)-\tilde{\phi}(U(t))\,dt,
\end{eqnarray*}
where $\tilde{\phi}(x)= \delta-\phi(U(t)) >0$.
Thus without loss of generality we can suppose from now on that $\phi\ge0$.
\end{comment}

 % \subsubsection{Solutions of Volterra equations for $w$ and $z$ functions.}\label{ss:convergence.sf}
 For a given scale function $W^{(q)}$ and rate function $\phi$, we now consider equations
 \eqref{eq:volterra.w.prime} and  \eqref{eq:volterra.z.prime}, which are Volterra equations of the type \eqref{eq:volt.eq}.
For both equations,
% \eqref{eq:volterra.w} and \eqref{eq:volterra.z.prime}, 
the kernel is given by
%$$K(x,y)=\frac{\phi(y)}{1-\phi(x)W^{(q)}(0)}  {W^{(q)\prime}}(x-y). $$
%\begin{description}
%\item[(i)] If $W^{(q)}(0)=0$, then
  %$K(x,y)={W^{(q)\prime}}(x-y)\phi(y)$.
   %\item[(ii)] if  $W^{(q)}(0)>0$, then
  %$K(x,y)=\frac{\phi(y)}{1-\phi(x)W^{(q)}(0)}  {W^{(q)\prime}}(x-y)$.
%\end{description}
\[
K(x,y)= \Xi_{\phi}(x)^{-1}\phi(y) W^{(q)\prime}((x-y)+).
\]
On the other hand for the function $g$ in \eqref{eq:volt.eq} we use \eqref{function_g_W} and \eqref{function_g_Z} 
to obtain the functions $w^{(q)}$ and $z^{(q)}$ respectively.

\begin{comment}
  We review now results concerning the smoothness of scale functions (see eg. \cite{kuznetsovetal2012, kyprianou2006}).
  \begin{itemize}
    \item For all $q\ge0$, $W^{(q)}(0)=0 $ if and only if $X$ has unbounded variation. \item Otherwise $W^{q)}(0)=1/c$, where $c$ is the drift.
  \item For all $q\ge0$, $W^{(q)}$ is continuous, a.e. differentiable and strictly increasing.
  \item For each $q\ge0$, $W^{(q)}\in {\mathcal C}^{1}(0,\infty)$ if and only if one of the following holds: (i) $\sigma\neq0$, (ii) $\int_0^1 x\Pi(dx)=\infty$,
    (iii) $\Pi(x)=\Pi(x,\infty)$ is continuous.
    \item If $X$ has unbounded variation, the ${W^{(q)\prime}}(0)=2/\sigma^2$. \item If $X$ has bounded variation, then ${W^{(q)\prime}}(0)=\frac{\Pi(0,\infty)+q}{c^2}$.
    \end{itemize}
Furthermore $Z^{(q)}(0)=1$ and $Z^{(q)^{\prime}}(0)=qW^{(q)}(0)$.
\end{comment}

%\par Let $w^{(q)}=w^{(q)}_{\phi}$ and $z^{(q)}=z^{(q)}_{\phi}$ be scale functions corresponding to $U$ defined by rate function $\phi$
%\green{do we need this notation?}.
%\blue{[We haven't defined $w^{(q)}_{\phi}$ or $z^{(q)}_{\phi}$.]}
%In the sequel, we sometimes skip subscript $\phi$.

\bigskip

Before we determine a majorant $\zeta(x,y)$ for the Neumann series $K^{\dagger}$
we require two preparatory lemmas. The proof of the first is obvious. For the proof of the second, we refer the reader to Appendix \ref{AppD}.

\begin{lemma}\label{lem:prep1}
Assume that $f$ is a non-negative measurable function. Suppose that for
some $s_0{>0}$, we have that $\int_0^\infty e^{-s_0x}f(x)\,dx<\infty$. Then, $f$ is finite a.e.
\end{lemma}
%\blue{I am not sure if this Lemma is useful? Remove it?}
\begin{lemma}\label{lem:zeta.finite}
Let $a>0$. The function 
\begin{equation}\label{eq:zeta}
\zeta(x)=\sum_{l=1}^\infty a^l{{(W^{(q)\prime}}})^{*l}(x+)
\end{equation}
is finite for all $x>0$, where ${{(W^{(q)\prime}}})^{*l}$ is the $l$-th convolution power.
\end{lemma}

Because $\phi(x)<c = 1/W^{(q)}(0)$ for all $x >0$ in the bounded variation case, and $W^{(q)}(0)=0$ in the unbounded variation case, for any $T>0$ we have that
\begin{equation}\label{eq:def_a} a_T:=\sup_{d\le y\le {T}}\left|\Xi_{\phi}(T)^{-1}\phi(y)\right|<\infty.
\end{equation}
%\blue{[I guess we can take $a$ depending on $x$ and formulate the following results on bounded sets, to avoid the bounded assumption on $\phi$?]}
For clarity, we defer the proof of the following lemma and proposition to Appendix \ref{AppD}, because the arguments are of a technical nature.
\begin{lemma}\label{lem:i-iii}
For any $T>d$, we have 
\[
K^{\dagger}(x,y)\le \zeta(x-y),\qquad\text{for any $d \leq y\leq x\leq T$,}
\]
where
\[
\zeta(x)=\sum_{l=1}^\infty a_T^l{{(W^{(q)\prime}}})^{*l}(x+)<\infty,\qquad\text{$x > 0$.}
\]
\end{lemma}
The next result proves the existence and uniqueness of solutions to \eqref{eq:volterra.w.prime} and \eqref{eq:volterra.z.prime}, 
and also provides expressions for the functions $w^{(q)}(\cdot;d)$ and $z^{(q)}$ in terms of the scale function $W^{(q)}$ 
of the driving L\'evy process $X$. Because the scale functions are solutions of the Volterra equation one might try to solve them 
numerically (see, e.g., Remark VIII.1.10 of \cite{asmussen_albrecher_2010}).

\begin{prop}\label{prop:solutions} 
For all $T>0$, the following holds. %\green{[Do we need this bound by $T$? Because it has been chosen arbitrarily?]}\blue{[Agreed. Let us delete $T$ from the statement of the Proposition]}
\begin{itemize}
\item[(i)] For all $d<x \leq T $,  we have that
%Under the assumptions of Lemma \ref{lem:i-iii} \blue{[Which assumptions?]}
 $$\int_d^xK^{\dagger}(x,y)\Xi_{\phi}(x)^{-1}W^{(q)\prime}(y-d)\,dy<\infty,$$ 
and hence 
\begin{equation}
\label{sf_explicit}
w^{(q)\prime}(x;d)=\Xi_{\phi}(x)^{-1}W^{(q)\prime}((x-d)+)+\int_d^x K^{\dagger}(x,y)\Xi_{\phi}(x)^{-1}W^{(q)\prime}(y-d)\,dy
\end{equation}
is the unique solution to \eqref{eq:volterra.w.prime}.
\item[(ii)] For $0 \leq x \leq T $, 
 $$\int_0^xK^{\dagger}(x,y)\Xi_{\phi}(x)^{-1}qW^{(q)}(y)\,dy<\infty,$$ 
 and hence 
\[
z^{(q)\prime}(x)= \Xi_{\phi}(x)^{-1} qW^{(q)}(x)+\int_0^x K^{\dagger}(x,y)\Xi_{\phi}(x)^{-1}qW^{(q)}(y)\,dy
\]
is the unique solution to \eqref{eq:volterra.z.prime}.
\end{itemize}
\end{prop}

\medskip
Consider now an approximating sequence $(\phi_n)_{n\geq 1}$ for $\phi$. Then, by definition we have that for any $x\in\mathbb{R}$ the sequence $(\phi_n(x))_{n\geq1}$ is non-decreasing, which implies that the sequence
$(\Xi_{\phi_n}(x)^{-1})_{n\geq1}$ is also non-decreasing. Therefore, if we define
\[
K_n(x,y):=\Xi_{\phi_n}(x)^{-1}\phi_n(y)W^{(q)\prime}((x-y)+),\qquad\text{for $n\geq 1$ and $x> y> d$,}
\] 
then the sequence $(K_n(x,y))_{n\geq1}$ is non-decreasing for any $d< y< x$. 
%\blue{I think we have to uniformize notation, and always denote the Kernel as $%K_{\phi}$.}\red{Agree}
From Proposition \ref{prop:solutions} we have that, for each $n\geq1$, the scale functions $w^{(q)}_n(\cdot;d)$ and $z^{(q)}_n$ associated with the level-dependent L\'evy process $U_n$ 
with rate function $\phi_n$ satisfy the following equations: 
\begin{align}\label{approx_sf}
 w_n^{(q)\prime}(x;d)= \Xi_{\phi_n}(x)^{-1}  W^{(q)\prime}((x-d)+)+\int_d^xK^{\dagger}_{n}(x,y) \Xi_{\phi_n}(x)^{-1}{W^{(q)\prime}}(y-d) \,dy
 \end{align}
and
\begin{align*}
 z_n^{(q)\prime}(x)=\Xi_{\phi_n}(x)^{-1} qW^{(q)}(x)+\int_0^xK^{\dagger}_{n}(x,y) \Xi_{\phi_n}(x)^{-1} qW^{(q)}(y)\,dy,
\end{align*}
where, for each $d\leq y\leq x$, $K^{\dagger}_{n}(x,y):=\sum_{l=1}^{\infty} K^{(l)}_n(x,y)$.

\begin{theorem}\label{conv_sf} For any $x > d$, we have
\begin{align*}
\lim_{n\to\infty}w_n^{(q)\prime}(x;d)=w^{(q)\prime}(x;d)\qquad\text{and}\qquad 
\lim_{n\to\infty}z_n^{(q)\prime}(x)=z^{(q)\prime}(x),
\end{align*}
where the functions $w^{(q)\prime}(\cdot;d)$ and $z^{(q)\prime}$ are the unique solutions to equations \eqref{eq:volterra.w.prime} and \eqref{eq:volterra.z.prime}, respectively.
\end{theorem}

\begin{proof} We prove the result for the function $w^{(q)}$.
The case for the function $z^{(q)}$ can be treated similarly. 
We now show by induction that, for $l \geq 1$,
	\begin{equation}\label{lim_ker_l}
	\lim_{n\to\infty}K_{n}^{(l)}(x,y)= K^{(l)}(x,y),\qquad\text{for $d<y<x$.}
	\end{equation}
Using the fact that $(\phi_n)_{n\geq1}$ is an approximating sequence for $\phi$, we obtain for any $x> y> d$  that
\[
\lim_{n\to\infty}K_n(x,y)=\lim_{n\to\infty} \Xi_{\phi_n}(x)^{-1} \phi_n(y)W^{(q)\prime}((x-y)+)=K(x,y).
\]
Hence, the result follows for the case that $l=1$. Assuming that the result holds true for $l\geq 1$,
we note that
\begin{equation*}
K_{n}^{(l+1)}(x,y)=\int_y^xK_{n}^{(l)}(x,z)\Xi_{\phi_n}(z)^{-1} \phi_n(y){W^{(q)\prime}}(z-y)\,dz,\qquad\text{{for $d<y<x$.}}
\end{equation*}
By assumption, we have that
\begin{equation*}
\lim_{n\to\infty}K_{n}^{(l)}(x,z)\Xi_{\phi_n}(z)^{-1} \phi_n(y)= K^{(l)}(x,z)\Xi_{\phi}(z)^{-1} \phi(y),\qquad\text{{for $d<y<z<x$.}}
\end{equation*}
On the other hand, by the monotonicity of the sequence $(\phi_n)_{n\geq1}$ we obtain for all $n\geq 1$
\begin{align*}
a_n:=\sup_{d\leq y\leq x}\Xi_{\phi_n}(x)^{-1} \phi_n(y)\leq \sup_{d\leq y\leq x}\Xi_{\phi}(x)^{-1} \phi(y)=:a.
\end{align*}
Hence, by  identity \eqref{eq:iii} in the proof of Lemma \ref{lem:i-iii} (see Appendix \ref{AppD}) we have that
\begin{equation}\label{kern_bound}
K_{n}^{(l)}(x,z)\le a^l(W^{(q)\prime})^{*l}((x-z)+),\qquad\text{for $0<z<x$, $n\geq 1$.}
\end{equation}
Therefore, we can deduce by dominated convergence that
\begin{equation*}
\lim_{n\to\infty}K_{n}^{(l+1)}(x,z)= K^{(l+1)}(x,z),\qquad\text{ for $d<z<x$.}
\end{equation*}
Therefore, proceeding by induction we obtain \eqref{lim_ker_l}.	
\par Now, using \eqref{kern_bound} and Lemma \ref{lem:i-iii}, we can use dominated convergence to conclude that for any $d<x<y$,
\begin{equation*}
\lim_{n\to\infty}K_{n}^{\dagger}(x,y)=K^{\dagger}(x,y).
\end{equation*}
We now observe that for $x>d$ it holds that %\red{[below $\phi$ should be $\phi_n?$]}
\begin{equation*}
\int_d^x K_{n}^{\dagger}(x,y)\Xi_{\phi_n}(x)^{-1} W^{(q)\prime}(y-d)\,dy\le a 
\int_d^x \zeta(x-y) W^{(q)\prime}(y-d)\,dy<\infty,\qquad\text{for every $n\geq 1$}.
\end{equation*}
Hence, by the dominated convergence theorem we obtain 
%\green{Since here we have intergal I wrote derivative instead of right-derivative, is it ok?}
\begin{equation*}
\lim_{n\to\infty}\int_d^xK_{n}^{\dagger}(x,y)\Xi_{\phi_n}(x)^{-1} W^{(q)\prime}(y-d)\,dy=
        \int_d^xK^{\dagger}(x,y)\Xi_{\phi}(x)^{-1} W^{(q)\prime}(y-d)\,dy,\qquad\text{for $x\geq d$}.
\end{equation*}
Finally, using \eqref{approx_sf} we obtain for $x>d$ that 
\begin{align*}
\lim_{n\to\infty}w_n^{(q)\prime}&(x;d)=\lim_{n\to\infty}\left[\Xi_{\phi_n}(x)^{-1} W^{(q)\prime}((x-d)+)+\int_d^x{K^{\dagger}_n}(x,y)
\Xi_{\phi_n}(x)^{-1} {W^{(q)\prime}}(y-d)\,dy\right]\\
&=\Xi_{\phi}(x)^{-1} W^{(q)\prime}((x-d)+)+\int_d^x{K^{\dagger}} (x,y) \Xi_{\phi}(x)^{-1} W^{(q)\prime}(y-d)\,dy=w^{(q)\prime}(x;d).
\end{align*}
\end{proof}
The next result will introduce new functions that will be used in the next section. 
Consider now the  scale function $w^{(q)}$ associated with the level-dependent L\'evy process $U$ with rate function $\phi$, 
defined in \eqref{sf_explicit}. Notice that for $\phi=\phi_k$ we write $u=u_k$.

\begin{lemma}\label{fun_u,v}
	For any $x\in\mathbb{R}$, if we denote $u^{(q)}(x):=\lim_{d\to-\infty}\frac{w^{(q)}(x;d)}{W^{(q)}(-d)}$, then $u^{(q)}$ is the solution to the following integral equation
	\begin{align}
	\label{fun_u}
	u^{(q)}(x)=e^{\Phi(q)x}+\int_{d'}^x  \phi(y) W^{(q)}(x-y)u^{(q)\prime}(y)\;dy.
	\end{align} 
\end{lemma}
\begin{proof} 
	First, we note by \eqref{sf_explicit} that
	\begin{align*}
	\frac{w^{(q)\prime}(x;d)}{W^{(q)}(-d)}=\Xi_{\phi}(x)^{-1} \frac{W^{(q)\prime}(x-d)}{W^{(q)}(-d)}+\int_d^x K^{\dagger}(x,y) \Xi_{\phi}(x)^{-1}\frac{W^{(q)\prime}(y-d)}{W^{(q)}(-d)}\,dy.
	\end{align*}
	Hence, using Exercise 8.5 in \cite{kyprianou2014}, the fact that $K^{\dagger}(x,y)=0$ for $y<d'$, and dominated convergence, we have that
	\begin{align*}
	\lim_{d\to-\infty}\frac{w^{(q)\prime}(x;d)}{W^{(q)}(-d)}=\Phi(q)\Xi_{\phi}(x)^{-1} 
	e^{\Phi(q)x}+\int_{d'}^x \Phi(q) \Xi_{\phi}(x)^{-1} K^{\dagger}(x,y)e^{\Phi(q)y}\,dy.
	\end{align*}
	%\green{[Sorry, I do not understand how to obtain the following.  Are we using (41)?  Or the equation immediately above?]}\blue{[We are using the equation immediately before and the fact that $\frac{W^{(q)\prime}(x-d)}{W^{(q)}(-d)}=\Phi(q)e^{\Phi(q)x}$ I think.]}
The previous identity implies that $u^{(q)\prime}$ is the unique solution to
\begin{align*}
u^{(q)\prime}(x)=\Xi_{\phi}(x)^{-1} \Phi(q)e^{\Phi(q)x}+\int_{d'}^x \Xi_{\phi}(x)^{-1} \phi(y) W^{(q)\prime}(x-y)u^{(q)\prime}(y)\;dy,\qquad\text{$x\in\mathbb{R}$.}
\end{align*}
Hence by defining $u^{(q)}(x)=1+\int_{d'}^xu^{(q)\prime}(y)dy$, and proceeding like in Lemma \ref{lem:u.uprime} we obtain that $u^{(q)}$ is the unique solution to
\begin{align*}
u^{(q)}(x)=e^{\Phi(q)x}+\int_{d'}^x  \phi(y) W^{(q)}(x-y)u^{(q)\prime}(y)\;dy,\qquad\text{$x\in\mathbb{R}$.}
\end{align*}
On the other hand we have using \eqref{eq:volterra.w}
\begin{equation*}
\lim_{d\to-\infty}\frac{w^{(q)}(x;d)}{W^{(q)}(-d)}=e^{\Phi(q)x}+\int_{d'}^x  \phi(y) W^{(q)}(x-y)u^{(q)\prime}(y)\;dy,\qquad\text{$x\in\mathbb{R}$.}
\end{equation*}
Hence by the uniqueness of the solution to \eqref{fun_u} we have that $u^{(q)}(x)=\lim_{d\to-\infty}\frac{w^{(q)}(x;d)}{W^{(q)}(-d)}$.
\end{proof}

\subsection{Fluctuation identities for level-dependent L\'evy processes}
We  recall  that $w^{(q)}$ is a {\it scale function} if it fulfills
equation \eqref{eq:volterra.w}. That is, if it satisfies
\begin{equation}
\label{sf_gd_1}
w^{(q)}(x;d)=W^{(q)}(x-d)+\int_d^xW^{(q)}(x-y)\phi(y)w^{(q)\prime}(y;d)\,dy.
\end{equation}
A similar definition is given for the scale function $z^{(q)}$, which is the solution
of 
\begin{equation}\label{sf_gd_2}
z^{(q)}(x)=Z^{(q)}(x)+\int_0^xW^{(q)}(x-y)\phi(y)z^{(q)\prime}(y)\,dy.
\end{equation}

This definition can be justified as follows.
Consider now a level-dependent L\'evy process $U$ with rate function $\phi$, and an approximating sequence $(\phi_n)_{n\geq1}$ for $\phi$. 
For each $n\geq 1$, we consider the associated multi-refracted L\'evy process $U_n$ with rate function $\phi_n$. 
By Proposition \ref{conv_processes}, we have the convergence of the sequence $(U_n)_{n\geq1}$ to the process $U$ uniformly on compact time intervals.

On the other hand, we showed in Section \ref{sec:mult.refr} that with the use of the $w^{(q)}_n$ and $z^{(q)}_n$ scale functions 
we can compute important fluctuation identities for the process $U_n$ for each $n\geq1$.
Finally, by Theorem \ref{conv_sf} we found that the sequences of scale functions $(w^{(q)}_n)_{n\geq0}$ and $(z^{(q)}_n)_{n\geq0}$ converge to the corresponding scale functions 
  $w^{(q)}$ and $z^{(q)}$ of the process $U$, respectively.

These facts imply that we can obtain fluctuation identities for the process $U$ as the limits of 
the respective identities for the sequence of multi-refracted L\'evy processes $(U_n)_{n\geq1}$, 
and hence these will be given in terms of the scale functions $w^{(q)}$ and $z^{(q)}$. We will first prove a preliminary result, which follows verbatim from
\cite{kyprianouloeffen2010}. 
\begin{lemma}\label{p=0}
Let $\overline{U}(t):=\sup_{0\leq s\leq t}U(s)$. For each given $x,a\in\mathbb{R}$, 
the level-dependent L\'evy process $U$ with rate function $\phi$ satisfies 
$\mathbb{P}_x(\overline{U}(t) = a) = 0$ for Lebesgue almost every $t>0$. 
\end{lemma}

Let $ a \in \mathbb{R}$ and define the following first-passage stopping times for the level-dependent process:
\begin{align*}
\kappa^{a,-} & := \inf\{t>0 \colon U(t)<a\} \quad \textrm{and} \quad \kappa^{a,+} := \inf\{t>0 \colon U(t)\geq a\}.
\end{align*}

In the following theorem we derive formulas for resolvents. Notice that \eqref{resolvent2_g} for $\phi=\phi_k$ is consistent with \eqref{resolvent2}, however we do not have a more explicit formula for $c^{(q)}(y;d)$ in a general case.
\begin{theorem}{\textbf{(Resolvents)}}\label{Resolvents_sd}\\
	Fix a Borel set $\mathcal{B} \subseteq \mathbb{R}$, then
	\begin{itemize}
		\item[(i)] For  $q\geq 0$ and $d\leq x \leq a$,
		\begin{align}\label{resolvent1_g}
		\e_x\left[\int_0^{\kappa^{a,+}\wedge \kappa^{d,-}}e^{-qt}\ind_{\{U(t)\in 
\mathcal{B}\}}\;dt\right]
		= \int_{\mathcal{B} \cap (d,a)} \Xi_{\phi}(y)^{-1} \left(\frac{w^{(q)}(x;d)}{w^{(q)}(a;d)}w^{(q)}(a;y)-w^{(q)}(x;y)\right)\; dy.
		\end{align}
	      \item[(ii)]  For $q > 0$ and $x \geq 0$, there exists $c^{(q)}(y;d)>0$ such that
                $w^{(q)}(x;d) c^{(q)}(y;d)-w^{(q)}(x;y)\ge 0$ and
		\begin{align}\label{resolvent2_g}
		\e_x\left[\int_0^{\kappa^{d,-}}e^{-qt}\ind_{\{U(t)\in \mathcal{B}\}}dt\right]
		=\int_{\mathcal{B} \cap (0,\infty)}\Xi_{\phi}(y)^{-1}\left(c^{(q)}(y;d)w^{(q)}(x;d)-w^{(q)}(x;y)\right)dy.
		\end{align}
		\item[(iii)] For  $q \geq 0$ and $ x\leq a$,
		\begin{align}\label{resolvent3_g}
		\e_x\left[\int_0^{\kappa^{a,+}}e^{-qt}\ind_{\{U(t)\in \mathcal{B} \}}dt\right]
		= \int_{\mathcal{B} \cap (-\infty,a)} \Xi_{\phi}(y)^{-1} \left(\frac{u^{(q)}(x)}{u^{(q)}(a)}w^{(q)}(a;y)-w^{(q)}(x;y)\right)dy,
		\end{align}
		where $u^{(q)}$ is given by \eqref{fun_u}.
	%	\item[(iv)] For  $q> 0$ and $ x \in \mathbb{R}, $
	%	\begin{align}\label{resolvent4_g}
	%	\e_x\left[\int_0^{\infty}e^{-qt}\ind_{\{U(t)\in \mathcal{B}\}}d\right]
	%	= \int_{\mathcal{B}} \Xi_{\phi}(y)^{-1}\left(\frac{u^{(q)}(x) v^{(q)}(y)}{A(q)} -w^{(q)}(x;y)\right)dy,
	%	\end{align} 
		%where $A(q):=\psi^{\prime}(\Phi(q))+\int_0^{\infty}
                %\frac{e^{-\Phi(q)y}\phi(y)}{\Xi_{\phi}(y)}\left(\Phi(q)e^{\Phi(q)y}+\int_0^y \Phi(q)K^{\dagger}(y,z)e^{\Phi(q)z}\,dz\right)dy.$
	\end{itemize}
\end{theorem}
\begin{proof}
(i) Consider an approximating sequence $(\phi_n)_{n\geq1}$ for the rate function $\phi$ of the level-dependent L\'evy process $U$. If we denote by $(U_n)_{n\geq 0}$ the corresponding sequence of non-increasing 
multi-refracted L\'evy processes, then by Proposition \ref{conv_processes}, we know that the sequence converges uniformly on compact sets to $U$. Because
\[
|\overline{U}_n(t)-\overline{U}(t)|\vee|\underline{U}_n(t)-\underline{U}(t)|\leq \sup_{s\in[0,t]}|U_n(s)-U(s)|, \qquad\text{$t\geq0$},
\]
we have for $t>0$ that 
\[
\lim_{n\uparrow\infty}(U_n(t),\overline{U}_n(t),\underline{U}_n(t))=(U(t),\overline{U}(t),\underline{U}(t)),
\]
where $\underline{U}(t)=\inf_{0\leq s\leq t}U(s)$ and $\underline{U}_n(t)=\inf_{0\leq s\leq t}U_n(s)$.
Now, using the fact that for each $t> 0$ the sequence $(U_n(t))_{n\geq1}$ is non-increasing, we have for $a,y\geq0$
\begin{align*}
\{\underline{U}(t)\geq d \}&=\bigcap_{n\geq 1}\{\underline{U}_n(t)\geq d\},\\
\{\overline{U}(t)\geq a\}&=\bigcap_{n\geq 1}\{\overline{U}_n(t)\geq a\},\\
\{U(t)\geq y\}&=\bigcap_{n\geq 1}\{U_n(t)\geq y\}.
\end{align*}
This means that for any $x\in\mathbb{R}$ and $t>0$, we have
\begin{align*}
\mathbb{P}_x\left(U(t)\geq y, \overline{U}(t)\geq a, \underline{U}(t)\geq d\right)
&=\mathbb{P}_x\left(\bigcap_{n\geq 1}\{U_n(t)\geq y, \overline{U}_n(t)\geq a, \underline{U}_n(t)\geq d\}\right)\\
&=\lim_{n\to\infty}\mathbb{P}_x\left(U_n(t)\geq y, \overline{U}_n(t)\geq a, \underline{U}_n(t)\geq d\right).
\end{align*}
By Lemma \ref{p=0} we have that $\mathbb{P}_x(\overline{U}(t)=a)=0$ for any $x\in\mathbb{R}$ and
for Lebesgue almost every $t>0$. This in turn implies that
\begin{align*}
\mathbb{P}_x\left(U(t)\geq y, \overline{U}(t)\leq a, \underline{U}(t)\geq d \right)&=
\mathbb{P}_x\left(U(t)\geq y,\underline{U}(t)\geq d\right)\\ &-\mathbb{P}_x\left(U(t)\geq y, \overline{U}(t)\geq a, \underline{U}(t)\geq d\right)\\
&=\lim_{n\to\infty}\mathbb{P}_x\left(U_n(t)\geq y,\underline{U}_n(t)\geq d\right)\\ &
-\lim_{n\to\infty} \mathbb{P}_x\left(U_n(t)\geq y, \overline{U}_n(t)\geq a, \underline{U}_n(t)\geq d\right)\\
&=\lim_{n\to\infty}\mathbb{P}_x\left(U_n(t)\geq y, \overline{U}_n(t)\leq a, \underline{U}_n(t)\geq d\right).
\end{align*}
Thus, we obtain by bounded convergence that
\begin{align*}
\lim_{n\to\infty}\mathbb{E}_x\left[\int_0^{\kappa_n^{a,+}\wedge\kappa_n^{d,-}}e^{-qt}\ind_{\{U_n(t)\geq y\}}dy\right]&=
\lim_{n\to\infty}\int_0^{\infty}e^{-qt}\mathbb{P}_x\left(U_n(t)\geq y, \overline{U}_n(t)\leq a, \underline{U}_n(t)\geq d \right)dt\\
&=\int_0^{\infty}e^{-qt}\mathbb{P}_x\left(U(t)\geq y, \overline{U}(t)\leq a, \underline{U}(t)\geq d\right)dt\\
&=\mathbb{E}_x\left[\int_0^{\kappa^{a,+}\wedge\kappa^{d,-}} e^{-qt}\ind_{\{U(t)\geq y\}}dy\right].
\end{align*}
%On the other hand, we recall that for $y\in(b_i,b_{i+1}]$ 
%\begin{align*}
%\Xi_{\phi_i}(y)=\prod_{j=1}^i\left(1-\delta_j W_{j-1}^{(q)}(0)\right)=1-W^{(q)}(0)\phi_i(y).
%\end{align*}
%Hence,
%\begin{align*}
%\sum_{i=0}^{n}\left[\prod_{j=1}^i\left(1-\delta_j W_{j-1}^{(q)}(0)\right)\right]^{-1}\ind_{\{y\in  (b_i,b_{i+1}]\}}&=\sum_{i=0}^{n}\left[1-W^{(q)}(0)\sum_{j=1}^i\delta_j\right]^{-1}\ind_{\{y\in (b_i,b_{i+1}]\}}\\
%&=\left[1-W^{(q)}(0)\phi_n(y)\right]^{-1}=\Xi_{\phi_n}(y)^{-1}.
%\end{align*}	
%Therefore,
%\begin{align}\label{aux}
%\sum_{i=0}^{n}& \frac{\frac{w_n^{(q)}(x;d)}{w_n^{(q)}(a;d)}w_n^{(q)}(a;y)-w_n^{(q)}(x;y)}{\prod_{j=1}^i\left(1-\delta_j W_{j-1}^{(q)}(0)\right)}\ind_{\{y\in (b_i,b_{i+1}]\}} dy\notag\\
%&=\Xi_{\phi_n}(y)^{-1} \left\{\frac{w_n^{(q)}(x;d)}{w_n^{(q)}(a;d)}w_n^{(q)}(a;y)-w_n^{(q)}(x;y)\right\} dy.
%\end{align}
We recall that
\begin{align*}
\Xi_{\phi_{n}}(y)=1-W^{(q)}(0)\phi_{n}(y).
\end{align*}
Hence,
\begin{align*}
\mathbb{E}_x\left[\int_0^{\kappa^{a,+}\wedge\kappa^{d,-}}e^{-qt}\ind_{\{U(t)\geq y\}}dt\right]=&\lim_{n\to\infty}\mathbb{E}_x\left[\int_0^{\kappa_n^{a,+}\wedge\kappa_n^{d,-}}e^{-qt}\ind_{\{U_n(t)\geq y\}}dt\right]\\
&=\lim_{n\to\infty}\int_{y}^{\infty} \Xi_{\phi_{n}}(s)^{-1} \left\{\frac{w_n^{(q)}(x;d)}{w_n^{(q)}(a;d)}w_n^{(q)}(a;s)-w_n^{(q)}(x;s)\right\} ds\\
&=\int_{y}^{\infty}\Xi_{\phi}(s)^{-1} \left\{\frac{w^{(q)}(x;d)}{w^{(q)}(a;d)}w^{(q)}(a;s)-w^{(q)}(x;s)\right\} ds.
\end{align*}
%Finally,
%\begin{align*}
%\mathbb{E}_x\left[\int_0^{\kappa^{a,+}\wedge\kappa^{0,-}}e^{-qt}\ind_{\{U(t)\geq y\}}dt\right]=&\lim_{n\to\infty}\mathbb{E}_x\left[\int_0^{\kappa_n^{a,+}\wedge\kappa_n^{0,-}}e^{-qt}\ind_{\{U^n(t)\geq y\}}dt\right]\\
%&=\lim_{n\to\infty}\int_{y}^{\infty} \Xi_{\phi_n}(s)^{-1} \left\{\frac{w_n^{(q)}(x;d)}{w_n^{(q)}(a;d)}w_n^{(q)}(a;s)-w_n^{(q)}(x;s)\right\} ds\\
%&=\int_{y}^{\infty}\Xi_{\phi}(s)^{-1} \left\{\frac{w^{(q)}(x;d)}{w^{(q)}(a;d)}w^{(q)}(a;s)-w^{(q)}(x;s)\right\} ds.
%\end{align*}
(ii) %\footnote{See new proof!} 
 Identity \eqref{resolvent2_g} follows by taking $a\uparrow\infty$ in \eqref{resolvent1_g}. % and applying Lemma \ref{fun_u,v} (ii).\\
%We have first to notice that
%\begin{eqnarray*}
  %1&\ge& \frac{\e_x \left[ \mathrm{e}^{-q \kappa^{a,+}}; \kappa^{a,+}< \kappa^{d,-}\right]}
  %{\e_x \left[ \mathrm{e}^{-q \kappa^{a,+}}; \kappa^{a,+}< \kappa^{y,-}\right]}=\frac{w^{(q)}(a;y)}{w^{(q)}(a;d)}\frac{w^{(q)}(x;d)}{w^{(q)}(x;y)},
  %\end{eqnarray*}
%because 
%$\{\tau^{a,+}<\tau^{d,-}\}\subset \{\tau^{a,+}<\tau^{y,-}\}$ and
%\begin{eqnarray*}
  %\lefteqn{\e_x \left[ \mathrm{e}^{-q \kappa^{a,+}}; \kappa^{a,+}< \kappa^{y,-}\right]}\\
  %&=&\e_x \left[ \mathrm{e}^{-q \kappa^{a,+}}; \{\kappa^{a,+}< \kappa^{y,-}\}\cap  \{\kappa^{a,+}< \kappa^{d,-}\} \right]+\e_x \left[ \mathrm{e}^{-q \kappa^{a,+}}; \{\kappa^{a,+}< \kappa^{y,-}\}\cap  \{\kappa^{a,+}\ge \kappa^{d,-}\} \right]\\
  %&\ge& \e_x \left[ \mathrm{e}^{-q \kappa^{a,+}}; \kappa^{a,+}< \kappa^{d,-}\right].
%\end{eqnarray*}
%Hence
%$$\frac{w^{(q)}(a;y)}{w^{(q)}(a;d)}\le \frac{w^{(q)}(x;d)}{w^{(q)}(x;y)}.$$ \red{[Is the above inequalities correct? But for below, are they necessary?]}
Since the left hand side of \eqref{resolvent1_g} is a measure, we must have
\[
\Xi_{\phi}(y)^{-1} \left(\frac{w^{(q)}(x;d)}{w^{(q)}(a;d)}w^{(q)}(a;y)-w^{(q)}(x;y)\right)\ge0,\qquad\text{for $x,y\in[d,a]$,}
\]
which yields 
$$\frac{w^{(q)}(a;y)}{w^{(q)}(a;d)}\ge \frac{w^{(q)}(x;y)}{w^{(q)}(x;d)}>0,\qquad\text{for $x,y\in[d,a]$.}$$
%Thus there exists a sequence $a_n\nearrow \infty$ for which
%$$\lim_{n\to\infty}\frac{w^{(q)}(a_n;y)}{w^{(q)}(a_n;d)}=c^{(q)}(y;d),$$
%and $c^{(q)}(y;d)>0$,  $c^{(q)}(d;d)=1$. \blue{[Why should it converge to finite limit through a sequence?]} Now the left hand side of \eqref{resolvent1_g} converges to a the measure
  %$$\mu(\mathcal{B};c^{(q)})\blue{\mu(\mathcal{B};x,c^{(q)})?}=\int_d^\infty  \Xi_{\phi}(y)^{-1} \left(c^{(q)}(y;d)w^{(q)}(x;d)-w^{(q)}(x;y)\right)\ind_{\mathcal{B}}(y)\;dy.$$
%Suppose that for another sequence, there exists a subsequence $a_n^{'}\nearrow\infty$, for which
%\blue{[I guess it is enough to say: "Suppose that there exists a sequence $a_n^{'}\nearrow\infty$, for which"]?}
%$$\lim_{n\to\infty}\frac{w^{(q)}(a_n^{'};y)}{w^{(q)}(a_n^{'};d)}=\tilde{c}^{(q)}(y;d)\neq c^{(q)}(y;d).$$
%However then $\mu(\cdot;c^{(q)})\neq \mu(\cdot;\tilde{c}^{(q)})$, which contradicts that the left hand side of \eqref{resolvent2_g} is converging to a measure. \blue{[I think the argument is not quite correct, in the sense that $\tilde{c}^{(q)}(y;d)\neq c^{(q)}(y;d)$ can hold in a set of zero Lebesgue measure?]}\\
 Now, using the above inequality
%$$\frac{w^{(q)}(a;y)}{w^{(q)}(a;d)}\ge \frac{w^{(q)}(x;y)}{w^{(q)}(x;d)}>0,\qquad\text{\blue{for $x,y\in[d,a]$}.}$$	
we have that the mapping $a\mapsto\frac{w^{(q)}(a;y)}{w^{(q)}(a;d)}$ is monotone increasing, for fixed $y>d$. Hence it converges and we denote 
\[
c^{(q)}(y;d):=\lim_{a\to\infty}\frac{w^{(q)}(a;y)}{w^{(q)}(a;d)},\qquad\text{ for $y>d$.}
\]
Now we note using monotone convergence 
\begin{align*}
\int_d^\infty  \Xi_{\phi}(y)^{-1} \left(c^{(q)}(y;d)w^{(q)}(x;d)-w^{(q)}(x;y)\right)\;dy=\e_x\left[\int_0^{\kappa^{d,-}}e^{-qt}dt\right]<\infty.
\end{align*}
This implies that $c^{(q)}(y;d)<\infty$ for $y>d$ Lebesgue a.e. Hence, identity \eqref{resolvent2_g} follows by monotone convergence by taking $a\uparrow\infty$ in \eqref{resolvent1_g}.

(iii) 
The result follows by taking the limit $d\downarrow-\infty$ in \eqref{resolvent1_g} and applying Lemma \ref{fun_u,v}.\\
%(iv) We note that using \eqref{fun_u}, Exercise 8.5 in \cite{kyprianou2006}, and dominated convergence, it follows that
%\begin{align}\label{lim_u}
%\lim_{a\to\infty}&\frac{u^{(q)}(a)}{W^{(q)}(a)}=\lim_{a\to\infty}\Bigg[\frac{e^{\Phi(q)a}}{W^{(q)}(a)}+\int_0^{a}\frac{W^{(q)}(a-y)}{W^{(q)}(a)}\phi(y)\Bigg(\frac{\Phi(q)%e^{\Phi(q)y}}{1-\phi(y)W^{(q)}(0)}\notag\\&+\int_0^y \Phi(q)\frac{K^{\dagger}(y,z)}{1-\phi(y)W^{(q)}(0)}e^{\Phi(q)z}\,dz\Bigg)dy\Bigg]\notag\\
%&=\psi^{\prime}(\Phi(q))\notag\\&+\int_0^{\infty}\Xi_{\phi}(y)^{-1}e^{-\Phi(q)y}\phi(y)\left(\Phi(q)e^{\Phi(q)y}+\int_0^y \Phi(q) K^{\dagger}(y,z) e^{\Phi(q)z}\,dz\right)%dy=A(q).
%\end{align}
%The result follows by taking the limit $a\uparrow\infty$ in \eqref{resolvent3_g} and applying \eqref{lim_u}.
\end{proof}
Now, we will prove the identities related to the two-sided exit problem for the level-dependent  L\'evy process $U$.
\begin{theorem}{\textbf{(Two-sided exit problem)}}\label{th:mr.recursion_sd}
	\begin{itemize}
		\item[(i)]
		For $d \leq x\leq a$ and $q\geq 0$,
		\begin{equation}\label{main_twosided_up_sd}
		\e_x \left[ \mathrm{e}^{-q \kappa^{a,+}} \ind_{\{\kappa^{a,+}< \kappa^{d,-}\}} \right] =\frac{\wq(x;d)}{\wq(a;d)}.
		\end{equation} 		
		\item[(ii)] For $ 0 \leq x\leq a$ and $q \geq 0$,
		\begin{equation}\label{main_twosided_down_sd}
		\e_x \left[ \mathrm{e}^{-q \kappa^{0,-}} \ind_{\{\kappa^{0,-}<\kappa^{a,+}\}} \right] =z^{(q)}(x)-\frac{z^{(q)}(a)}{\wq(a)}\wq(x).
		\end{equation}
	\end{itemize}
\end{theorem}
\begin{proof}
(i) As for the previous result, we can take a non-increasing sequence of multi-refracted L\'evy process $(U^n)_{n\geq 0}$ 
that converges uniformly on compact sets to the level-dependent L\'evy process $U$ with the rate function $\phi$. Then, by the proof of Theorems 4 and 5 in \cite{kyprianouloeffen2010},

\begin{align*}
&\e_x \left[ \mathrm{e}^{-q \kappa^{a,+}} \ind_{\{\kappa^{a,+}< \kappa^{d,-}\}} \right]\\&=
q\frac{\int_0^{\infty} e^{-qt}\mathbb{P}_x(U(t)\in\mathbb{R},\underline{U}(t)\geq d)dt-\int_0^{\infty} e^{-qt}
	\mathbb{P}_x(U(t)\in[0,a],\underline{U}(t)\geq d,\overline{U}(t)\leq a)dt}
{1-q\int_0^{\infty} e^{-qt}\mathbb{P}_a(U(t)\in(-\infty,a],\overline{U}(t)\leq a)dt}.
\end{align*}	
	
On the other hand, by the proof of Theorem \ref{Resolvents_sd}, it follows that

\begin{align*}
\lim_{n\to\infty}&q\frac{\int_0^{\infty} e^{-qt}\mathbb{P}_x(U^n(t)\in\mathbb{R},\underline{U}_n(t)\geq d))dt-\int_0^{\infty} e^{-qt}\mathbb{P}_x(U^n(t)\in[0,a],\underline{U}_n(t)\geq d,
	\overline{U}_n(t)\leq a))dt}{1-q\int_0^{\infty} e^{-qt}\mathbb{P}_a(U^n(t)\in(-\infty,a],\overline{U}_n(t)\leq a)dt}\\
&=q\frac{\int_0^{\infty} e^{-qt}\mathbb{P}_x(U(t)\in\mathbb{R},\underline{U}(t)\geq d)dt-\int_0^{\infty} e^{-qt}\mathbb{P}_x(U(t)\in[0,a],\underline{U}(t)\geq d,\overline{U}(t)\leq a)dt}
{1-q\int_0^{\infty} e^{-qt}\mathbb{P}_a(U(t)\in(-\infty,a],\overline{U}(t)\leq a)dt}.
\end{align*}	
	
Hence, by Theorem \ref{th:mr.recursion} (i) and Theorem \ref{conv_sf}, we obtain

\begin{align*}
\e_x &\left[ \mathrm{e}^{-q \kappa^{a,+}} \ind_{\{\kappa^{a,+}< \kappa^{d,-}\}} \right]\\
&=\lim_{n\to\infty}q\frac{\int_0^{\infty} e^{-qt}\mathbb{P}_x(U^n(t)\in\mathbb{R},\underline{U}_n(t)\geq d)dt-\int_0^{\infty} e^{-qt}\mathbb{P}_x(U^n(t)\in[0,a],
	\underline{U}_n(t)\geq d,\overline{U}_n(t)\leq a)dt}{1-q\int_0^{\infty} e^{-qt}\mathbb{P}_a(U^n(t)\in(-\infty,a],\overline{U}_n(t)\leq a)dt}\\
&=\lim_{n\to\infty}\frac{w_n^{(q)}(x;d)}{w_n^{(q)}(a;d)}=\frac{\wq(x;d)}{\wq(a;d)}.
\end{align*}	
	
(ii) By (i) and the proof of Theorems 4 and 5 in \cite{kyprianouloeffen2010}, we have
\begin{align*}
\e_x &\left[ \mathrm{e}^{-q \kappa^{0,-}} \ind_{\{\kappa^{0,-}<\kappa^{a,+}\}} \right]\\
=&1-q\int_0^{\infty}\mathbb{P}_x(U(t)\in\mathbb{R},\underline{U}(t)\geq 0)dt-
 \frac{\wq(x)}{\wq(a)}\left(1-q\int_0^{\infty}\mathbb{P}_a(U(t)\in\mathbb{R},\underline{U}(t)\geq 0)dt\right). 
\end{align*}
Then, by the proof Theorem \ref{Resolvents_sd}, for any $x\geq0$ it holds that
\begin{align*}
\lim_{n\to\infty}\left(1-q\int_0^{\infty}\mathbb{P}_x(U^n(t)\in\mathbb{R},\underline{U}_n(t)\geq 0)dt\right)=1-q\int_0^{\infty}\mathbb{P}_x(U(t)\in\mathbb{R},\underline{U}(t)\geq 0)dt.
\end{align*}
Therefore, using Theorems \ref{th:mr.recursion} (i) and \ref{conv_sf} we obtain that
\begin{align*}
\e_x &\left[ \mathrm{e}^{-q \kappa^{0,-}} \ind_{\{\kappa^{0,-}<\kappa^{a,+}\}} \right]\\
&=\lim_{n\to\infty}\left(1-q\int_0^{\infty}\mathbb{P}_x(U^n(t)\in\mathbb{R},\underline{U}_n(t)\geq 0)dt\right)\\
&-\lim_{n\to\infty}\frac{\wq_n(x)}{\wq_n(a)} \left(1-q\int_0^{\infty}\mathbb{P}_a(U^n(t)\in\mathbb{R},\underline{U}_n(t)\geq 0)dt\right)\\
&=\lim_{n\to\infty}\left(z^{(q)}_n(x)-\frac{z_n^{(q)}(a)}{w_n^{(q)}(a)}w_n^{(q)}(x)\right)
=z^{(q)}(x)-\frac{z^{(q)}(a)}{w^{(q)}(a)}w^{(q)}(x).
\end{align*}
\end{proof}
By proceeding in a similar fashion to the proof of the previous result, we obtain the following theorem.
\begin{theorem}{\textbf{(One-sided exit problem)}}\label{one_sided_sd}\\

	\item[(i)] For $x \geq 0$ and $q>0$, 
		\begin{equation}\label{main_onesided_down_sd}
		\e_x \left[ \mathrm{e}^{-q \kappa^{0,-}} \ind_{\{\kappa^{0,-}<\infty\}} \right] =z^{(q)}(x)-C^{(q)}\wq(x),
		\end{equation} 
		where $C^{(q)}:=\lim_{a \to \infty} \frac{z^{(q)}(a)}{\wq(a)}$.
		\item[(ii)]
		For $x\leq a$ and $q \geq 0$,
		\begin{equation}\label{main_onesided_up_sd}
		\e_x \left[ \mathrm{e}^{-q \kappa^{a,+}} \ind_{\{\kappa^{a,+}<\infty\}} \right] =
		\frac{u^{(q)}(x)}{u^{(q)}(a)}.
		\end{equation}
\end{theorem}
\begin{proof}
(i) First, we note that from Theorem \ref{th:mr.recursion_sd} (ii) the ratio $\frac{z^{(q)}(a)}{\wq(a)}$ is monotone in $a$, then taking the limit on both sides in Theorem \ref{th:mr.recursion_sd} (ii) 
we obtain the result.\\
%\begin{align}\label{lim_z_a}
%\lim_{a\to\infty}\frac{z^{(q)}(a)}{W^{(q)}(a)}&=\lim_{a\to\infty}\left(\frac{Z^{(q)}(a)}{W^{(q)}(a)}+\int_0^a\frac{W^{(q)}(a-y)}{W^{(q)}(a)}\phi(y)z^{(q)\prime}(y)dy\right)\notag\\
%&=\frac{q}{\Phi(q)}+\int_0^{\infty}e^{-\Phi(q)y}\phi(y)z^{(q)\prime}(y)dy.
%\end{align}
%Therefore by taking $a\uparrow\infty$ in \eqref{main_twosided_down_sd} and using \eqref{fun_v} and \eqref{lim_z_a}, we obtain the result.\\
(ii) The result follows as in the proof of Theorem \ref{th:mr.recursion_sd} (i).
\end{proof}

%\section{Ruin functions for the level-dependend risk processes}
%Consider a model with scale function $w$ corresponding to rate function $\psi$.
Now, we will compute the ruin probability for the level-dependent L\'evy process $U$ with rate function $\phi$, under the assumption 
that $\mathbb{E}[X(1)]>0$.
Following the considerations from Corollary \ref{ruin_probab},
the ruin function is given by
$$\Psi(x):=\mathbb{P}_x\left(\kappa^{0,-}<\infty\right)=1- \lim_{a\to\infty}\frac{w(x)}{w(a)}.$$
 Notice that the result below is consistent with Corollary
  \ref{ruin_probab}, if $\phi=\phi_k$. For the sake of brevity we leave simple calculations to the reader.
%Hence, we have the following result.
\begin{prop} 
Assume that $x\geq0$.
\begin{itemize}
	\item[(i)] If $\mathbb{E}[X(1)]\leq  0$, then $\Psi(x)=1$ for all $x\geq0$.
  \item[(ii)] If $\mathbb{E}[X(1)]>0$ and 
  $\int_0^\infty \phi(x)w^{\prime}(x)\,dx$ exists, then
  the ruin function
  $$\Psi(x)=1-A^{-1}w(x),$$
  where
  $$A=\frac{1+\int_0^\infty \phi(x)w^{\prime}(x)\,dx}{\mathbb{E}[X(1)]}$$
  %Denoting $g$ to be the derivative of $-\Psi$, it fulfills the following Volterra equation
  %$$g(x)=A^{-1}W^{'}(x)+\int_0^x W^{'}(x-y)\phi(y)g(y)\,dy.$$
and $\Psi$ satisfies the following Volterra equation:
  \begin{align*}
  \Psi(x)=1-A^{-1}W(x)+\int_0^xW(x-y)\phi(y)\Psi^{\prime}(y)\md y.
  \end{align*}
  Moreover, when $\int_0^\infty \phi(x)w'(x)\,dx = \infty$ it follows that $A=\infty$, and hence $\Psi=1$. 
\end{itemize}
\end{prop}
\begin{proof}
We recall that $w$ satisfies the following integral equation:
\begin{align*}
w(x)=W(x)+\int_0^xW(x-y)\phi(y)w^{\prime}(y) \md y.
\end{align*}
Hence, if $\mathbb{E}[X(1)]\leq0$ then $\lim_{a\to\infty}W(a)=\infty$, which using the fact that $w(a)\geq W(a)$, for every $a\geq 0$ implies that
\[
\lim_{x\to\infty}w(x)=\infty.
\]
On the other hand, if $\mathbb{E}[X(1)]>0$ then $\lim_{a\to\infty}W(a)=1/\mathbb{E}[X(1)]$, and hence
\begin{equation}
\lim_{a\to\infty}w(a)=\frac{1}{\mathbb{E}[X(1)]}\left(1+\int_0^{\infty}\phi(y)w^{\prime}(y)\md y\right).
	\end{equation}
Finally, we note that
\begin{align*}
\Psi(x)=1-A^{-1}w(x)&=1-A^{-1}\left(W(x)+\int_0^xW(x-y)\phi(y)w^{\prime}(y) \md y\right)\\
&=1-A^{-1}W(x)+\int_0^xW(x-y)\phi(y)\Psi^{\prime}(y)\md y.
\end{align*}
\end{proof}

\begin{example}[Ornstein-Uhlenbeck type process]
For a general spectrally negative L\'evy process $X$ we consider the particular case when $\phi(x)=x\ind_{\{x> 0\}}$ %\blue{(I am not sure we can consider $\phi(x)=x1_{\{x> d'\}}$ because the function $\phi$ is not Lipschitz continuous)} \green{[I guess this also violates the nonegativity?]}\blue{[Yeah I think so too.]} 
in SDE \eqref{sde1}. We refer to the process $U$, given as the unique solution to the following stochastic differential equation,
\begin{align*}
U(t)=X(t)-\int_0^tU(s)\ind_{\{U(s)>0\}}ds,\ \text{$t\geq 0$,}
\end{align*}
as an Ornstein-Uhlenbeck type process. \\

Notice that for fluctuation identities with $0\le d<a$, $U$ evolves as a standard L\'evy-driven
Ornstein-Uhlenbeck process. Such processes were considered in the literature, e.g. by Hadjiev \cite{hadjiev}, Novikov \cite{novikov} and Patie \cite{patie1}, in which the Laplace transforms of the first passage times were found.\\

%and we present the formula for the scale function $w^{(q)}(x;d)$.
%In order to obtain the associated fluctuaction identities associated to this type of processes we only need to compute expressions for the scale functions $w^{(q)}(\cdot;d)$ and $z^{(q)}(\cdot)$ for $d\in \mathbb{R}$. 
%Using Volterra equation  \eqref{eq:volterra.w.prime} we get that $w^{(q)}(x;d)=\int_d^x w^{(q)\prime}(y;d) \md y$ where $w^{(q)\prime}$ is represented by formula \eqref{sf_explicit}. 
%The kernel is given by $K^{\dagger}(x,y)=\sum_{l=1}^\infty K^{(l)}(x,y)$, where $K^{(1)}=K$, $K^{(l+1)}(x,y)=\int_y^xK^{(l)}(x,w)K(w,y)\,dw$
%and $K(x,y)= \Xi_{\phi}(x)^{-1} y W^{(q)\prime}((x-y)+)$, for $x>y>d$. 

Following Proposition \ref{prop:solutions},  we have for $d\in\mathbb{R}$,
\begin{align}\label{uo_sf}
w^{(q)\prime}(x;d)&=\Xi_{\phi}(x)^{-1}W^{(q)\prime}((x-d)+)+\int_d^x K^{\dagger}(x,y)\Xi_{\phi}(x)^{-1}W^{(q)\prime}(y-d)\,dy\notag\\
z^{(q)\prime}(x)&= \Xi_{\phi}(x)^{-1} qW^{(q)}(x)+\int_0^x K^{\dagger}(x,y)\Xi_{\phi}(x)^{-1}qW^{(q)}(y)\,dy,
\end{align}
where for $d\leq y<x$, $\Xi_{\phi}(x)=1-xW^{(q)}(0)\ind_{\{x>0\}}$, $K^{\dagger}(x,y)=\sum_{l=1}^\infty K^{(l)}(x,y)$,
\begin{align*}
K^{(1)}(x,y)&=\Xi_{\phi}(x)^{-1} y\ind_{\{y>0\}} W^{(q)\prime}((x-y)+), \text{and},\\ K^{(l+1)}(x,y)&=\int_y^xK^{(l)}(x,w)K(w,y)\,dw \ \text{for $l\geq 1$.}
\end{align*}
Hence, the scale functions are given by
\begin{align*}
w^{(q)}(x;d)=W^{(q)}(0)+\int_d^xw^{(q)\prime}(y;d)dy, \ \text{and} \ z^{(q)}(x) =1+\int_0^xz^{(q)\prime}(y)dy .
\end{align*}
%\red{We want to emphasis that the above representation is useful for numerical computations to get the scale function $w^{(q)}$} 
%\red{Moreover using this approach we can also get scale functions $z^{(q)}$ and $u^{(q)}$ to obtain another fluctuation identities. [Do you think it is ok?? What else to add??]}
We remark that the representation of the functions $w^{(q)}(\cdot,d)$ and $z^{(q)}(\cdot)$ given in \eqref{uo_sf} using the kernel $K^{\dagger}$, provides a natural scheme to compute the scale functions numerically.
\end{example}

\appendix
\section{Proofs for the bounded variation case}\label{BVcase}

In this appendix, we prove Theorems \ref{th:mr.recursion}(i) and \ref{Resolvents}(i) for the case that $X$ has paths of bounded variation. 
The proofs for the case of unbounded variation are deferred to Appendix \ref{UBVcase}.

\subsection{Proof of Theorem \ref{th:mr.recursion}(i) }
%We start with \textbf{the proof of Theorem \ref{th:mr.recursion}(i)}.
The proof is inductive. First, it is easy to check that for $k=1$, formula (\ref{main_twosided_up}) agrees with (\ref{two-sided_kypr&loeffen}). 

 Now, we shall assume that equation (\ref{main_twosided_up}) holds true for $k-1$, and show that it also holds for $k$.  
We let $p(x;\delta_1,...,\delta_k) :=\e_x \left[ \mathrm{e}^{-q \kappa_k^{a,+}} \ind_{\{\kappa_k^{a,+}< \kappa_k^{d,-}\}} \right]$. 
We follow the main idea given in Theorem 16 of \cite{kyprianouloeffen2010}. 
For $x \leq b_k$, it follows from the strong Markov property and the assumption that the identity holds for $k-1$ that
\begin{equation}\label{under_b_k}
p(x;\delta_1,...,\delta_k)=\e_x \left[ \mathrm{e}^{-q \kappa_{k-1}^{b_k,+}} \ind_{\{\kappa_{k-1}^{b_k,+}< \kappa_{k-1}^{d,-}\}} \right]
p(b_k;\delta_1,...,\delta_k) =\frac{\wq_{k-1}(x;d)}{\wq_{k-1}(b_k;d)} p(b_k;\delta_1,...,\delta_k). 
\end{equation}

For $b_k \leq x \leq  a$, by applying  the strong Markov property once more and considering equation (\ref{under_b_k}) and the expectation given in Theorem 23 (iii) of
\cite{kyprianouloeffen2010}, we obtain that 

\begin{eqnarray}\label{above_b_k}
\lefteqn{p(x;\delta_1,...,\delta_k)= \e_x \left[ \mathrm{e}^{-q \tau_k^{a,+}} \ind_{\{\tau_k^{a,+}< \tau_k^{b_k,-}\}} \right]+\e_x \left[ \mathrm{e}^{-q \kappa_k^{a,+}} \ind_{\{\kappa_k^{a,+}< \kappa_k^{d,-}\}} 
\ind_{\{\tau_k^{a,+}> \tau_k^{b_k,-}\}} \right]}\nonumber\\&&=
\frac{W_k^{(q)}(x-b_k)}{W_k^{(q)}(a-b_k)}+\e_x \left[ \mathrm{e}^{-q \tau_k^{b_k,-}} \ind_{\{\tau_k^{b_k,-}< \tau_k^{a,+}\} } 
p(U_k(\tau_k^{b_k,-});\delta_1,...,\delta_k) \right] \nonumber\\&&=
\frac{W_k^{(q)}(x-b_k)}{W_k^{(q)}(a-b_k)}+\frac{p(b_k;\delta_1,...,\delta_k)}{\wq_{k-1}(b_k;d)}\e_x \left[ \mathrm{e}^{-q \tau_k^{b_k,-}} \ind_{\{\tau_k^{b_k,-}< \tau_k^{a,+}\} } \wq_{k-1}(X_k(\tau_k^{b_k,-});d) \right]\nonumber\\\nonumber&&=
\frac{W_k^{(q)}(x-b_k)}{W_k^{(q)}(a-b_k)}+\frac{p(b_k;\delta_1,...,\delta_k)}{\wq_{k-1}(b_k;d)}
\int_{0}^{a-b_k}\int_{(y,\infty)} \wq_{k-1}(b_k+y-\theta;d) \\&& \quad \times 
\left[\frac{W_k^{(q)}(x-b_k)W_k^{(q)}(a-b_k-y)}{W_k^{(q)}(a-b_k)}-W_k^{(q)}(x-b_k-y)\right] \Pi(\textrm{d}\theta)\textrm{d}y. \end{eqnarray}

Now,  by setting $x=b_k$ in (\ref{above_b_k}) and using the fact that $W_k^{(q)}(0)=\frac{1}{c-(\delta_1+...+\delta_k)}$,  we obtain 
\begin{eqnarray}\label{p_for_b_k}
\lefteqn{p(b_k;\delta_1,...,\delta_{k-1},\delta_k)=\wq_{k-1}(b_k;d)\Big\{(c-\delta_1-\dots -\delta_k)W_k^{(q)}(a-b_k)\wq_{k-1}(b_k;d)}\nonumber\\&&-\int_{0}^{a-b_k}\int_{(y,\infty)}
  \wq_{k-1}(b_k+y-\theta;d) W_k^{(q)}(a-b_k-y)\Pi(\textrm{d}\theta)\textrm{d}y \Big\}^{-1}. 
\end{eqnarray}

We shall now simplify this expression.

First, if we take $\delta_k=0$, then by the inductive  hypothesis we obtain that
\begin{eqnarray}\label{p_for_delta=0}
p(b_k;\delta_1,...,\delta_{k-1},0)=\e_{b_k} \left[ \mathrm{e}^{-q \kappa_{k-1}^{a,+}} \ind_{\{\kappa_{k-1}^{a,+}< \kappa_{k-1}^{d,-}\}} \right]=\frac{\wq_{k-1}(b_k;d)}{\wq_{k-1}(a;d)}.
\end{eqnarray}
Then, it follows from (\ref{p_for_b_k}) and (\ref{p_for_delta=0}) and the fact that $W_k = W_{k-1}$ when $\delta_k = 0$ that
\begin{multline}\label{integral_for_delta=0}
\int_{0}^{a-b_k}\int_{(y,\infty)} \wq_{k-1}(b_k+y-\theta;d) W_{k-1}^{(q)}(a-b_k-y)\;\Pi(\textrm{d}\theta)\;\textrm{d}y\\
=\left(c-\sum_{i=1}^{k-1} \delta_i\right)W_{k-1}^{(q)}(a-b_k)\wq_{k-1}(b_k;d)-\wq_{k-1}(a;d).
\end{multline}
Let $\lambda > \varphi_{k}(q) $.
As $a\geq b_k$ is taken arbitrarily, we set $a=u$ and take the Laplace transforms from $b_k$ to $\infty$ of both sides of (\ref{integral_for_delta=0}). 
% \green{It is a little confusing with $x$ of $p(x, ...)$  Not a big deal, but we could change $x$ to $u$ if you don't mind.} \blue{ OK, I changed it, please check if everywhere}.
By Fubini's theorem, the Laplace transform of the left-hand side of (\ref{integral_for_delta=0}) becomes

\begin{multline} \label{eqn_kazu_referenced}
\int_{b_k}^{\infty}e^{-\lambda u}\int_{0}^{\infty}\int_{(y,\infty)} \wq_{k-1}(b_k+y-\theta;d) W_{k-1}^{(q)}(u-b_k-y)\;\Pi(\textrm{d}\theta) \;\textrm{d}y\;\textrm{d}u\\=
\frac{e^{-\lambda b_k}}{\psi_{k-1}(\lambda)-q} 
\int_{0}^{\infty} e^{-\lambda y} \int_{y}^{\infty} \wq_{k-1}(b_k+y-\theta;d) \;\Pi(\textrm{d}\theta) \;\textrm{d}y.
\end{multline}

On the other hand, the Laplace transform of the right-hand side of (\ref{integral_for_delta=0}) becomes
\begin{multline*} 
\int_{b_k}^{\infty}e^{-\lambda u} \left((c-\delta_1-...-\delta_{k-1})W_{k-1}^{(q)}(u-b_k)\wq_{k-1}(b_k;d)-\wq_{k-1}(u;d)\right)\;\textrm{d}u\\=
\frac{e^{-\lambda b_k}(c-\delta_1-...-\delta_{k-1})}{\psi_{k-1}(\lambda)-q}\wq_{k-1}(b_k;d)-\int_{b_k}^{\infty} e^{-\lambda u} \wq_{k-1}(u;d) 
\;\textrm{d}u.
\end{multline*}

Hence, by matching these, we obtain 

\begin{multline} \label{integrals1}
\int_{0}^{\infty}  \int_{(y, \infty)} e^{-\lambda y} \wq_{k-1}(b_k+y-\theta;d) \;\Pi(\textrm{d}\theta)\; \textrm{d}y\\=
\left(c- \sum_{i=1}^{k-1}\delta_i\right)\wq_{k-1}(b_k;d)-\left(\psi_{k-1}(\lambda)-q\right) e^{\lambda b_k}\int_{b_k}^{\infty} e^{-\lambda u} \wq_{k-1}(u;d) \;\textrm{d}u.
\end{multline}

Now  using \eqref{eqn_kazu_referenced} and (\ref{integrals1}) and Fubini's theorem we obtain that 

\begin{eqnarray}\label{integrals2}\nonumber
\int_{b_k}^{\infty}e^{-\lambda u} \int_{0}^{\infty}\int_{(y,\infty)} \wq_{k-1}(b_k+y-\theta;d) W_k^{(q)}(u-b_k-y)\;\Pi(\textrm{d}\theta)\;\textrm{d}y \; \textrm{d}u\\\nonumber=
\frac{e^{-\lambda b_k}}{ \psi_{k}(\lambda)-q}\Big((c-\delta_1-...-\delta_{k-1}) \wq_{k-1}(b_k;d)\\-\left(\psi_{k-1}(\lambda)-q\right) 
e^{\lambda b_k}\int_{b_k}^{\infty} e^{-\lambda u} \wq_{k-1}(u;d) \;\textrm{d}u\Big). 
\end{eqnarray}

Finally, inversion of the Laplace transform (\ref{integrals2}) with respect to $\lambda$ 
(for details we refer the reader to p. 34 of \cite{kyprianouloeffen2010}) gives that, for all $u\geq b_k$,
\begin{multline}\label{integrals3}
 \int_{0}^{\infty}\int_{(y,\infty)} \wq_{k-1}(b_k+y-\theta;d) W_k^{(q)}(u-b_k-y)\;\Pi(\textrm{d}\theta)\;\textrm{d}y \\=
\left(c-\sum_{i=1}^{k-1} \delta_i\right) \wq_{k-1}(b_k;d)  W_k^{(q)}(u-b_k) - \wq_{k-1}(u;d)
-\delta_k \int^{x}_{b_{k}}W_k^{\left(q\right)}(u-y)w_{k-1}^{\left(q\right)'}(y;d)\;\md y.
\end{multline}

Then by using the above formula together with (\ref{p_for_b_k}), we compute that 

\begin{equation}\label{p_for_b_k2}
 p(b_k;\delta_1,...,\delta_k)=\wq_{k-1}(b_k;d)\Big\{
 \wq_{k-1}(a;d) -\delta_k \int^{a}_{b_{k}}W_k^{\left(q\right)}(a-y)w_{k-1}^{\left(q\right)'}(y;d)\md y\Big\}^{-1}.
\end{equation}

Finally, putting  (\ref{p_for_b_k2}) into (\ref{under_b_k}) and (\ref{above_b_k}) gives the equation (\ref{main_twosided_up}).  
\begin{flushright}$\blacksquare$\end{flushright}

%\blue{Please note that here in Remark should be $x$ instead of $u$.}
\begin{remark}\label{useful_identity}
It is easy to see that from equation (\ref{integrals3}) we obtain the following identity, which will be crucial for the remainder of the paper:
\begin{align}\label{integrals4}\nonumber
 &\int_{0}^{a-b_k}\int_{(y,\infty)} \wq_{k-1}(b_k+y-\theta;d)\left[\frac{W_k^{(q)}(x-b_k)W_k^{(q)}(a-b_k-y)}{W_k^{(q)}(a-b_k)}-W_k^{(q)}(x-b_k-y)\right] 
\Pi(\textrm{d}\theta)\;\textrm{d}y \\\nonumber & \qquad=
 -\frac{W_k^{(q)}(x-b_k)}{W_k^{(q)}(a-b_k)}\left(\wq_{k-1}(a;d)+\delta_k \int^{a}_{b_{k}}W_k^{\left(q\right)}(a-y)w_{k-1}^{\left(q\right)'}(y;d)\md y\right)\\ & \qquad+
\wq_{k-1}(x;d)+\delta_k \int^{x}_{b_{k}}W_k^{\left(q\right)}(x-y)w_{k-1}^{\left(q\right)'}(y;d)\;\md y.
\end{align}

\end{remark}

\subsection{Proof of Theorem \ref{Resolvents}(i).} 
The proof is inductive. Using equation (\ref{w_k_above_b_k}) it is easy to check for $k=1$ that the formula (\ref{resolvent1}) 
agrees with \cite{kyprianouloeffen2010} {Theorem 6(i)}. Now, we assume that equation (\ref{resolvent1}) 
holds true for $k-1$. By the strong Markov property, we have for $x\leq b_k$ that
\begin{eqnarray}\label{resolvent_under_b_k}\nonumber
\lefteqn{V^{(q)}(x, dy):=\e_x\left[\int_0^{\kappa_k^{a,+}\wedge \kappa_k^{d,-}}e^{-qt}\ind_{\{U_k(t)\in \;dy\}}\;dt\right]}\\\nonumber &=&
\e_x\left[\int_0^{\kappa_{k-1}^{b_k,+}\wedge \kappa_{k-1}^{d,-}} e^{-qt}\ind_{\{U_{k-1}(t)\in \;dy\}}\;dt\right]+ 
\e_x \left[ \mathrm{e}^{-q \kappa_{k-1}^{b_k,+}} \ind_{\{\kappa_{k-1}^{b_k,+}< \kappa_{k-1}^{d,-}\}} \right]V^{(q)}(b_k, dy) \\\nonumber &=&
 \sum_{i=0}^{k-1} \frac{\frac{w_{k-1}^{(q)}(x;d)}{w_{k-1}^{(q)}(b_k;d)}w_{k-1}^{(q)}(b_k;y)-w_{k-1}^{(q)}(x;y)}
{\Xi_{\phi_i}(y)}\ind_{\{y\in (b_i,b_{i+1}]\}}dy+ \frac{w_{k-1}^{(q)}(x;d)}{w^{(q)}_{k-1}(b_k;d)}V^{(q)}(b_k, dy),\\
\end{eqnarray}
where the last equality holds by the inductive hypothesis and Theorem \ref{th:mr.recursion}(i).
 
 For $b_k\leq x\leq a$, we have that
\begin{align*}
&V^{(q)}(x, dy)\\&=
\e_x\left[\int_0^{\tau_{k}^{b_k,-} \wedge \tau_{k}^{a,+} } e^{-qt}\ind_{\{X_k(t)\in dy \}}dt\right] + \e_x\left[ e^{-q \tau_{k}^{b_k,-} }\ind_{\{ \tau_{k}^{b_k,-} <\tau_{k}^{a,+}\}} V^{(q)}\left(X_k(\tau_{k}^{b_k,-} ), dy \right)\right].
 \end{align*}
 The former expectation on the right-hand side is equal to
 \begin{multline*}
\int_0^{\infty} e^{-qt} \p_x\left(X_k(t)\in dy, y\in [b_k,a], t<\tau_{k}^{b_k,-} \wedge \tau_{k}^{a,+}\right) \;dt  \\
=\left\{\frac{W_{k}^{(q)}(x-b_k)}{W_{k}^{(q)}(a-b_k)}W_{k}^{(q)}(a-y)-W_{k}^{(q)}(x-y)\right\}\ind_{\{y\in (b_k,a]\}}\;dy,
\end{multline*}
while by \eqref{resolvent_under_b_k}, the latter becomes
\begin{align*}
&\int_0^{\infty}\int_z^{\infty} V^{(q)}(z-\theta, dy) 
\left[\frac{W_{k}^{(q)}(x-b_k)}{W_{k}^{(q)}(a-b_k)}W_{k}^{(q)}(a-b_k-z)-W_{k}^{(q)}(x-b_k-z)\right]\;\Pi(d\theta)\;dz  \\
&=\Bigg\{\int_0^{\infty}\int_z^{\infty}\sum_{i=0}^{k-1}\frac{1}{\Xi_{\phi_i}(y)} \left( \frac{w_{k-1}^{(q)}(z-\theta+b_k;d)}
{w_{k-1}^{(q)}(b_k;d)}w_{k-1}^{(q)}(b_k;y)-w_{k-1}^{(q)}(z-\theta+b_k;y)\right)
\ind_{\{y\in (b_i,b_{i+1}]\}} \\ \quad &\times 
\left[\frac{W_{k}^{(q)}(x-b_k)}{W_{k}^{(q)}(a-b_k)}W_{k}^{(q)}(a-b_k-z)-W_{k}^{(q)}(x-b_k-z)\right]\;\Pi(d\theta)\;dz
\Bigg\} \;dy \\ &+
\int_0^{\infty}\int_z^{\infty}\frac{w_{k-1}^{(q)}(z-\theta+b_k;d)}{w^{(q)}_{k-1}(b_k;d)}V^{(q)}(b_k, dy)  \\
&\times\left[\frac{W_{k}^{(q)}(x-b_k)}{W_{k}^{(q)}(a-b_k)}W_{k}^{(q)}(a-b_k-z)-W_{k}^{(q)}(x-b_k-z)\right]\;\Pi(d\theta)\;dz.
\end{align*}

Next, we use Remark \ref{useful_identity} to simplify the expression for $V^{(q)}(x, dy)$ for $b_k \leq x \leq a$ as follows:
\begin{eqnarray}\label{resolvent_simply}\nonumber
\lefteqn{V^{(q)}(x,dy)=
\left\{\frac{W_{k}^{(q)}(x-b_k)}{W_{k}^{(q)}(a-b)}W_{k}^{(q)}(a-y)-W_{k}^{(q)}(x-y)\right\} \ind_{\{y\in (b_k,a]\}}\;dy}
\\\nonumber &+&
\Bigg\{\sum_{i=0}^{k-1}\Bigg[ \frac{w_{k-1}^{(q)}(b_k;y)}{w_{k-1}^{(q)}(b_k;d)}\left(- \frac{W_k^{(q)}(x-b_k)}{W_k^{(q)}(a-b_k)}\wq_{k}(a;d)+\wq_{k}(x;d)\right)\\\nonumber &+&
\frac{W_k^{(q)}(x-b_k)}{W_k^{(q)}(a-b_k)}\wq_{k}(a;y)-\wq_{k}(x;y)\Bigg\}\frac{\ind_{\{y\in (b_i,b_{i+1}]\}}}{\Xi_{\phi_i} (y)}\; dy \\&+&
\frac{V^{(q)}(b_k, dy)}{w^{(q)}_{k-1}(b_k)} \left( - \frac{W_k^{(q)}(x-b_k)}{W_k^{(q)}(a-b_k)}\wq_{k}(a;d)+\wq_{k}(x;d) \right).
\end{eqnarray}
Finally, setting $x=b_k$ in (\ref{resolvent_simply}) and using the fact that $w^{(q)}_{k}(b_k)=w^{(q)}_{k-1}(b_k)$ leads us to an 
explicit formula for $V^{(q)}(b_k,dy)$. Then, putting the expression for $V^{(q)}(b_k,dy)$ into (\ref{resolvent_simply})
 gives us that
\begin{align*}
&\e_x\left[\int_0^{\kappa_k^{a,+}\wedge \kappa_k^{d,-}}e^{-qt}\ind_{\{U_k(t)\in dy\}}dt\right]
= \sum_{i=0}^{k} \frac{\frac{w_k^{(q)}(x;d)}{w_k^{(q)}(a;d)}w_k^{(q)}(a;y)-w_k^{(q)}(x;y)}
{\Xi_{\phi_i}(y)}\ind_{\{y\in (b_i,b_{i+1}]\cap (-\infty, a)\}} dy\\
&= \Xi_{\phi_k}(y)^{-1} \left\{\frac{w_k^{(q)}(x;d)}{w_k^{(q)}(a;d)}w_k^{(q)}(a;y)-w_k^{(q)}(x;y)
\right\} \ind_{\{y \in (-\infty, a)\}} dy. \textrm{\hspace{5cm}} \blacksquare
\end{align*}

\section{Proofs for the unbounded variation case}\label{UBVcase}

In this appendix, we first show the existence of the multi-refracted L\'evy process $U_k$ (as in Theorem \ref{lem:mr.exist}) 
for the case that $X$ is of unbounded variation. Together with Lemma \ref{lemma_uniqueness}, this implies that $U_k$ is the unique strong solution of 
the SDE \eqref{procU_n}. We then prove Theorems \ref{th:mr.recursion}(i) and \ref{Resolvents}(i) for the case of unbounded variation.

\subsection{Existence}
 Now, we will prove  that there exists a process $U_k$ that is a solution
	to	\eqref{procU_n} under the assumption that $X$ has unbounded variation paths. To this end, we will use an approximation 
	method as in \cite{kyprianouloeffen2010}. First, recall that it is known (see, e.g., Bertoin \cite{bertoin1996} or Kyprianou and Loeffen
		\cite{kyprianouloeffen2010}, Lemma 12) that for any spectrally negative L\'evy process $X$ of unbounded variation 
		we can find a sequence of bounded variation L\'evy processes $X^{(n)}$ such that, for
		each $t>0$,
		\[
		\lim_{n\to\infty}\sup_{s\in[0,t]}|X^{(n)}(s)-X(s)|=0 \quad a.s.
		\]
		Let the constants $t > 0$ and  $\eta>0$ be fixed, and let $N \in \mathbb{N}$ be sufficiently large that for all $n,m \geq N ,$ 
		$\sup_{s \in [0,t]} |X^{(n)}(s)-X^{(m)}(s)| \leq \eta$. 
		
		Now, by $U^{(n)}_k$ we denote the sequence of $k$-multi-refracted pathwise solutions associated with $X^{(n)}$ for $n \geq 1$. 
		That is the solutions to the following SDE:
		\begin{align*}
			dU_k^{(n)}(t)=dX^{(n)}(t)-\sum_{i=1}^k\delta_i\ind_{\{U_k^{(n)}(t)>b_i\}}dt.
		\end{align*}
		
	The next step is to prove that
		\begin{equation}\label{convergence_U_k}
		\lim_{n,m\to\infty}\sup_{s\in[0,t]}|U_k^{(n)}(s)-U_k^{(m)}(s)|=0 \quad a.s.,
		\end{equation}
		from which we deduce that $\{U_k^{(n)}(s)\}_{s\in[0,t]}$ is a Cauchy sequence in the Banach space consisting of c\'adl\'ag mappings 
		equipped with the supremum norm.
To this end, for any $k\geq 1$, we adopt the reasoning given in the proof of Lemma 12 of \cite{kyprianouloeffen2010}, because in our case we have that
\begin{eqnarray*}
	\lefteqn{	A^{(n,m)}_k(s):= (U_k^{(n)}(s) - U_k^{(m)}(s))-(X^{(n)}(s)-X^{(m)}(s))}\\
	&=&\sum_{i=1}^k\delta_i \int_0^s \left(\ind_{\{U_k^{(m)}(\upsilon)>b_i, U_k^{(n)}(\upsilon)\leq b_i\}}-
		\ind_{\{U_k^{(m)}(\upsilon)\leq b_i, U_k^{(n)}(\upsilon)>b_i\}} \right) \, d\upsilon.
		\end{eqnarray*}
		Then, taking  advantage of the fact that $0<\delta_1,...,\delta_k$, we show as in \cite{kyprianouloeffen2010} that \\
		$\sup_{s \in [0,t]} |A^{(n,m)}_k(s)| \leq \eta$. Thus, we have shown that 
		$$\lim_{n \rightarrow \infty}U_k^{(n)}(t)=U_k^{(\infty)}(t) \quad \textrm {a.s.}$$
		for a stochastic process $U_k^{(\infty)}=\{U_k^{(\infty)}(t): t\geq 0\}$.
		Now, repeating the reasoning presented in Lemma 21 of \cite{kyprianouloeffen2010} 
		(using the fact that the resolvent obtained in Theorem \ref{Resolvents} for the bounded variation case has a density), 
		we show that for all driving L\'evy processes $X$ with paths of unbounded variation when $x$ is fixed we have for $1\leq j \leq k$ that 
\begin{equation}\label{at_b_j}
\mathbb{P}_x(U_k^{(\infty)}(t)= b_j) = 0 \textrm{ for almost every $t \geq  0$}.
\end{equation}
Finally, using \eqref{at_b_j} we obtain that 
\begin{align*}
			&\lim_{n \rightarrow \infty}U_k^{(n)}(t)=\lim_{n \rightarrow \infty}\Big(X^{(n)}(t)-\sum_{i=1}^k\delta_i \int_0^t 
			\ind_{\{U_k^{(n)}(s)>b_i\}}ds \Big) \\&  = X(t)-\sum_{i=1}^k\delta_i \int_0^t \ind_{\{U^{(\infty)}_k(s)>b_i\}}ds \quad a.s.,
		\end{align*}
which gives us that $U_k:=U_k^{(\infty)}$ solves \eqref{procU_n}, as desired.

\subsection{Proof of Theorem \ref{Resolvents}(i) }
Following the argument given in the proof of Lemma \ref{lem:mr.exist}, we conclude that for any $k\in\mathbb{N}$ the sequence $(U^{(n)}_k)_{n\geq 1}$ 
converges uniformly on compact sets to the process $U_k$.

Consider now for any $k\geq 1$ the sequence of scale functions $(w^{(q),(n)}_k)_{n\geq 1}$ 
%\red{[How about $(w^{(q,n)}_k)_{n\geq 1}$?]}\blue{- I prefer to stay with $w^{(q),(n)}$, because $n$ is an index which is not related to $(q)$-do you agree??}
 associated with the sequence of $k$-multi-refracted 
L\'evy processes $U^{(n)}_k$. In the next step, we will prove the convergence of $(w^{(q),(n)}_k)_{n\geq 1}$  to the scale function $w^{(q)}_k$ 
defined in \eqref{eq:mr.recursion}.
Therefore, for the case that $k=1$ we have by the proof of Lemma 20 in \cite{kyprianouloeffen2010} that, for any $x\geq d$,
\begin{align}\label{sf_ub_1}
\lim_{n\to\infty}w^{(q),(n)}_1(x;d)&=\lim_{n\to\infty}W^{(q),(n)}(x-d)+\delta_1\int_{b_1}^{x}W_1^{(q),(n)}(x-y)W^{(q),(n)\prime}(y-d)\;dy\notag\\
&=W^{(q)}(x-d)+\delta_1\int_{b_1}^{x}W_1^{(q)}(x-y)W^{(q)\prime}(y-d)\;dy=w_1^{(q)}(x;d),
\end{align}
where $W^{(q),(n)}$ ($W^{(q)}$) and $W_1^{(q),(n)}$ ($W_1^{(q)}$) are the scale functions associated with the processes $(X^{(n)}(t))_{t\geq0}$ ($(X(t))_{t\geq0}$) and 
$(X_1^{(n)}(t):=X^{(n)}(t)-\delta_1 t)_{t\geq0}$ ($(X_1(t))_{t\geq0}$), respectively.  
\par For each $k>1$ and $x\geq d$, we define the function $w^{(q),(n)}_k$ by the following recursion:
\begin{align}\label{req_w^{(q),(n)}}\nonumber
 &w^{(q),(n)}_k(x;d):= w^{(q),(n)}_{k-1}(x;d)+\delta_{k}\int^{x}_{b_{k}}W_k^{(q),(n)}(x-y)w_{k-1}^{(q),(n)\prime}(y;d)\md y\\\nonumber&=
w^{(q),(n)}_{k-1}(x;d)+\delta_k W_k^{(q),(n)}(0)w^{(q),(n)}_k(x;d)-\delta_k w_{k-1}^{(q),(n)}(b_k;d)W_k^{(q)}(x-b_k)\\&
+\delta_k \int_{b_k}^x W_k^{(q),(n)\prime}(x-y)w_{k-1}^{(q),(n)}(y;d)\;dy,
\end{align}
where the last equality follows from integration by parts. 
Assume now for the inductive step  that for any $x\geq d$ it holds that
\begin{align*}
\lim_{n\to\infty}w^{(q),(n)}_{k-1}(x;d)=w_{k-1}^{(q)}(x;d).
\end{align*}
We will prove that the above limit holds for $k$. To this end, we recall that the scale function $w^{(q),(n)}_k$ for the process $U^{(n)}_k$ 
satisfies the recurrence relation \eqref{req_w^{(q),(n)}}, where $W^{(q),(n)}_k$ is the scale function of the process 
$(X_k^{(n)}(t):=X^{(n)}(t)-\sum_{j=1}^k\delta_j t)_{t\geq0}$.
Now, by the proof of Lemma 20 in \cite{kyprianouloeffen2010} (see also Remark 3.2 of \cite{perezyamazaki2017}) 
we have that $W^{(q),(n)}_k$ and $W^{(q),(n)\prime}_k$ converge to $W^{(q)}_k$ and $W^{(q)\prime}_k$, respectively. 
Hence by the dominated convergence theorem we have for any $x\in \mathbb{R}$ and $x\geq d$ that
\begin{align*}
\lim_{n\to\infty}w^{(q),(n)}_k(x;d)&=\lim_{n\to\infty}\Bigg(w_{k-1}^{(q),(n)}(x;d)+\delta_k W_k^{(q),(n)}(0)w^{(q),(n)}_k(x;d)\\&
-\delta_kw_{k-1}^{(q),(n)}(b_k;d)W_k^{(q),(n)}(x-b_k)
+\delta_k \int_{b_k}^xW_k^{(q),(n)\prime}(x-y)w_{k-1}^{(q),(n)}(y;d)\;dy\Bigg)\\
&=w^{(q)}_{k-1}(x;d)-\delta_kw_{k-1}^{(q)}(b_k;d)W_k^{(q)}(x-b_k)+\delta_k\int_{b_k}^xW_k^{(q)\prime}(x-y)w_{k-1}^{(q)}(y;d)\;dy\\
&= w^{(q)}_{k-1}(x;d)+\delta_{k}\int^{x}_{b_{k}}W_k^{(q)}(x-y)w_{k-1}^{(q)\prime}(y;d)\md y\\
&=w^{(q)}_{k}(x;d).
\end{align*}
Therefore, we have that for any $k\geq 1$ the sequence $(w^{(q),(n)}_k)_{n\geq 1}$ converges pointwise to $w_k^{(q)}$.
\par If we denote by $V^{(q),(n)}_k$ the resolvent measure of the process $U^{(q),(n)}_k$ killed when exiting the interval $[0,a]$, 
then for each $n\geq 1$ and any open interval $\mathcal{B}\subset [0,\infty]$ we have
\begin{align*}
V^{(q),(n)}(x,\mathcal{B}):=\int_0^{\infty}e^{-qt}\mathbb{P}_x\left(U^{(n)}_k(t)\in \mathcal{B}, \overline{U}^{(n)}_k(t)\leq a,
 \underline{U}^{(n)}_k(t)\geq 0\right)dt,
\end{align*}
where $\overline{U}^{(n)}_k(t):=\sup_{0\leq s\leq t}U^{(n)}_k(s)$ and $\underline{U}^{(n)}_k(t):=\inf_{0\leq s\leq t}U^{(n)}_k(s)$.
\par As in the proof of Theorem 6 in \cite{kyprianouloeffen2010}, the uniform convergence on compact sets of the sequence 
$(U^{(n)}_k)_{n\geq 1}$ to $U_k$ implies that for each $t>0$, it holds a.s.~that 
\[
\lim_{n\uparrow \infty}((U^{(n)}_k(t),\overline{U}^{(n)}_k(t),\underline{U}^{(n)}_k(t))=((U_k(t),\overline{U}_k(t),\underline{U}_k(t)).
\]
Using the ideas in the proof of Theorem 6 in \cite{kyprianouloeffen2010}, it is not difficult to see that  $\mathbb{P}_x(U_k(t)\in\partial \mathcal{B})=0$.
\par Hence by using the convergence of the sequence $(w^{(q),(n)}_k)_{n\geq1}$ we obtain that
\begin{align*}
V^{(q)}(x,\mathcal{B}):&=\int_0^{\infty}e^{-qt}\mathbb{P}_x\left(U_k(t)\in \mathcal{B}, \overline{U}_k(t)\leq a, \underline{U}_k(t)\geq 0\right)dt\\
&=\lim_{n\to\infty}\int_0^{\infty}e^{-qt}\mathbb{P}_x\left(U^{(n)}_k(t)\in \mathcal{B}, \overline{U}^{(n)}_k(t)\leq a, \underline{U}^{(n)}_k(t)\geq 0\right)dt\\
&=\lim_{n\to\infty}\int_{\mathcal{B}}\left(\sum_{i=0}^{k} \frac{\frac{w_k^{(q),(n)}(x)}{w_k^{(q),(n)}(a)}w_k^{(q),(n)}(a;y)-w_k^{(q),(n)}(x;y)}
{\prod_{j=1}^i\left(1-\delta_j W_{j-1}^{(q),(n)}(0)\right)}\ind_{\{y\in [b_i,b_{i+1}]\}} \right)\;dy\\
&=\int_{\mathcal{B}}\left(\sum_{i=0}^{k} \frac{\frac{w_k^{(q)}(x)}{w_k^{(q)}(a)}w_k^{(q)}(a;y)-w_k^{(q)}(x;y)}
{\prod_{j=1}^i\left(1-\delta_j W_{j-1}^{(q)}(0)\right)}\ind_{\{y\in [b_i,b_{i+1}]\}} \right)\;dy \\
&=\int_{\mathcal{B}} \Xi_{\phi_k}^{-1}(y)\left(\frac{w_k^{(q)}(x)}{w_k^{(q)}(a)}w_k^{(q)}(a;y)-w_k^{(q)}(x;y)\right)\;dy,
\end{align*}
which proves the identity (\ref{resolvent1}) for the unbounded variation case. \begin{flushright}$\blacksquare$\end{flushright}

\section{Proofs for the general case}\label{Generalcase}
\subsection{Proof of Theorem \ref{Resolvents}(ii)-(iv)} 
We start by providing the proof of identity \eqref{resolvent2}.
To compute the resolvent of the first passage time below $0$, we take $d=0$ and $a \to \infty$ in \eqref{resolvent1}. 
To this end, we note that from Lemma \ref{lemma_convergence_assumptotics_w} we get the following:
\begin{align*}
\lim_{a\to\infty}\frac{w_k^{(q)}(a;y)}{W_{k+1}^{(q)}(a)}=\lim_{a\to\infty}\left(\frac{w_{k-1}^{(q)}(a;y)}{W_k^{(q)}(a)}\frac{W_k^{(q)}(a)}{W_{k+1}^{(q)}(a)}+\delta_k\int_{b_k}^a
\frac{W_k^{(q)}(a-z)}{W_{k+1}^{(q)}(a)}w_{k-1}^{(q)\prime}(z;y)dz\right)=0.
\end{align*}
Therefore,
\begin{align}\label{function_v}\nonumber
v_k^{(q)}(y):&=\lim_{a\to\infty}\frac{w_k^{(q)}(a;y)}{W_{k}^{(q)}(a)}=\lim_{a\to\infty}\left(\frac{w_{k-1}^{(q)}(a;y)}{W_k^{(q)}(a)}+
\delta_k\int_{b_k}^a\frac{W_k^{(q)}(a-z)}{W_{k}^{(q)}(a)}w_{k-1}^{(q)\prime}(z;y)dz\right)\\&=\delta_k\int_{b_k}^{\infty} 
e^{-\varphi_k(q)z}w_{k-1}^{(q)\prime}(z;y)dz,
\end{align} 
where the last equality follows from Lebesgue's dominated convergence theorem by noting that 
$W_k^{(q)}(a-z) / W_{k}^{(q)}(a) \leq e^{-\varphi_k(q) z} $ and we have from Lemma \ref{lemma_convergence_assumptotics_w} 
that $\int^{\infty}_{b_{k}} e^{-\varphi_k(q)y}$ $w_{k-1}^{(q)\prime}(y;d)\md y < \infty$.
%\begin{flushright}$\blacksquare$\end{flushright}
%\red{$\blacksquare$ should be moved later?}

	Next, we prove identities \eqref{resolvent3} and \eqref{resolvent4}.
	To prove (\ref{resolvent3}), we take the limit $d \rightarrow -\infty $ in (\ref{resolvent1}) and the result follows since
	\begin{align*}
&\lim_{d\to-\infty}\frac{w_{1}^{(q)}(x;d)}{W^{(q)}(-d)} =\lim_{d\to-\infty}\frac{
W^{(q)}(x-d)+\delta_1\int_{b_1}^x W_1^{(q)}(x-z)W^{(q)\prime}(z-d)\;dz }{W^{(q)}(-d)}\\
& =e^{\Phi(q)x}+ \delta_1 \Phi(q)\int_{b_1}^x e^{\Phi(q)z} W_1^{(q)}(x-z)dz=u^{(q)}_1(x). 
\end{align*}
%\blue{[Here and in the rest of the proof change $\Phi(q)$ to $\varphi_0(q)$?]}
The above limit holds because  $\lim_{d\to -\infty} \frac{W^{(q)}(y-d)}{W^{(q)}(-d)}=e^{\Phi(q)y}$ and
$\lim_{d\to-\infty}  \frac{W^{(q)\prime}(y-d)}{W^{(q)}(-d)}=\Phi(q)e^{\Phi(q)y} $. Moreover,
\begin{align*}
&\lim_{d\to-\infty}\frac{ w_{1}^{(q)\prime}(x;d)}{W^{(q)}(-d)} \\&=\lim_{d\to-\infty}\frac{
W^{(q)\prime}(x-d)\left(1+\delta_1 W_1^{(q)}(0)\right)+\delta_1\int_{b_1}^x W_1^{(q)\prime}(x-z)W^{(q)\prime}(z-d)dz}{W^{(q)}(-d)}\\
& =\Phi(q)e^{\Phi(q)x}\left(1+\delta_1 W_1^{(q)}(0)\right)+ \delta_1 \Phi(q)\int_{b_1}^x e^{\Phi(q)z} W_1^{(q)\prime}(x-z)dz=u^{(q)\prime}_1(x). 
\end{align*}

Thus, proceeding by induction we assume that 
\begin{equation}\label{limit-d}
\lim_{d\to-\infty} \frac{ w_{k-1}^{(q)}(x;d)}{W^{(q)}(-d)} =u_{k-1}^{(q)}(x)
\qquad \textrm{and} \qquad 
\lim_{d\to-\infty} \frac{ w_{k-1}^{(q)\prime}(x;d)}{W^{(q)}(-d)} =u_{k-1}^{(q)\prime}(x).
\end{equation} 

Therefore
\begin{align*}\nonumber
&\lim_{d\to -\infty}\frac{ w_{k}^{(q)}(x;d)}{W^{(q)}(-d)}=\lim_{d\to -\infty}\frac{
w_{k-1}^{(q)}(x;d) + \delta_k\int_{b_k}^x W_k^{(q)}(x-z) w_{k-1}^{(q)\prime}(z;d)dz}{W^{(q)}(-d)}\\&
= u_{k-1}^{(q)}(x) + \delta_k\int_{b_k}^x W_k^{(q)}(x-z) u_{k-1}^{(q)\prime}(z)dz=u_k^{(q)}(x).
\end{align*}

We obtain the equality in (\ref{resolvent4}) by taking the limit $a \rightarrow \infty $ in (\ref{resolvent3}) and applying 
Lemma \ref{lemma_convergence_assumptotics_w}. Moreover, using the arguments given in Lemma \ref{lemma_convergence_assumptotics_w} 
we obtain that
$ \lim_{a\to\infty}\frac{u_k^{(q)}(a)}{W_{k+1}^{(q)}(a)}=0$,
and finally 
$$
\lim_{a\to\infty} \frac{ u_{k}^{(q)}(a)}{W_{k-1}^{(q)}(a)} = 
\delta_k \int_{b_k}^{\infty} e^{-\varphi_k(q)z} u_{k-1}^{(q)\prime}(z)\;dz. \textrm{\hspace{7cm}} \blacksquare $$

\subsection{Proof of Theorem \ref{th:mr.recursion}(i)}
To obtain the identity (\ref{main_twosided_up}) for a general spectrally negative L\'evy process $X$, it suffices to note that for $q >0$, by applying the 
strong Markov property, we have that
\begin{eqnarray*}
	\lefteqn{\e_x \left[ \mathrm{e}^{-q \kappa_k^{a,+}} \ind_{\{\kappa_k^{a,+}<\kappa_k^{0,-}\}} \right] \cdot q\int_0^{\infty}e^{-qt}
		\mathbb{P}_a(U_k(t)\in \mathbb{R}_+,t<\kappa_k^{0,-})\;dt}\\&=&
	q\int_0^{\infty}e^{-qt}\mathbb{P}_x(U_k(t)\in \mathbb{R}_+,t<\kappa_k^{0,-})\;dt\\
	&-& q\int_0^{\infty}e^{-qt}\mathbb{P}_x(U_k(t)\in [0,a],t<\kappa_k^{0,-}\wedge\kappa_k^{a,+})\;dt.
\end{eqnarray*}
The above probabilities can be obtained directly from the resolvent  measures given in Theorem \ref{Resolvents}.
\begin{flushright}$\blacksquare$\end{flushright}

\subsection{Proof of Theorem \ref{th:mr.recursion}(ii)}
The proof follows by noting that, for $x \leq a$,
\begin{eqnarray*}
\lefteqn{\e_x \left[ \mathrm{e}^{-q \kappa_k^{0,-}} \ind_{\{\kappa_k^{0,-}<\kappa_k^{a,+}\}} \right]}\\&=&
\e_x \left[ \mathrm{e}^{-q \kappa_k^{0,-}} \ind_{\{\kappa_k^{0,-}<\infty\}} \right]-
\e_x \left[ \mathrm{e}^{-q \kappa_k^{a,+}} \ind_{\{\kappa_k^{a,+}<\kappa_k^{0,-}\}} \right]\e_a \left[ \mathrm{e}^{-q \kappa_k^{0,-}} 
\ind_{\{\kappa_k^{0,-}<\infty\}} \right].
\end{eqnarray*}

Now, we compute the first term of the right-hand side of the above equality as
\begin{align*}
\e_x \left[ \mathrm{e}^{-q \kappa_k^{0,-}} \ind_{\{\kappa_k^{0,-}<\infty\}} \right]=
1-q\int_0^{\infty}e^{-qt}\mathbb{P}_x(U_k(t)\in \mathbb{R}_+,t<\kappa_k^{0,-})\;dt.
\end{align*}
Using \eqref{resolvent2} we note that
\begin{align*}
&\int_0^{\infty}e^{-qt}\mathbb{P}_x(U_k(t)\in \mathbb{R}_+,t<\kappa_k^{0,-})\;dt
=\int_0^{\infty}
\frac{\frac{w_k^{(q)}(x)}{v_k^{(q)}(0)}v_k^{(q)}(y)-w_k^{(q)}(x;y)}
{\Xi_{\phi_k}(y)}  \;dy\\
&=\int_0^{\infty}\sum_{i=0}^{k}
\frac{\frac{w_k^{(q)}(x)}{v_k^{(q)}(0)}v_k^{(q)}(y)-w_k^{(q)}(x;y)}
{\Xi_{\phi_i}(y)} \ind_{\{y\in[b_i,b_{i+1}]\}} \;dy.
\end{align*}

%\red{[Is defined $b_0 = 0$?]} 
%\red{[I am a bit confused below.  It seems below that we are setting $b_2 = \infty$?  I think it is ok to set $b_{k+1} = \infty$. But in this argument $b_2$ is already defined to be finite? We could just say here that this is for the case $k=1$?]}

%\red{[below: $\Xi_{\phi_1}(y)$ has to be in the integral.  also sometimes $(y)$ is missing.]}

Using equalities \eqref{small w}, (\ref{w_k_above_b_k}), and (\ref{useful_inedity2}) we obtain that
\begin{align*}
&\sum_{i=0}^1
\int_{b_i}^{b_{i+1}}\frac{w_1^{(q)}(x;y)}{\Xi_{\phi_i}(y)}\;dy
=\int_0^{b_1}w_1^{(q)}(x;y)\;dy+\int_{b_1}^{x}\frac{w_1^{(q)}(x;y)}{\Xi_{\phi_1}(y)}\;dy\\
&=\int_0^{x}W^{(q)}(x-y)\;dy+ \delta_1 \int_0^{b_1} \int_{b_1}^{x} W_1^{(q)}(x-z)W^{(q)\prime}(z-y)dz \;dy \\ &
+\int_{b_1}^{x}W_1^{(q)}(x-y)dy - \int_{b_1}^{x}W^{(q)}(x-y)\;dy\\
&=\int_0^{x}W^{(q)}(y)\;dy +\delta_1\int_{b_1}^{x}W_1^{(q)}(x-y)\left(W^{(q)}(y)-W^{(q)}(y-b_1)\right)\;dy\\&+
\delta_1\int_{0}^{x-b_1}W_1^{(q)}\left(x-(b_1+y)\right)W^{(q)}(y)\;dy\\
&=\int_0^x W^{(q)}(y)\;dy+\delta_1\int_{b_1}^{x}W_1^{(q)}(x-y)W^{(q)}(y)\;dy.
\end{align*}
Hence,
\begin{align*}
z_1^{(q)}(x)&:=1+q\sum_{i=0}^1 \int_{b_i}^{b_{i+1}}\frac{w_1^{(q)}(x;y)}{\Xi_{\phi_i}(y)}\;dy
=Z^{(q)}(x)+q\delta_1\int_{b_1}^{x}W_1^{(q)}(x-y)W^{(q)}(y)\;dy\\
&=Z^{(q)}(x)+\delta_1\int_{b_1}^{x}W_1^{(q)}(x-y)Z^{(q)\prime}(y)\;dy.
\end{align*}
Now, assume by induction that
$$ z^{(q)}_{k-1}(x):=1+q\sum_{i=0}^{k-1}\int_{b_i}^{b_{i+1}}\frac{w_k^{(q)}(x;y)}{\Xi_{\phi_i}(y)}\;dy, \quad \textrm{where}
 \quad b_0=-\infty, b_{k}=x.  $$

Then, by using identities \eqref{eq:mr.recursion}, (\ref{w_k_above_b_k}), and (\ref{useful_inedity2}) we have that
\begin{align*}
z^{(q)}_k(x):&= 1+q\sum_{i=0}^k \int_{b_i}^{b_{i+1}}\frac{w_k^{(q)}(x;y)}{\Xi_{\phi_i}(y)}\;dy
\\&=1+q\sum_{i=0}^{k-1} \int_{b_i}^{b_{i+1}}\frac{w_k^{(q)}(x;y)}{\Xi_{\phi_i}(y)}\;dy+q\int_{b_k}^{x}W_k^{(q)}(x-y)\;dy\\
&=1+\sum_{i=0}^{k-1} q\int_{b_i}^{b_{i+1}}\frac{w_{k-1}^{(q)}(x;y)}{\Xi_{\phi_i}(y)}\;dy + q 
\int_{b_k}^{x}\frac{w_{k-1}^{(q)}(x;y)}{\Xi_{\phi_{k-1}}(y)}\;dy +q\int_{b_k}^{x}W_k^{(q)}(x-y)\;dy\\
&- q\int_{b_k}^{x}W_{k-1}^{(q)}(x-y) \;dy 
+ q \delta_k\sum_{i=0}^{k-1} \int_{b_i}^{b_{i+1}} \frac{1}{\Xi_{\phi_i}(y)}
\int_{b_k}^{x} W_k^{(q)}(x-z)w_{k-1}^{(q)\prime}(z;y)\;dz\;dy\\
&=1+q\sum_{i=0}^{k-1} \int_{b_i}^{b_{i+1}}\frac{1}{\Xi_{\phi_i}(y)}\left(w^{(q)}_{k-1}(x;y)+
\delta_k\int_{b_k}^xW_k^{(q)}(x-z)w_{k-1}^{(q)\prime}(z;y)\;dz\right)\;dy\\
&=z_{k-1}^{(q)}(x)+q\delta_k \int_{b_k}^x \sum_{i=0}^{k-1}\int_{b_i}^{b_{i+1}} \frac{1}{\Xi_{\phi_i}(y)} W_k^{(q)}(x-z)w_{k-1}^{(q)\prime}(z;y)\;dy\; dz\\
&+q\delta_k\int_{b_k}^{x}W_k^{(q)}(x-z)W_{k-1}^{(q)}(z-b_k)\;dz\\
&=z_{k-1}^{(q)}(x)+q\delta_k \int_{b_k}^x  W_k^{(q)}(x-z) \sum_{i=0}^{k-1} \int_{b_i}^{b_{i+1}}\frac{1}{\Xi_{\phi_i}(y)} w_{k-1}^{(q)\prime}(z;y)\;dy\; dz\\
&+q\delta_k\int_{b_k}^{x}W_k^{(q)}(x-z)\frac{d}{dz}\left(\int_{b_k}^{z} W_{k-1}^{(q)}(z-y)dy\right)\;dz\\
&=z_{k-1}^{(q)}(x)+\delta_k\int_{b_k}^xW_k^{(q)}(x-z)z_{k-1}^{(q)\prime}(y)\;dy. 
\end{align*}

Hence, we can write
\begin{align*}
\e_x \left[ \mathrm{e}^{-q \kappa_k^{0,-}} \ind_{\{\kappa_k^{0,-}<\infty\}} \right]=
1-q\sum_{i=0}^k  \int_{b_i}^{b_{i+1}} \frac{1}{\Xi_{\phi_i}(y)} \left\{\frac{w_k^{(q)}(x)}{v_k^{(q)}(0)}v_k^{(q)}(y)-w_k^{(q)}(x;y)\right\}dy.
\end{align*}
%\red{[delete $dy$ in the middle?]}

\par Finally, by putting the pieces together we obtain that
\begin{align*}
\e_x &\left[ \mathrm{e}^{-q \kappa_k^{0,-}} \ind_{\{\kappa_k^{0,-}<\kappa_k^{a,+}\}} \right]=1-q\sum_{i=0}^k
 \int_{b_i}^{b_{i+1}} \frac{1}{\Xi_{\phi_i}(y)}\left\{\frac{w_k^{(q)}(x)}{v_k^{(q)}(0)}v_k^{(q)}(y)-w_k^{(q)}(x;y)\right\}dy\\
&-\frac{w_k^{(q)}(x)}{w_k^{(q)}(a)}\left(1-q\sum_{i=0}^k\int_{b_i}^{b_{i+1}}\frac{1}{\Xi_{\phi_i}(y)}\left\{
\frac{w_k^{(q)}(a)}{v_k^{(q)}(0)}v_k^{(q)}(y)-w_k^{(q)}(a;y)\right\}dy\right)\\
&=z_k^{(q)}(x)-\frac{w_k^{(q)}(x)}{w_k^{(q)}(a)}z_k^{(q)}(a). \textrm{\hspace{10cm}} \blacksquare
\end{align*}
%\red{[delete $dy$ in the middle (two places)?]}

\subsection{Proof of  Corollary \ref{one_sided}(i)-(ii)}
	To obtain (\ref{main_onesided_down}), we take the limit $a\rightarrow \infty$ in (\ref{main_twosided_down}). The result follows because from (\ref{function_v}) we have
	\begin{align*}
v_k^{(q)}(0)&=\lim_{a\to\infty}\frac{w_k^{(q)}(a)}{W_{k}^{(q)}(a)}=\delta_k\int_{b_k}^{\infty} e^{-\varphi_k(q)z}w_{k-1}^{(q)\prime}(z)dz.
\end{align*}
Then, using Lemma \ref{lemma_convergence_assumptotics_w} and Remark \ref{remark_convergence_assumptotics_w} we get 
\begin{align*}
\lim_{a\to\infty}\frac{z_k^{(q)}(a)}{W_{k+1}^{(q)}(a)}=\lim_{a\to\infty}\left(\frac{z_{k-1}^{(q)}(a)}{W_k^{(q)}(a)}\frac{W_k^{(q)}(a)}{W_{k+1}^{(q)}(a)}
+\delta_k\int_{b_k}^a\frac{W_k^{(q)}(a-z)}{W_{k+1}^{(q)}(a)}z_{k-1}^{(q)\prime}(z)\;dz\right)=0.
\end{align*}
Therefore,
\begin{align}\nonumber
\lim_{a\to\infty}\frac{z_k^{(q)}(a)}{W_{k}^{(q)}(a)}=\lim_{a\to\infty}\left(\frac{z_{k-1}^{(q)}(a)}{W_k^{(q)}(a)}+
\delta_k\int_{b_k}^a\frac{W_k^{(q)}(a-z)}{W_{k}^{(q)}(a)}z_{k-1}^{(q)\prime}(z)\;dz\right)=
\delta_k\int_{b_k}^{\infty} e^{-\varphi_k(q)z}z_{k-1}^{(q)\prime}(z)\;dz.
\end{align}
%To obtain (\ref{main_onesided_up}), we take the limit $d\rightarrow -\infty$ in (\ref{main_twosided_up}) and use
%equation \eqref{limit-d}. 
The identity (\ref{main_onesided_up}) is immediate by taking the limit $d\rightarrow -\infty$ in (\ref{main_twosided_up}) and using
equation \eqref{limit-d}.
\begin{flushright}$\blacksquare$\end{flushright}
%\red{[Above, we can use Lemma \ref{lemma_convergence_assumptotics_w}.]} 
%\blue{I think in this case we cannot use Lemma  \ref{lemma_convergence_assumptotics_w}, because here
%we take the limit in $w_{k}^{(q)}(x;d)$ with respect to the second coordinate...}
\subsection{Proof of  Corollary \ref{ruin_probab}}

To obtain the ruin probability, we first take $q=0$ in (\ref{main_twosided_down}) giving
that 
\begin{equation}\label{eq:ruin.k}
\mathbb{P}_x(\kappa_k^{0,-}<\kappa_k^{a,+})=z_k(x)-\frac{w_k(x)}{w_k(a)}.
\end{equation} 
From the definition of $z_k$ given in Theorem \ref{th:mr.recursion}(ii), and because $Z=1$, it follows immediately by induction that $z_k=1$. 
%Next  we take the limit $a \rightarrow \infty$ and it follows that 
	%$$\lim_{a\to \infty} w_k(a)= \frac{1-\sum_{j=1}^{k} \delta_j w_{j-1}(b_j)}{\mathbb{E}[X_1]-\sum_{j=1}^k\delta_j} ,$$
	%where $w_0=W$.
%The proof of the above equality is inductive, at first for $k=1$ we obtain that (since $W(\infty)=1/\mathbb{E}[X_1]>0$)
%\begin{align*}
%w_1(\infty)&=\lim_{a\to \infty}\left(W(a)+ \delta_1 \int_{b_1}^{a} W_1(a-y) W'(y) dy \right)\\ &= \frac{1}{\mathbb{E}[X_1]}+\frac{\delta_1}{\mathbb{E}[X_1]-\delta_1}  \int_{b_1}^{\infty}W'(y) dy = \frac{1}{\mathbb{E}[X_1]}+\frac{\delta_1\left(\frac{1}{\mathbb{E}[X_1]}-W(b_1)\right)}{\mathbb{E}[X_1]-\delta_1} =\frac{1-\delta_1 W(b_1)}{\mathbb{E}[X_1]-\delta_1}.
%\end{align*}
%And we obtain 
%\begin{align*}
%\Psi_1(x)=1-\left[\frac{\mathbb{E}[X_1]-\delta_1}{1-\delta_{1}W(b_{1})}\right]w_1(x),
%\end{align*}
%which agrees with the result in Theorem 5 (ii) in \cite{kyprianouloeffen2010}.
Next, we will prove by induction that
	\begin{equation}\label{lim_ruin_prob}\lim_{a\to \infty} w_k(a)= \frac{1-\sum_{j=1}^{k} \delta_j w_{j-1}(b_j)}{\mathbb{E}[X(1)]-\sum_{j=1}^k\delta_j} ,\end{equation}
	where $w_0=W$.
	Now, for $k=1$, using the fact that $W(\infty)=1/\mathbb{E}[X(1)]$ (which is well defined because $\mathbb{E}[X(1)] > 0$ by assumption), 
	we obtain the following:
	\begin{align*}
	w_1(\infty)&=\lim_{a\to \infty}\left(W(a)+ \delta_1 \int_{b_1}^{a} W_1(a-y) W'(y) dy \right)\\ &= \frac{1}{\mathbb{E}[X(1)]}+\frac{\delta_1}{\mathbb{E}[X(1)]-\delta_1}  \int_{b_1}^{\infty}W'(y) dy \\ &= \frac{1}{\mathbb{E}[X(1)]}+\frac{\delta_1\left(\frac{1}{\mathbb{E}[X(1)]}-W(b_1)\right)}{\mathbb{E}[X(1)]-\delta_1} =\frac{1-\delta_1 W(b_1)}{\mathbb{E}[X(1)]-\delta_1}.
	\end{align*}
	Hence,
	\begin{align*}
	\mathbb{P}_x(\kappa_1^{0,-}<\kappa_1^{a,+})=1-\left[\frac{\mathbb{E}[X(1)]-\delta_1}{1-\delta_{1}W(b_{1})}\right]w_1(x),
	\end{align*}
	which agrees with the result in Theorem 5 (ii) in \cite{kyprianouloeffen2010}.

Now, for the induction step assume that 
\begin{equation}\label{w_k-1 infty}
\lim_{a\to \infty} w_{k-1}(a)= \frac{1-\sum_{j=1}^{k-1} \delta_j w_{j-1}(b_j)}{\mathbb{E}[X(1)]-\sum_{j=1}^{k-1}\delta_j} .\end{equation}
Then, we compute using (\ref{w_k-1 infty})
\begin{align*}
w_k(\infty)&=\lim_{a\to \infty}\left(w_{k-1}(a)+ \delta_k \int_{b_k}^{a} W_k(a-y) w_{k-1}'(y) dy \right)\\ &= w_{k-1}(\infty)+\frac{\delta_k \left( w_{k-1}(\infty)-  w_{k-1}(b_k) \right) }{\mathbb{E}[X(1)]-\sum_{j=1}^{k}\delta_j}
= \frac{1-\sum_{j=1}^{k} \delta_j w_{j-1}(b_j)}{\mathbb{E}[X(1)]-\sum_{j=1}^k\delta_j}.
\end{align*}
Finally, we obtain \eqref{rp2}. Then, applying \eqref{lim_ruin_prob} in \eqref{eq:ruin.k} gives the result.
\begin{flushright}$\blacksquare$\end{flushright}

\section{Proofs for  Section \ref{ss:integral}}\label{AppD}

\subsection{Proof of Lemma \ref{lem:zeta.finite}}
     To show that $\zeta(x)$ is well defined, we consider its Laplace transform.
     Thus, we have that
     $$\int_0^\infty e^{-sx}\sum_{l=1}^\infty a^l 
{(W^{(q)\prime})}^{*l}(x)\,dx=\sum_{l=1}^\infty \left(a\int_0^\infty e^{-sx}{W^{(q)\prime}}(x)\,dx\right)^l.$$
In order to proof the convergence, we have to verify that for sufficiently large $s>0$, it holds that $a\int_0^\infty e^{-sx}{W^{(q)\prime}}(x)\,dx<1$. 
Indeed,
\begin{comment}
For $q=0$ 
$$  \int_0^\infty e^{-sx}W^{\prime}(x)\,dx \le   \int_0^\infty W^{\prime}(x)\,dx=W(\infty)-W(0)<\infty,$$
for all $s>0$, so by Lemma \ref{lem:prep1} the statement is true.
  For $q>0$, we use {the representation $W^{(q)}(x)=e^{\Phi(q)x}W_{\Phi(q)}(x)$}
from which 
${W^{(q)\prime}}(x)=\Phi(q)e^{\Phi(q)x}W_{\Phi(q)}(x)+e^{\Phi(q)x}W^{\prime}_{\Phi(q)}(x)$
(the derivative exists besides a denumerable number of points) \blue{remove this expression?}.
Then \blue{for $s>\Phi(q)$}
\begin{align*}
\int_0^\infty e^{-sx}\blue{W^{(q)\prime}}(x)\,dx&=\Phi(q)\int_0^\infty e^{-(s-\Phi(q))x}W_{\Phi(q)}(x)\,dx+\int_0^\infty  e^{-(s-\Phi(q))x}W^{\prime}_{\Phi(q)}(x)\,dx
\end{align*}
is finite for big enough $s$.
\blue{I guess it is easier using an integration by parts?
\end{comment}
 for $s>\Phi(q)$
\begin{align}\label{lap_der}
\int_0^\infty e^{-sx}{W^{(q)\prime}}(x)\,dx&=-W^{(q)}(0)+s\int_0^\infty e^{-sx}{W^{(q)}}(x)\,dx\notag\\
&=-W^{(q)}(0)+\frac{s}{\psi(s)-q}.
\end{align}
Now using Proposition I.2 (ii) and Corollary VII.5 (iii) in \cite{bertoin1996}, we have that 
\begin{equation}\label{limit_LT}
\lim_{s\to\infty}\left(\frac{s}{\psi(s)-q}-W^{(q)}(0)\right)=0.
\end{equation}
Hence, the Laplace transform given in \eqref{lap_der} is less than $1/a$ for sufficiently large $s$.
The proof is completed by applying Lemma \ref{lem:prep1}.
\begin{flushright}$\blacksquare$\end{flushright}

\subsection{Proof of Lemma \ref{lem:i-iii}}
  
 We start by noting that
$K(x,y)\le a_T {W^{(q)\prime}}((x-y)+)$.
% Suppose that for some $a,b\ge 0$
 %\begin{equation}\label{eq:assumption.sublin}\sup_{y\le w\le x}|\phi(w)|\le a+b%(x-y),\qquad 0\le y<x.
% \end{equation}
    Now,  assume that 
     \begin{equation}\label{eq:iii}  K^{(l)}(x,y)\le a_T^l{(W^{(q)\prime})^{*l}((x-y)+)}
     \end{equation}
       for $l\in\mathbb{N}$.
     Then, for $x> y\geq0$, we have that
     \begin{align*}
         K^{(l+1)}(x,y)&=\int_{y}^xK^{(l)}(x,z)K(z,y)\,dz\\
       &\le {a_T^{l+1}}\int_y^x(W^{(q)\prime})^{*l}(x-z){W^{(q)\prime}}(z-y)\,dz\\
       &= {
  a_T^{l+1}\int_0^{x-y}(W^{(q)\prime})^{*l}(x-y-z){W^{(q)\prime}}(z)\,dz}\\
&={a_T^{l+1}(W^{(q)\prime})^{*(l+1)}((x-y)+)}.
     \end{align*}
Hence, by induction we can conclude that \eqref{eq:iii} holds for every $l\in\mathbb{N}$. Therefore,
     $$K^{\dagger}(x,y)\le \sum_{l=1}^\infty a_T^lW^{(q)'*l}((x-y)+)=\zeta(x-y).$$

%we can do it because the function $e^{-sx}{W^{(q)\prime}}(x)$
 %    is positive.
% and in view of Proposition \ref{prop:wq} also integrable). 
\begin{flushright}$\blacksquare$\end{flushright}
\subsection{Proof of Proposition \ref{prop:solutions}}
(i)
Considering the discussion in Section 3.2, we only need to show that$$\int_d^xK^{\dagger}(x,y)\Xi_{\phi}(x)^{-1}W^{(q)\prime}(y-d)\,dy<\infty.$$ 
Because $\phi(x)<1/W^{(q)}(0)$ for all $x\geq0$ in the bounded variation case and $W^{(q)}(0)=0$ in the unbounded variation case, 
we have that $b:=\Xi_{\phi}(x)^{-1}<\infty$. Therefore,
\begin{align*}
\int_d^x \zeta(x-y)g(y;d)\,dy\le b \int_d^x\zeta(x-y){W^{(q)\prime}}(y-d)\,dy&=b\int_0^{x-d}\zeta(x-y-d){W^{(q)\prime}}(y)\,dy\\
&=\frac{b}{a_T}\sum_{l=2}^\infty 
a_T^l(W^{(q)\prime})^{*l}((x-d)+),
\end{align*}
%\green{[$a_T$ instead of $a$ below and above?]}\blue{[Agreed.]}
which is convergent by Lemma \ref{lem:i-iii} for any $x\in(d,T]$.\\
\begin{comment}
In the bounded variation case, since $\phi(0+)=0$ {$\phi(x)<1/W^{(q)}(0)$ for all $x>0$?} we have that
$\sup_{y\in [0,T]}\frac{1}{(1-\phi(M)W^{(q)}(0))}=b<\infty$  and therefore 
 $\int_0^x \zeta(x-y)g({y})\,dy\le \int_0^x\zeta(x-y)b{W^{(q)\prime}}(y)\,dy={\frac{b}{a}}\sum_{l=2}^\infty 
a^l(W^{(q)\prime})^{*l}(x)$, where $a$ is defined in \eqref{eq:def_a}, {which is convergent by Lemma
\ref{lem:i-iii}.}
The case of unbounded variation can be treated similarly, since $W^{(q)}(0)=0$.\\
\end{comment}
(ii) In this case, we need only to check that $$\int_0^xK^{\dagger}(x,y)\Xi_{\phi}(x)^{-1}qW^{(q)}(y)\,dy<\infty.$$ 
Then,
\begin{align*}
\int_0^x \zeta(x-y)g(y)\,dy&=\int_0^x g(x-y)\zeta(y)\,dy\\
&=\sum_{l=1}^\infty a_T^l\int_0^x q W^{(q)}(x-y) {(W^{(q)\prime})^{*l}}(y)\,dy 
\\
&\le qW^{(q)}(x)\sum_{l=1}^\infty a_T^l\int_0^x  {(W^{(q)\prime})^{*l}}(y)\,dy=  qW^{(q)}(x)\int_0^x \zeta(y)\,dy.
\end{align*}
Furthermore, we have that
\begin{align*}
\int_0^{\infty}e^{-s x}\int_0^x{(W^{(q)\prime})^{*l}}(y)\,dydx&=\frac{1}{s}\int_0^{\infty}e^{-sy}{(W^{(q)\prime})^{*l}}(y)dy\\
&=\frac{1}{s}\left(\int_0^\infty e^{-sx}{W^{(q)\prime}}(x)\,dx\right)^{l}\\
&=\frac{1}{s}\left(\frac{s}{\psi(s)-q}-W^{(q)}(0)\right)^{l}.
\end{align*}
Now from \eqref{limit_LT} we get that 
$$\int_0^\infty e^{-sx}\int_0^x\zeta(y)\,dy\,dx<\infty$$
for sufficiently large $s>0$, which yields the finiteness of $\int_0^x\zeta(y)\,dy$ for any $x\in[0,T]$.
\begin{flushright}$\blacksquare$\end{flushright}

\section*{Acknowledgement}
We want to thank the anonymous referees for the careful reading, constructive comments and suggestions, which
significantly improved the presentation and the readability of the paper.

\bibliographystyle{alpha}
\bibliography{REFERENCES}

\end{document}